\documentclass[12pt]{article}
\usepackage[margin=1in]{geometry}
\usepackage[utf8]{inputenc} 
\usepackage[T1]{fontenc}    
\RequirePackage[colorlinks,citecolor=blue,urlcolor=blue]{hyperref}
\usepackage{hyperref}       
\usepackage{amsfonts}       
\usepackage{nicefrac}       
\usepackage{microtype}      
\usepackage[numbers,sort&compress]{natbib}
\usepackage{times}
\usepackage{bbm}
\usepackage{amsmath}
\usepackage{amsthm}
\usepackage{graphicx}
\usepackage{subcaption}
\usepackage[capitalise,noabbrev]{cleveref}
\usepackage{makecell}
\usepackage{booktabs}
\usepackage{multirow}
\usepackage{accents}

\let\hat\widehat

\newtheorem{theorem}{Theorem}

\newtheorem{lemma}{Lemma}

\newtheorem{corollary}{Corollary}

\newtheorem{proposition}{Proposition}
\newtheorem{remark}{Remark}
\newtheorem{fact}{Fact}

\newtheorem{definition}{Definition}

\DeclareMathOperator*{\argmin}{arg\,min}

\newcommand{\nat}{\theta}
\newcommand{\Nat}{\Theta}
\newcommand{\mep}{\mu}

\newcommand{\Mep}{M}
\newcommand{\LR}{\mathrm{L}}
\newcommand{\GLR}{\mathrm{GL}}
\newcommand{\KL}{\mathrm{KL}}
\newcommand{\D}{D_{\psi^*_{\mep_0}}}

\newcommand{\high}{\mathrm{high}}
\newcommand{\C}{\mathrm{C}}
\newcommand{\CI}{\mathrm{CI}}

\newcommand{\ML}{\mathrm{ML}}
\newcommand{\DM}{\mathrm{DM}}

\newcommand{\nmin}{n_{\min}}
\newcommand{\nminbar}{\bar{n}_{\min}}
\newcommand{\nmax}{n_{\max}}

\newcommand{\n}{n}
\newcommand{\ii}{i}

\newcommand{\st}{N}

\title{Nonparametric iterated-logarithm extensions of the sequential generalized likelihood ratio test}

\author{%
	Jaehyeok Shin, Aaditya Ramdas and Alessandro Rinaldo \\
	Carnegie Mellon University \\
	\texttt{\{shinjaehyeok, aramdas, arinaldo\}@cmu.edu}
}

\begin{document}
	\maketitle
	\begin{abstract}
We develop a nonparametric extension of the sequential generalized likelihood ratio (GLR) test and corresponding time-uniform confidence sequences for the mean of a univariate distribution. By utilizing a  geometric interpretation of the GLR statistic, we derive a simple analytic upper bound on the probability that it exceeds any prespecified boundary; these are intractable to approximate via simulations due to infinite horizon of the tests and the composite nonparametric nulls under consideration. Using time-uniform boundary-crossing inequalities, we carry out a unified nonasymptotic analysis of expected sample sizes of one-sided and open-ended tests over nonparametric classes of distributions (including sub-Gaussian, sub-exponential, sub-gamma, and exponential families). Finally, we present a flexible and practical method to construct time-uniform confidence sequences that are easily tunable to be uniformly close to the pointwise Chernoff bound over any target time interval.  
\end{abstract}

\section{Introduction} \label{sec::intro}

Consider the following setup for open-ended sequential testing: we observe i.i.d.\ data sequentially from an infinite stream $X_1,X_2,\dots$ generated by an unknown distribution $P$ over the real line with finite first moment and belonging to large nonparametric class of distributions. We wish to test a one-sided hypothesis about its mean $\mep := \int x \mathrm{d}P(x)$ by deciding, at each time point, whether to reject the null hypothesis or instead to continue sampling, possibly indefinitely. 
With a slight abuse of notation, we denote with $\mathbb{P}_{\mep}$ and $\mathbb{E}_{\mep}$ the probability and expectation when the mean of the  data generating distribution is equal to $\mep$. (In many parametric models, one can safely assume that there is a one to one map between $\mathbb{P}_{\mep}$ and $\mep$ and, therefore, that the mean value parametrization is well-defined, but in the nonparametric settings considered in this paper there could be more than one distribution with the same mean $\mu$; as discussed in Remark~\ref{rem:PmuEmu} below, this does not affect the validity of our results.) 
For a fixed $\alpha \in (0,1)$, we are concerned with developing level $\alpha$ open-ended tests ---also known as tests with power one \cite{robbins1970statistical}--- for the  one-sided hypothesis problem about $\mu$ of the form
\begin{equation} \label{eq::one_sided_hypo}
      H_0 : \mep \leq \mep_0 ~~\text{vs.}~~ H_1: \mep > \mep_1,  \text{ for some $\mep_1 \geq \mep_0$. }
\end{equation}
Formally, a level $\alpha$ open-ended test consists of a stopping time $\st$ with respect to the natural filtration generated by the data,  which satisfies the constraints
\begin{equation} \label{eq::constraint}
\begin{aligned}
\text{Bounded type-1 error: }\quad    \mathbb{P}_{\mep}\left(\st < \infty\right) &\leq \alpha,~~\text{if }~ \mep \leq \mep_0, \\
\text{Asymptotically power one: } \quad    \mathbb{P}_{\mep}\left(\st < \infty\right) & = 1 ,~~\text{if }~ \mep  > \mep_1.
\end{aligned}    
\end{equation}
At each time $n$, we either stop and reject the null if $N = n$ or  continue sampling if $N>n$. In particular, when $\mu \leq \mu_0$, we never stop with probability at least $1-\alpha$. The case of $\mu_0=\mu_1$ in which there is no separation between the null and alternative hypothesis will be of special interest.

The possibility of sampling indefinitely is characteristic of tests of power one~\cite{robbins1974expected}, and stems from allowing for arbitrarily small signal strengths (or, equivalently of no separation between the null and the alternative). This might initially be viewed as an undesirable property. However, open-ended tests are typically adaptive to the underlying signal strength and will stop early when $\mu$ is much larger than $\mu_1$. Indeed, open-ended tests not only have practical applications such as post-marketing drug and vaccine safety surveillance \cite{kulldorff2011maximized} but also serve as building blocks of more complicated sequential analyses. For instance, if we want to relax the power one constraint to a level $\beta \in (0,1)$ of type-2 error control, we can introduce another stopping time $M$ corresponding to a level $\beta$ open-ended test with swapped null and alternative hypotheses. Then, the minimum of the two stopping times $N$ and $M$ can be used as a sequential testing procedure that simultaneously controls type-1 and type-2 error at level $\alpha$ and $\beta$, respectively. In this case, with probability $1$ the procedure will stop in finite time under both the null and alternative hypotheses. 

 The expected sample size $\mathbb{E}_{\mep}N$ under an alternative distribution with mean $\mep > \mep_1$ is a traditional and widely used measure to quantify the performance of a sequential testing procedure satisfying the error constraints in \eqref{eq::constraint}; see, e.g.,  \cite{wald1948optimum, farrell1964asymptotic,robbins1974expected, tartakovsky.book}. In particular, the smaller the expected time to rejection under the alternative, the better the test. In parametric settings, if the testing problem is simple,
i.e., if it takes the form  
\begin{equation} \label{eq::simple_hypo}
    H_0 : \mep = \mep_0 ~~\text{vs}~~ H_1: \mep = \mep_1
\end{equation}
for some $\mu_0 \neq \mu_1$,
then the optimal testing procedure  is the sequential probability ratio test (SPRT), originally put forward by Wald~\cite{wald1945sequential} and further studied by Wald and Wolfowitz~\cite{wald1948optimum}. In detail, based on observations $X_1, X_2, \dots$, the one-sided SPRT is defined by the stopping time 
\begin{equation}
    \st := \inf\left\{ \n \geq 1: \log \LR_\n(\mep_1,\mep_0) \geq A \right\},
\end{equation}
where  $A \geq 0$ is an appropriately chosen threshold and $\LR_\n(\mep_1,\mep_0)$ is the likelihood ratio (LR) statistic given by
\begin{equation}
    \LR_\n(\mep_1,\mep_0) := \prod_{\ii=1}^\n \frac{p_{\mep_1}(X_\ii)}{p_{\mep_0}(X_\ii)}.
\end{equation}
Here $p_{\mep_1}$ and $p_{\mep_0}$ are the probability densities functions (with respect to a common dominating measure) of the data generating distributions under the alternative and null hypothesis,  respectively. If the threshold $A$ is chosen to satisfy the constraint in \eqref{eq::constraint}, then the one-sided SPRT is optimal in the sense that it minimizes $\mathbb{E}_{\mep_1} \st$, the expected sample size for a rejection under the alternative, among all test satisfying the error constraints; see, e.g.,  \cite{wald1948optimum, chow1971great}. 

More generally, one may wish to design a sequential test that satisfies the error constraints \eqref{eq::constraint} while nearly minimizing the expected sample size  uniformly over many possible null and alternatives distributions. However, it is well known that in this composite setting the SPRT does not yield such a guarantee~\cite{kiefer1957some, lai1988nearly}.

A natural way to extend the SPRT to accommodate composite null and alternative hypotheses in parametric settings is to use the generalized likelihood ratio (GLR) statistic:
\begin{equation}
 \GLR_\n(\mep_1,\mep_0) :=  \frac{\sup_{\mep  > \mep_1~\text{or}~\mep \leq \mep_0} \prod_{\ii=1}^\n  p_{\mep} (X_\ii)}{\sup_{\mep \leq \mep_0} \prod_{\ii=1}^\n  p_{\mep} (X_\ii)}.
\end{equation}
(Note that the above GLR statistic cannot be smaller than one by definition.) The one-sided sequential GLR (SGLR) test can then be defined by the stopping time
\begin{equation}
    \st_g := \inf\left\{ \n \geq 1:  \log \GLR_\n(\mep_1, \mep_0) \geq g_\alpha(\n) \right\},
\end{equation}
where $ g_\alpha : \mathbb{N} \to [0, \infty)$ is a {\it boundary function,} appropriately chosen to ensure that the error constraints in~\eqref{eq::constraint}, namely $\sup_{\mep \leq \mep_0}\mathbb{P}_\mep (N_g < \infty) \leq \alpha$ and $\inf_{\mep > \mep_1} \mathbb{P}_\mep (N_g < \infty) = 1$, are fulfilled.  
 
For exponential families, Farrell~\cite{farrell1964asymptotic} derived a sharp asymptotic lower bound on the expected sample size of any composite sequential test satisfying the constraint~\eqref{eq::constraint} in the moderate confidence regime in which the testing level $\alpha$ is fixed but $\mu_1$ approaches $\mu_0$.  The author further proposed a procedure  to threshold the GLR statistic that attains this lower bound in the limit as the gap  $|\mep_1 - \mep_0|\to0$, but did not  provide an explicit boundary function $g_\alpha$. Lorden \cite{lorden1973open} obtained an explicit boundary and nonasymptotic bounds on the testing errors and the expected sample sizes for well-separated alternatives ($\mep_1 > \mep_0$); see \cref{subSec::well-separated}.

\begin{remark}
Throughout the paper, we will follow the convention from the sequential analysis literature of using the term asymptotic to describe a vanishing separation between the null and the alternative hypotheses or a vanishing value of $\alpha$.  Accordingly, we will say that a result holds nonasymptotically when it holds for all finite values of $\mu_0$, $\mu_1$ or $\alpha$ and not only in the limit.
\end{remark}

In the more challenging non-separated case ($\mep_1 = \mep_0$),  virtually all of the existing SGLR tests are designed under parametric assumptions. In this case, one may hope to calibrate a boundary to a desired level $\alpha$ using simulations, but this is a non-trivial task for one-sided, open-ended tests because the type-1 error guarantee must hold for an infinite time horizon. Further, in nonparametric settings, the choice of which distribution to use is itself not obvious. Our analytic boundaries solve these issues. Also, most older works only deliver asymptotic analyses of  error bounds and of  expected sample sizes in the high-confidence regime where $\alpha \to 0$; see, e.g., \cite{chernoff1961sequential, lai1988nearly}. In contrast, our boundaries allow for a thorough nonasymptotic analysis in this setting.

It is also important to point out that the recent literature on best-arm identification  has produced several time-uniform law of iterated logarithm (``finite LIL'') bounds that allow for nonasymptotic analyses of type-1 error bounds with explicit boundary functions; see \cite{jamieson_lil_2014, zhao2016adaptive, howard2021time, kaufmann2018mixture, garivier2021nonasymptotic}. However, most existing boundaries are of a rigid form and are difficult to tune to be as tight as possible over any target time interval. Further, the expected sample sizes of corresponding testing procedures have been studied mostly in the asymptotic framework of the high confidence regime in which $\alpha \to 0$ but $\mep$  is fixed. As a result, nonasymptotic expected sample size analyses for the ``moderate confidence regime'' (fixed $\alpha$, $\mu \to \mu_0$) are not yet thoroughly studied beyond the (sub-)Gaussian case. In this paper, we bring to bear and sharpen tools from this line of work and apply them in novel ways to the problem of designing SGLR tests. 

Below we outline our main contributions, which advance both the theory and practice of sequential testing based on the GLR statistic. A technical summary of our results can be also found in \cref{tab::summary} in the Discussion \cref{sec:discussion}.

\begin{enumerate}
    \item We design new SGLR tests that satisfy the error constraints in \eqref{eq::constraint} with an explicit boundary function that are applicable in nonparametric settings in which a likelihood function is not available. Specifically, we present a unified analysis of sequential testing for sub-Gaussian, sub-exponential, and exponential family distributions (among others) via a new geometric interpretation of GLR statistics. 
    \item We derive novel nonasymptotic bounds on the sample size for the SGLR tests  that hold both in expectation and with high probability and are valid under any alternative. Though our results apply to nonparametric families of distributions, the bounds match in rate the known lower bounds for exponential family distributions in the moderate confidence regime  where $\alpha$ is fixed and $|\mep_1 - \mep_0| \to 0$.  
    \item Leveraging the duality between sequential tests and confidence sequences~\cite[Section 6]{howard2021time}, we develop a flexible method to construct confidence sequences which can be easily tuned to be uniformly close to the fixed-sample Chernoff bound on prespecified time intervals. 
\end{enumerate}

The rest of the paper is organized as follows. In \cref{sec::sub-psi}, we introduce the sub-$\psi_{\Mep}$ family of distributions, a nonparametric generalization of the exponential family which includes sub-Gaussian, sub-exponential, and exponential family distributions as special cases. \cref{sec::GLR-like test} presents the sequential GLR-like (SGLR-like) test, which is a nonparametric counterpart of the SGLR test for the sub-$\psi_{\Mep}$ family of distributions. We then derive nonasymptotic bounds on expected sample sizes, which demonstrate that the proposed SGLR-like test can detect the alternative signal in a sample efficient way. In \cref{sec::CS}, we introduce a flexible method to build anytime-valid confidence sequences that can be tuned to be close to the pointwise Chernoff bound on target time intervals. We conclude with a brief summary of our contribution and discussions on  future directions. In the interest of space, we defer  proofs and simulations to the supplement. 

\section{GLR statistic for the exponential family and its nonparametric extension } \label{sec::sub-psi}
Before presenting our main results in full generality and in order to build some intuition for our results, we first review the GLR statistic in the standard setting of exponential families of distributions. Consider a natural exponential family of distributions with densities of the form 
\begin{equation} \label{eq::EF_density}
    p_{\nat}(x) = \exp\left\{\nat x - B(\nat) \right\}, ~~\nat \in \Nat \subset \mathbb{R},
\end{equation}
with respect to a reference Borel measure $\nu$ on the real line,  where $\Nat \subset \left\{\nat \in \mathbb{R} :  \int e^{\nat x } \nu(\mathrm{d}x) < \infty \right\}$ is the natural parameter space and $B \colon \Nat \rightarrow \mathbb{R}$ is a strictly convex function given by $\nat \mapsto B(\theta) = \int e^{\nat x } \nu(\mathrm{d}x)$.  Throughout, we assume  $\Nat$ to be nonempty and open. For each natural parameter $\nat \in \Nat$, let $\mep = \mep(\theta)$ be the corresponding mean parameter
\[
\mep = \int x p_{\nat}(x) \nu(\mathrm{d}x). 
\]
It is well known \cite{barndorff2014information} that, under the stated assumptions, there is a one-to-one correspondence between natural and mean parameters via $\mep = \nabla B (\nat)$, where throughout the manuscript, for a differentiable univariate function $f$, $\nabla f$ will refer to its derivative function. Therefore, we can reparameterize the  exponential family based on the mean parameter space $M := \left\{ \nabla B(\nat) : \nat \in \Nat \right\}$, so that,  each point $\mu \in M$ will uniquely identify the density $p_\mu := p_{(\nabla B)^{-1}(\mu)}$.

When the samples $X_1,X_2, \ldots$ are i.i.d. from a distribution in the family, for any $\mep_0, \mep_1 \in \Mep$, the likelihood ratio (LR) statistic based on first $\n$ samples is defined by
\begin{equation}\label{eq:LR}
    \LR_{\n}(\mep_1, \mep_0):= \prod_{\ii=1}^\n  \frac{p_{\mep_1} (X_\ii)}{p_{\mep_0} (X_\ii)}.
\end{equation}
For fixed choices of $\mep_0$ and $\mep_1$, it is well-known that the normalized log LR statistics can be expressed as a function of the sample mean $\bar{X}_{\n} = \frac{1}{n}  \sum_{i=1}^n X_i $ in two equivalent forms as follows:
\begin{align}
\frac{1}{n}\log \LR_{\n}(\mep_1, \mep_0) &= \KL\left(\bar{X}_{\n}, \mep_0 \right) - \KL \left(\bar{X}_{\n}, \mep_1 \right) \label{eq::LR_diff_form} \\
& = \KL\left(\mep_1, \mep_0\right) + \nabla_z \KL\left(z, \mep_0\right)|_{z= \mep_1} \left(\bar{X}_{\n} - \mep_1\right) \label{eq::LR_tangent_line_form},
\end{align}
where $\KL (z_1, z_2)$ is the Kullback-Leibler (KL) divergence from $p_{z_2}$ to $p_{z_1}$, for $z_1, z_2 \in \Mep$. For the completeness, we derive the above identities in Appendix~\ref{appen:facts_on_sub_psi}. From the first expression in \eqref{eq::LR_diff_form}, we see that the normalized LR statistic is equal to the difference between the KL divergences from the distribution in the family parametrized by the sample mean $\bar{X}_{\n}$ to the one corresponding to the null and alternative hypotheses. Perhaps more importantly for our derivations below, the second expression in \eqref{eq::LR_tangent_line_form} shows that the normalized log LR statistic, as a function of $\bar{X}_{\n}$, is also the tangent line to the KL divergence function $z \mapsto \KL(z,\mep_0)$ at $z = \mep_1$.  See \cref{fig::GLR_stat} for an illustration.

Now, recall that we are concerned with one-sided composite testing problem 
\begin{equation} \label{eq::GLR_hypo}
    H_0: \mep \leq \mep_0 ~~\text{vs}~~H_1: \mep > \mep_1,
\end{equation}
where $\mep_1 \geq \mep_0 \in \Mep$. From the expression of the LR statistic in \eqref{eq::LR_diff_form}, it can be easily shown that the corresponding GLR statistic is given by 
\[
\GLR_n(\mep_1, \mep_0) = \sup_{z > \mep_1} \LR_\n (z, \mep_0) \vee 1.
\]
Using the alternative expression of the LR statistic in \eqref{eq::LR_tangent_line_form}, we  conclude that the normalized log GLR statistic can be written as $\frac{1}{\n}\log\GLR_\n(\mep_1, \mep_0)  := f(\bar{X}_{\n};\mep_1,\mep_0)$, where
\begin{equation} \label{eq::GLR_fn}
\begin{aligned}
f(z ; \mep_1, \mep_0):=
    \begin{cases} 
    \left[\KL(\mep_1,\mep_0) + \nabla_z \KL(z, \mep_0)|_{z= \mep_1}(z -\mep_1)  \right]_+  & \mbox{if } z\leq \mep_1 \\
    \KL(z, \mep_0) &\mbox{if } z> \mep_1
    \end{cases}.
\end{aligned}
\end{equation}

\begin{figure}
    \begin{center}
    \includegraphics[scale =  0.75]{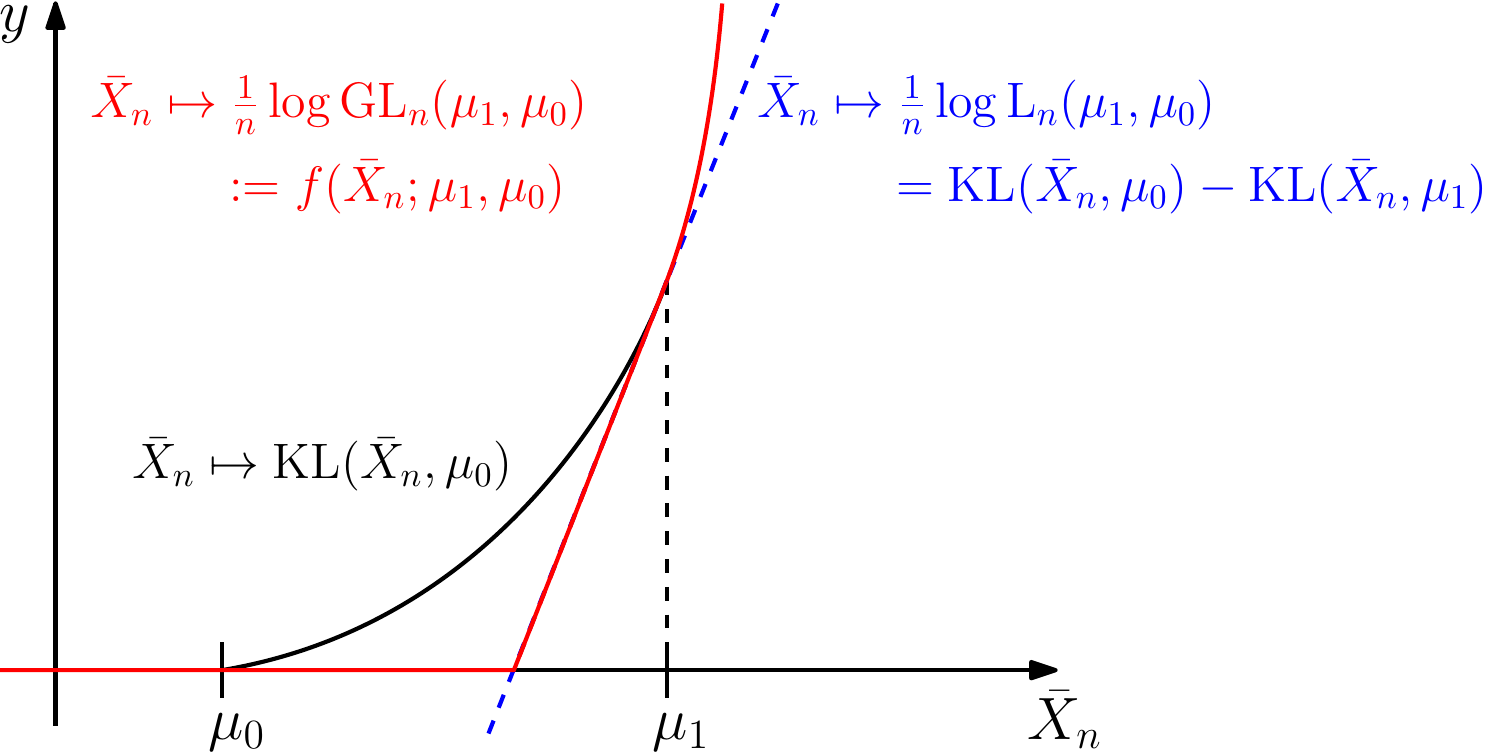}
    \end{center}
    \caption{ Illustration of normalized log LR and GLR statistics for exponential family distributions.  The dashed blue line corresponds to the normalized log LR statistic as a function of the sample mean $\bar{X}_{\n}$ which is tangent to the KL divergence function at $\mep_1$. The red line shows the normalized log GLR statistic which is equal to the KL divergence for $\bar{X}_{\n} > \mep_1$ and its ``clipped'' tangent line for $\bar{X}_{\n} \leq \mep_1$, respectively. } 
    \label{fig::GLR_stat}
\end{figure}

\cref{fig::GLR_stat} depicts the normalized log LR (dashed blue line) and GLR statistics (red line) based on the first $n$ samples, as functions of  $\bar{X}_{\n}$.
In particular, the normalized log GLR statistic is equal to the KL divergence between $p_{\bar{X}_{\n}}$ and $p_{\mep_0}$ when $\bar{X}_{\n} > \mep_1$ and to its tangent line at $\mep_1$, clipped at zero, otherwise.

The above expression for the normalized GLR statistic and its simple but revealing geometric interpretation provide the conceptual underpinning and intuition for much of the contributions made in this article.
In the next subsection, we introduce a class of distributions, called sub-$\psi$ distributions \cite{howard2020time, shin2019bias} that exhibit analogous GLR statistic for the testing problem at hand and thus  can be viewed as a natural nonparametric generalization of the exponential family distributions.

\subsection{Extensions to the \texorpdfstring{sub-$\psi_{\Mep}$}{sub-psi} family distributions} \label{subSec::def_sub-psi}
In this section we will assume throughout some familiarity with basic concepts from convex analysis; see, e.g.,  \cite{rockafellar1997convex}. Let $M$ be an open, convex subset of $\mathbb{R}$ such that, for each $\mep \in M$, there exists  an extended real-valued convex function $\psi_{\mep}$ that is finite and strictly convex on a common closed supporting set $\Lambda \subset \mathbb{R}$ and differentiable on its nonempty interior $\Lambda^{\mathrm{o}}$ containing $0$, with $\psi_{\mep}(0) = 0$ and $\nabla\psi_{\mep}(0) = \mep$. We also assume that, for each $\mep \in M$, the the convex conjugate of $\psi_{\mep}$, denoted with  $\psi_{\mep}^*$, is finite and differentiable on $\Mep$.

For any $\mep \in M$, a collection $\mathcal{P}$ of probability distributions over the real line is said to be a {\it sub-$\psi_{\mep}$ family} if, for each $P \in \mathcal{P}$,  $\mathrm{supp}(P) \subset \overline{\Mep}$, $\mathbb{E}_{X\sim P}[X] = \mep$ and 
\begin{equation}\label{eq:psi_mu}
\log \mathbb{E}_{X \sim P}\left[ e^{\lambda X}\right] \leq \psi_{\mep}(\lambda),~~\forall \lambda \in \Lambda.
\end{equation}
We will denote such class with $\mathcal{P}_{\psi_{\mep}}$. Notice that, for each  $\mep \in M$, $\mathcal{P}_{\psi_{\mep}}$ is a nonparametric statistical model and the set $M$ plays the role of all possible mean parameters of interest. Finally, we will write
\[
\mathcal{P}_{\psi_{\Mep}}:=\bigcup_{\mep \in \Mep} \mathcal{P}_{\psi_{\mep}}
\] 
for the collection of all sub-$\psi_{\mep}$ families as $\mu$ ranges in $M$, which we will then refer to as a {\it sub-$\psi_{\Mep}$ family of  distributions.}

We will further require that a sub-$\psi_{\Mep}$ family satisfies the following \emph{order-preserving} property for the Bregman divergences arising from the conjugate functions $\psi^*_{\mep}$. 

\begin{definition}
For each $\mep \in \Mep$, let $D_{\psi_{\mep}^*}(\cdot, \cdot)$ be the Bregman divergence with respect to $\psi_{\mep}^*$. We say a sub-$\psi_{\Mep}$ family of distributions has an order-preserving class of Bregman divergences  if 
\begin{equation}
    D_{\psi^*_{z_0}}(\mep_1, z_0) \geq     D_{\psi^*_{\mep_0}}(\mep_1, \mep_0)
\end{equation}
holds for any $z_0, \mep_0, \mep_1 \in \Mep$ such that  $\mep_1 \leq \mep_0 \leq z_0$ or $z_0 \leq \mep_0 \leq \mep_1$.
\end{definition}
At a high-level, the order-preserving property  implies that the Bregman divergence expresses a natural ordering of a sub-$\psi_{\Mep}$ family  with respect to the mean parametrization. This ordering is naturally suited to handle the hypothesis testing problem we are studying.  As an important special case, sub-Gaussian distributions with a common  variance parameter $\sigma^2$ form a sub-$\psi_{\Mep}$ family with an order-preserving class of Bregman divergence where $\Mep = \Lambda = \mathbb{R}, \psi_{\mep}(\lambda) = \lambda \mep + \frac{\sigma^2}{2} \lambda^2$ and $\D(\mep_1,\mep_0) = \frac{1}{2\sigma^2}(\mep_1 -\mep_0)^2$. Another important example of a sub-$\psi_{\Mep}$ family of distributions with an order-preserving class of the Bregman divergences is the family of Bernoulli distributions where we have $\Mep = (0,1), \Lambda = \mathbb{R}, \psi_{\mep}(\lambda) = \log\left(1-\mep + \mep e^\lambda\right)$ and $\D(\mep_1,\mep_0) = \KL(\mep_1,\mep_0) = \mep_1 \log\left(\frac{\mep_1}{\mep_0}\right) + (1-\mep_1) \log\left(\frac{1-\mep_1}{1-\mep_0}\right)$. These two examples are representative distributions belonging to the following  two large nonparametric classes of  distributions:
\begin{enumerate}
    \item {\bf Additive sub-$\psi$ distributions.} A sub-$\psi_{\Mep}$ family is said to be additive if for any $\mep \in M$,
    \begin{equation}
 \psi_{\mep}(\lambda) - \lambda \mep = \psi_0(\lambda) := \psi(\lambda)~~~ \text{ for all } \lambda \in \mathbb{R}.
 \end{equation}
It can be checked that sub-Gaussian and sub-exponential distributions are instances of additive sub-$\psi$ distributions. For each additive sub-$\psi$ distribution, the Bregman divergence can be expressed as $\D(\mep_1, \mep_0) = \psi^*(\mep_1 - \mep_0)$ for each $\mep_1, \mep_0 \in \Mep$.

\item {\bf Exponential family-like (EF-like) sub-$B$ distributions.}  A sub-$\psi_{\Mep}$ family of distributions is called an EF-like sub-$B$ family if there exists an extended real-valued convex function $B$ which is finite, strictly convex  and differentiable on $\Lambda := (\nabla B)^{-1}(\Mep)$ such that, for each $\mep \in \Mep$,
    \begin{equation}
 \psi_{\mep}(\lambda)  = B(\lambda + \theta_\mep) - B(\theta_\mep),~\forall\lambda \in \mathbb{R},
 \end{equation}
 where $\theta_\mep := (\nabla B)^{-1}(\mep)$. 
All exponential family distributions and sub-Gaussian distributions are instances of EF-like sub-$B$ distributions. In fact, for exponential family of distributions, it is immediate to see that each $\psi_{\mep}$ is the logarithm of the moment generating function. For each EF-like sub-$B$ distribution, the Bregman divergence can be written as  $\D(\mep_1, \mep_0) = D_{B^*}(\mep_1. \mep_0)$, which, for a class of exponential families with densities satisfying \eqref{eq::EF_density} is equal to the dual of the KL divergence; see, e.g., \cite{JMLR:v6:banerjee05b}. 
\end{enumerate} 
From the expressions of the Bregman divergence described above, it follows that all additive sub-$\psi$ and EF-like sub-$B$ families have classes of Bregman divergences satisfying the order-preserving property. For completeness, this fact and related properties of sub-$\psi_\Mep$ family of  distributions are proven in Appendix~\ref{appen:facts_on_sub_psi}.

We now describe how to construct LR- and GLR-like statistics for the nonparametric class of distributions $\mathcal{P}_{M}$ in ways that mirror exactly the derivation of the LR and GLR statistics in exponential families, as described in the previous section. 
We will assume throughout that  $X_1, X_2, \dots$ is a sequence of independent, though not necessarily identically distributed, random variables with the same but unknown finite mean $\mep$, each drawn from a distribution belonging to a common   sub-class $\mathcal{P}_{\psi_{\mep}}$ of a sub-$\psi_\Mep$ family of distributions. \begin{remark}\label{rem:PmuEmu}
Slightly overloading notation, we denote with $\mathbb{P}_{\mep}$ and $\mathbb{E}_{\mep}$ the probability and expectation for the stochastic process $(X_\n)_{\n \in \mathbb{N}}$. That is, every statement written with respect to $\mathbb{P}_{\mep}$ and $\mathbb{E}_{\mep}$ refers to the case where each independent observation of the underlying stochastic process $(X_\n)_{\n \in \mathbb{N}}$ has a distribution in $\mathcal{P}_{\psi_{\mep}}$ with the same mean $\mep$.  We emphasize that, because of the nonparametric nature of our models, the observations need not be identically distributed. Further, since there could be more than one distribution with mean $\mep$, every statement related to $\mathbb{P}_{\mep}$ and $\mathbb{E}_{\mep}$ should be understood as a reference to any possible sub-$\psi_{\mep} $ distribution of   $(X_\n)_{\n \in \mathbb{N}}$. 
\end{remark}

 Now, consider the following test for the mean:
\begin{equation}
H_0 : \mep = \mep_0~~\text{vs}~~H_1 : \mep = \mep_1,
\end{equation}
for some $\mep_1, \mep_0 \in \Mep$. Let $\lambda_1 := \nabla \psi_{\mep_0}^{*}(\mep_1)$, and define
\begin{equation} \label{eq::LR-like_def}
       \LR_{\n}(\mep_1, \mep_0) := \exp\left\{\n\left[ \lambda_1\bar{X}_{\n} - \psi_{\mep_0}(\lambda_1)\right]\right\}.
\end{equation}
The above expression has the same form of the likelihood ration statistics \eqref{eq:LR} for parametric exponential families. Thus, with a slight abuse of notation, we refer to $\LR_{\n}(\mep_1, \mep_0)$ as the {\it LR-like statistic} for the above simple hypothesis testing based on first $\n$ samples.

Just like in the case of an exponential family, for a sub-$\psi_{\Mep}$ family of distributions,  the normalized log LR-like statistics can be re-written as
\begin{align}
\frac{1}{n}\log \LR_{\n}(\mep_1, \mep_0) &= \D\left(\bar{X}_{\n}, \mep_0 \right) - \D \left(\bar{X}_{\n}, \mep_1 \right) \label{eq::LR-like_diff_form} \\
& = \D\left(\mep_1, \mep_0\right) + \nabla_z \D\left(z, \mep_0\right)|_{z= \mep_1} \left(\bar{X}_{\n} - \mep_1\right) \label{eq::LR-like_tangent_line_form},
\end{align}
where $D_{\psi_{\mep_0}^*} (z_1, z_2)$ is the Bregman divergence from $z_2$ to $z_1$ with respect to $\psi_{\mep_0}^*$ for  $z_1, z_2 \in \Mep$. See Appendix~\ref{appen:facts_on_sub_psi} for details. Thus, similarly to the exponential family case, the normalized LR-like statistic is equal to the difference between two Bregman divergences, one from each hypothesis, to the sample mean. Also, from the second expression in \eqref{eq::LR-like_tangent_line_form}, we can check that the normalized log LR-like statistic is given by the tangent line to the Bregman divergence function $z \mapsto \D(z,\mep_0)$ at $z = \mep_1$. This is not a coincidence. Recall that every exponential family distribution is an EF-like sub-$B$ distribution. In this case, the corresponding LR and LR-like statistics are equal to each other since, for any $\mep_0, \mep_1 \in \Mep$, 
\begin{equation} \label{eq::LR-like_LR_equiv}
 D_{\psi_{z}^*}(\mep_1, \mep_0)  =  D_{B^*}(\mep_1, \mep_0)= \KL(\mep_1, \mep_0),~~\forall z \in \Mep.
\end{equation}
Finally, based on the definition of the LR-like statistic, for the one-sided testing problem 
\[
H_0: \mep \leq \mep_0 ~~\text{vs}~~H_1: \mep > \mep_1,
\]
with $\mep_1 \geq \mep_0 \in \Mep$, the {\it GLR-like statistic} is defined by
\begin{equation} \label{eq::GLR-like_def}
    \GLR_\n (\mep_1, \mep_0) := \sup_{z > \mep_1}  \LR_\n (z,\mep_0) \vee 1.
\end{equation}
From the expressions of LR statistic in \eqref{eq::LR-like_tangent_line_form}, we can derive a closed form expression for the normalized log GLR-like statistic  $\frac{1}{\n}\log\GLR_\n(\mep_1, \mep_0)  := f(\bar{X}_{\n};\mep_1,\mep_0)$, where
\begin{equation} \label{eq::GLR-like_fn}
	f(z ; \mep_1, \mep_0):=
	\begin{cases} 
		\left[\D(\mep_1,\mep_0) + \nabla_z \D(z, \mep_0)\mid_{z = \mep_1}(z -\mep_1)  \right]_+  & \mbox{if } z\leq \mep_1 \\
		\D(z, \mep_0) &\mbox{if } z >\mep_1
	\end{cases}. 
\end{equation}
This expression matches exactly the one for the normalized GLR statistics in exponential families given in \eqref{eq::GLR_fn} and in fact recovers it as a special case since for exponential families the Bregman divergence correspond to the KL divergence.
\begin{figure}
    \begin{center}
    \includegraphics[scale =  0.75]{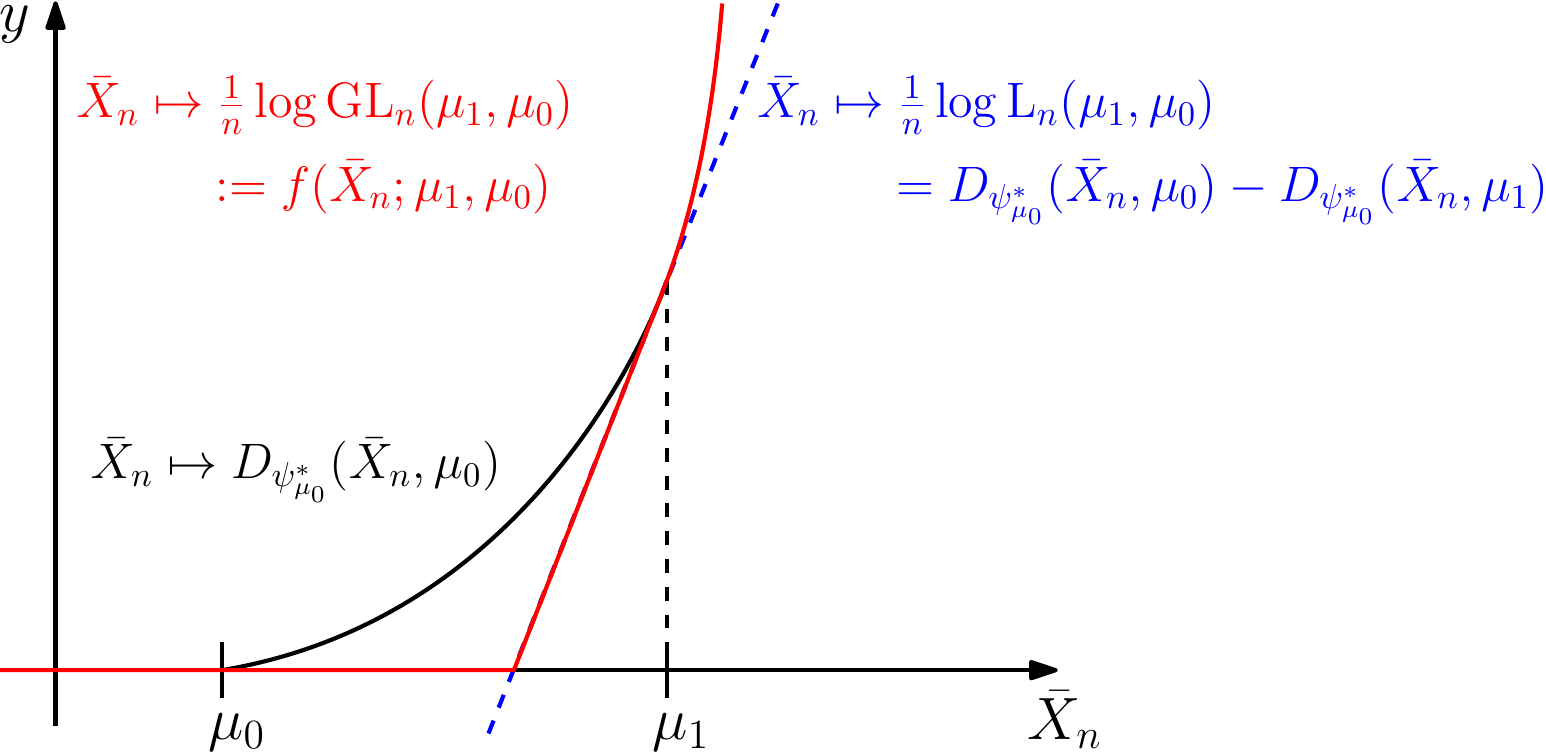}
    \end{center}
    \caption{ Illustration of normalized log LR-like and GLR-like statistics for the sub-$\psi_{\Mep}$ family of distributions. The dashed blue line corresponds to the normalized log LR-like statistic as a function of the sample mean $\bar{X}_{\n}$ which is tangent to the Bregman divergence function at $\mep_1$. The red line shows the normalized log GLR-like statistic which is equal to the Bregman divergence for $\bar{X}_{\n} > \mep_1$ and its ``clipped'' tangent line for $\bar{X}_{\n} \leq \mep_1$, respectively.} 
    \label{fig::GLR-like_stat}
\end{figure}

See \cref{fig::GLR-like_stat} for an illustration of the relationship between normalized log LR-like and GLR-like statistics based on first $n$ samples which is identical to the exponential family case in \cref{fig::GLR_stat}. The dashed blue line corresponds to the normalized log LR-like statistic as a function of the sample mean $\bar{X}_{\n}$ which is tangent to the Bregman divergence function at $\mep_1$. The red line shows the normalized log GLR-like statistic which is equal to the Bregman divergence for $\bar{X}_{\n} > \mep_1$ and its ``clipped'' tangent line for $\bar{X}_{\n} \leq \mep_1$, respectively. Since both LR- and GLR-like statistics depend only on the sample mean $\bar{X}_n$ and not on the entire history, we can update both test statistics (at each step) using constant time and memory. This property makes it possible to run sequential tests and confidence sequences introduced below in a fully online fashion.

\section{SGLR-like test for \texorpdfstring{sub-$\psi_{\Mep}$}{sub-psi} family distributions} \label{sec::GLR-like test}

\noindent We now describe sequential tests based on the GLR-like statistics for one-sided hypotheses
\begin{equation} \label{eq::GLR_hypo-sub-psi}
    H_0: \mep \leq \mep_0 ~~\text{vs}~~H_1: \mep > \mep_1,
\end{equation}
for some fixed $\mep_1 \geq \mep_0 \in \Mep$. 

Consider the settings described in the previous section and let $\alpha \in (0, 1)$ be fixed. A level $\alpha$ SGLR-like test is defined by a stopping time  of the form
\begin{equation} \label{eq::GLR_one_sided_g}
    \st_{\GLR} := \inf\left\{\n \geq 1:  \log\GLR_\n(\mep_1, \mep_0) \geq g_\alpha(\n)\right\},
\end{equation}
where $g_\alpha$ is a given, positive function  on $[1,\infty)$. In particular, if $\mep_1 = \mep_0$ then 
the SGLR-like test can be simplified to
\begin{equation}\label{eq::GLR_one_sided_indiff}
    \st_{\GLR} = \inf\left\{\n \geq 1:  \bar{X}_{\n} \geq \mep_0, n \D(\bar{X}_{\n}, \mep_0) \geq g_\alpha(\n)\right\}.
\end{equation}

At each time point, we check if the stopping criteria are met and, if so, we reject the null hypothesis. In general, by the law of large numbers, if $g_\alpha(n)/n \to 0$ as $n \to \infty$, then the power one guarantee  $\mathbb{P}_{\mep}(\st_{\GLR} <\infty) = 1$ can be easily satisfied for each $\mep >\mep_1$. 
However, it is a nontrivial task to design a proper boundary function $g_\alpha$ satisfying the  type-1 error control
 \begin{equation} \label{eq::type1err_control}
    \sup_{\mep \leq \mep_0}\mathbb{P}_{\mep} \left(\st_{\GLR} < \infty \right)
  \leq \alpha.
 \end{equation}
To tackle this challenge, we develop a general methodology to  bound the boundary crossing probability of the event in \eqref{eq::GLR_one_sided_g}. This result plays a key role in building the SGLR-like test. 

\begin{theorem}\label{thm::boundary_crossing_GLR}
  Let the boundary function $g: [1,\infty) \to [0,\infty)$ satisfy the following conditions:
    \begin{enumerate}
        \item $g$ is nonnegative and nondecreasing;
        \item the mapping $t \mapsto g(t) / t$ is nonincreasing on $[1,\infty)$ and $\lim_{t\to \infty} g(t) /t = 0$. 
    \end{enumerate}
    Then, for any sub-$\psi_{\Mep}$ family of distributions, the crossing probability under the null is such that
    \begin{align} \label{eq::GLR_bound_thm1}
    &\sup_{\mep \leq \mep_0}\mathbb{P}_{\mep} \left(\exists\n \geq 1:   \log\GLR_\n(\mep_1, \mep_0) \geq g(\n)\right) \\
    \label{eq::GLR_bound_thm1.2}
    &\quad \quad \quad \quad \quad \quad \leq \begin{cases} e^{-g(1)} & \mbox{if } \D(\mep_1,\mep_0) \geq g(1) ~, \\
    	\inf_{\eta >1} \sum_{k=1}^{K_\eta}\exp\left\{-g(\eta^k) / \eta\right\} & \mbox{otherwise } ~,
    \end{cases} 
\end{align}
where, for any $\eta > 1$,  $K_{\eta} \in \mathbb{N}\cup \{0, \infty\}$ is defined  by
\begin{equation} \label{eq::K_number}
K_{\eta}:= \inf\left\{k \in \{0\}\cup \mathbb{N}: \D(\mep_1,\mep_0) \geq \frac{g(\eta^k)}{\eta^{k}}\right\}.
\end{equation}
 \end{theorem}

The second boundary value for the probability of the curve-crossing event in the expression \eqref{eq::GLR_bound_thm1.2} is obtained using a technique known as ``stitching" or ``peeling" (see, e.g., \cite{howard2021time})  that is designed to derive uniform probabilistic guarantees over an infinite time horizon. Roughly speaking, the time course is divided into geometrically spaced time epochs, and in each of them, a line crossing inequality is derived.  The final bound is obtained by appropriately stitching together these separate linear boundaries. The parameter $\eta>1$ is the ratio of adjacent ``intrinsic time'' (accumulated variance) epochs used in the stitching process. A smaller value of $\eta$ yields a better approximation of the curve-crossing events but requires controlling a larger number (namely, $K_\eta$) of line-crossing events. The optimal choice of $\eta$ minimizing the bound achieves the optimal balance for this trade-off. The infimum over $\eta > 1$ in the bound explicitly shows how to choose the best  $\eta$.
The stitching construction is only required when $\D(\mep_1,\mep_0) \geq g(1)$. In the other case, it is sufficient to use a single line-crossing inequality. See Appendix~\ref{appen::proofs_of_thms} for a detailed explanation and a formal proof.

  \begin{remark}
In the above theorem, we follow the convention that $\inf \emptyset = \infty$. From the two conditions assumed for the boundary function $g$, we can check that $K_\eta = 0$ if and only if $\D(\mep_1,\mep_0) \geq g(1)$ and $K_\eta = \infty$ if and only if $\mep_1 = \mep_0$ for each $\eta > 1$. 
 \end{remark} 
{
 \begin{remark} \label{remark::thm1_for_process}
 If the mapping $\mep \to \psi_{\mep}(\lambda)$ is concave for each fixed $\lambda$, it is possible to extend \cref{thm::boundary_crossing_GLR} to a broader class of stochastic processes. In detail, let $(X_i)_{i \in \mathbb{N}}$ be a real-valued process adapted to an underlying filtration $(\mathcal{F}_i)_{i \in \{0\}\cup\mathbb{N}}$ such that, conditioned on $\mathcal{F}_{i-1}$, each $X_i$ follows a sub-$\psi_{\mep^i}$ distribution with $\mep^i := \mathbb{E}\left[X_i \mid \mathcal{F}_{i-1}\right]$. That is, for each $i \in \mathbb{N}$ and $\lambda \in \mathbb{R}$, 
\begin{equation} \label{eq::sub-psi_process}
    \log \mathbb{E}\left[ e^{\lambda X_i } \mid \mathcal{F}_{i-1}\right] \leq \psi_{\mep^i}(\lambda)  ~~a.s.
\end{equation}
Let $\mathcal{P}_{\mep_0}$ be the set of probability distributions of $(X_i)_{i \in \mathbb{N}}$ such that $\frac{1}{\n}\sum_{i=1}^\n \mep^i := \bar{\mep}_\n \leq \mep_0$ for all $\n$. Then, it is possible to show that
  \begin{align}
	&\sup_{P \in \mathcal{P}_{\mep_0}} P \left(\exists\n \geq 1:   \log\GLR_\n(\mep_1, \mep_0) \geq g(\n)\right) \\
	&\quad \quad \quad \quad \quad \quad \leq \begin{cases} e^{-g(1)} &\mbox{if } \D(\mep_1,\mep_0) \geq g(1) \\
		\inf_{\eta >1} \sum_{k=1}^{K_\eta}\exp\left\{-g(\eta^k) / \eta\right\} & \mbox{otherwise}.
	\end{cases}
\end{align}
    Note that the concavity condition for sub-$\psi_\mep$ function is satisfied by all additive sub-$\psi$ distributions and many important exponential families with discrete support, including Bernoulli, Poisson, Geometric and Negative binomial with known number of failures (Appendix~\ref{appen::proofs_of_thms}).
 \end{remark}
 }
 
 In the following subsections we will deploy  \cref{thm::boundary_crossing_GLR} to develop SGLR-like tests for the one-sided hypothesis \eqref{eq::GLR_hypo-sub-psi}. We will analyze separately the case in which the null and alternative hypotheses are well-separated ($\mep_1 >\mep_0$), and the case of no separation ($\mep_1 = \mep_0$).

\subsection{SGLR-like tests for well-separated alternatives} \label{subSec::well-separated}
In this subsection, we focus on the scenario where $\mep_1$ is strictly larger than $\mep_0$. In many applications, this separation condition can be derived from prior knowledge of the underlying distribution or from  requirements on the minimal effect sizes one seeks to detect. Furthermore, even if we intend to run an open-ended testing procedure, there might be an upper limit on the sample size due to time and budget constraints on the experiment.  In this case, the separation of null and alternative hypotheses can be imposed indirectly by the upper limit because if the null and alternative hypotheses are too close to each other then no fixed-level test can detect such a small separation given the upper limit on the sample size; see \cite{lai2004power}. In  \cref{rmk::how_to_choose_alternative} below, we will present a natural way of choosing the boundary of the alternative space for $\mep_1$ given an upper limit on sample size $\nmax$ for the sequential testing procedure proposed in this subsection. 

Now, choosing a constant boundary $g \in \mathbb{R}^+$, \cref{thm::boundary_crossing_GLR} immediately  yields that
\begin{equation} \label{eq::const_bound}
\begin{aligned}
        \sup_{\mep \leq \mep_0}\mathbb{P}_{\mep} \left(\exists \n \geq 1:   \log\GLR_\n(\mep_1, \mep_0) \geq g\right) 
 \leq \begin{cases} e^{-g} &\mbox{if } D_1 \geq g, \\
  \inf_{\eta >1} \left\lceil \log_\eta \left(\frac{g}{D_1}\right)\right\rceil e^{-g / \eta} & \mbox{otherwise}, \\
  \end{cases}
 \end{aligned}
\end{equation}
where $D_1 := \D(\mep_1,\mep_0)$.

\begin{remark}
For the constant boundary case, the term in the upper bound that depends on the infimum over $\eta > 1$ can be rewritten as 
\begin{equation} \label{eq::K_eta_computation}
     \inf_{\eta >1} \left\lceil \log_\eta \left(\frac{g}{D_1}\right)\right\rceil e^{-g / \eta} = \inf_{k \in \mathbb{N}} k \exp\left\{-g \left(\frac{D_1} {g}\right)^{1/k}\right\}.
\end{equation}
The right expression can be evaluated efficiently since it optimizes over integers, not reals. 
\end{remark}

For any given level $\alpha \in (0,1]$, let $ g_\alpha(\mep_1,\mep_0) > 0$ be a constant boundary value that  makes the right hand side of \eqref{eq::const_bound} equal to $\alpha$.  The boundary value $g_\alpha(\mep_1,\mep_0)$ can be numerically computed using  equation~\eqref{eq::K_eta_computation}. Alternatively, as shown in Appendix~\ref{appen::proof_of_props},  $g_\alpha(\mep_1, \mep_0)$ can be upper bounded as follows:
\begin{equation} \label{eq::const_g_upper_bound}
\begin{aligned}
    g_\alpha(\mep_1,\mep_0) \leq \inf_{\eta > 1}\left\{\eta \log\left(\frac{1}{\alpha}\left[1+ 2\log_\eta\left(\frac{\eta\sqrt{\eta}}{\alpha D_1 \log \eta}\vee 1\right)\right]\right)\right\}.
\end{aligned}
\end{equation}
This bound is slightly loose but useful for the numerical computation of $g_\alpha (\mep_1, \mep_0)$.

 Based on $g_\alpha(\mep_1,\mep_0)$, let $N_{\GLR}(g_\alpha, \mep_1,\mep_0)$ be the stopping time of the SGLR-like test defined by
\begin{equation} \label{eq::const_GLRT_stopping_time}
\begin{aligned}
    N_{\GLR}(g_\alpha, \mep_1,\mep_0)
    := \inf\left\{n \geq 1: \log \GLR_n(\mep_1, \mep_0) \geq g_\alpha(\mep_1,\mep_0) \right\}.
\end{aligned}
\end{equation}
By \cref{thm::boundary_crossing_GLR}, we can check that $N_{\GLR}(g_\alpha, \mep_1,\mep_0)$ induces a valid level $\alpha$ test.  Furthermore, from  Lorden's inequality \cite{lorden1970excess}, we can derive a nonasymptotic upper bound on the expected sample size under any alternative distribution, as shown next.

\begin{theorem} \label{thm::upper_bound_const}
Let $N_{\GLR}(g_\alpha, \mep_1,\mep_0)$ be the stopping time of the SGLR-like test defined in \eqref{eq::const_GLRT_stopping_time}. Then, $\sup_{\mep \leq \mep_0}\mathbb{P}_{\mep}\left(N_{\GLR}(g_\alpha, \mep_1,\mep_0) < \infty \right) \leq \alpha$. Furthermore, if the observations $X_1, X_2,\dots$ are i.i.d. then, for any $\mep \geq \mep_1 \in \Mep$, we have
\begin{equation}\label{eq::upper_bound_const_boundary}
    \begin{aligned}
        \mathbb{E}_{\mep}\left[N_{\GLR}(g_\alpha, \mep_1,\mep_0)\right]
        \leq \frac{g_\alpha(\mep_1,\mep_0)}{\D(\mep, \mep_0)} + \left[ \frac{\sigma_{\mep}\nabla\psi_{\mep_0}^*(\mep)}{\D(\mep, \mep_0)}\right]^2 + 1,
    \end{aligned}
\end{equation}
    where $\sigma_\mep^2 := \sup_{P \in \psi_\mep}\int (x-\mep)^2 \mathrm{d}P$ is the maximal variances of all the probability distributions in the sub-$\psi_{\mep}$ class which can be upper bounded as $\sigma_\mep^2  \leq \nabla^2 \psi_\mep(0) = 1 /\nabla^2 \psi_\mep^*(\mep) $.
\end{theorem}

In the special but highly relevant case of $\sigma^2$-sub-Gaussian distributions, the upper bound on the expected sample size in \eqref{eq::upper_bound_const_boundary} can be written as
\begin{equation}
        \mathbb{E}_{\mep}\left[N_{\GLR}(g_\alpha, \mep_1,\mep_0)\right] \leq \frac{2\sigma^2\left[ g_\alpha(\mep_1,\mep_0) + 2\right]}{(\mep-\mep_0)^2}  + 1.
\end{equation}
The proof of \cref{thm::upper_bound_const} can be found in Appendix~\ref{appen::proofs_of_thms}.

Under the same setting of well-separated hypotheses, Lorden \cite{lorden1973open} investigated the properties of the SGLR test for exponential family distributions with a constant boundary. For any constant $g > 0$, Lorden proved that
  \begin{equation} \label{eq::Lorden}
  \begin{aligned}
    \sup_{\mep \leq \mep_0}\mathbb{P}_{\mep} \left(\n \geq 1:   \log\GLR_\n(\mep_1, \mep_0) \geq g\right) 
   \leq \begin{cases} e^{-g} &\mbox{if } \KL(\mep_1,\mep_0) \geq g~, \\
\left(1 + \frac{g}{\KL(\mep_1,\mep_0)} \right)e^{-g} & \mbox{otherwise }.
  \end{cases}
 \end{aligned}
\end{equation}
In fact,  the above bound can be immediately extended to hold also over sub-$\psi_{\Mep}$ families by replacing the KL divergence term $\KL(\mep_1,\mep_0)$ with the Bregman divergence $\D(\mep_1, \mep_0)$.

To compare our result with Lorden's bound, let $g_\alpha^L(\mep_1,\mep_0)$ be the smallest boundary value one can obtain from Lorden's bound in \eqref{eq::Lorden} for any given $\alpha$. In general, neither our choice $g_\alpha(\mep_1,\mep_0)$ nor Lorden's choice $g_\alpha^L(\mep_1,\mep_0)$ dominates the other one. However, in the moderate confidence regime where $\alpha$ is fixed and $\mep_1 \to \mep_0$, our bound increases at an exponentially slower rate compared to the one stemming from  Lorden's $g_\alpha^L(\mep_1,\mep_0)$. In detail, we can check that in the moderate confidence regime we have
\begin{align}
    g_\alpha^L(\mep_1,\mep_0) &= O\left(\log\left(1/ \D(\mep_1,\mep_0)\right)\right), \\
    g_\alpha(\mep_1,\mep_0) &= O\left(\log\log\left(1/ \D(\mep_1,\mep_0)\right)\right), \label{eq::asym_g_alpha}
\end{align}
where the log-log dependency makes it possible to apply our testing method to cases in which the separation is exponentially small. This is well illustrated in Figure~\ref{fig::const_boundary} by comparing the curves  $|\mep_1-\mep_0|^{-1} \mapsto g_\alpha^L$ and $|\mep_1-\mep_0|^{-1} \mapsto g_\alpha$ for a normal distributions with $\sigma = 1$ (notice that the separation is shown on a log-scale).  From the plot, we can check that, even for the exponentially small separation (i.e., $|\mep_1-\mep_0| \approx 10^{-10}$), the boundary value remains at a practical level.

\begin{figure}
    \begin{center}
    \includegraphics[scale =  0.75]{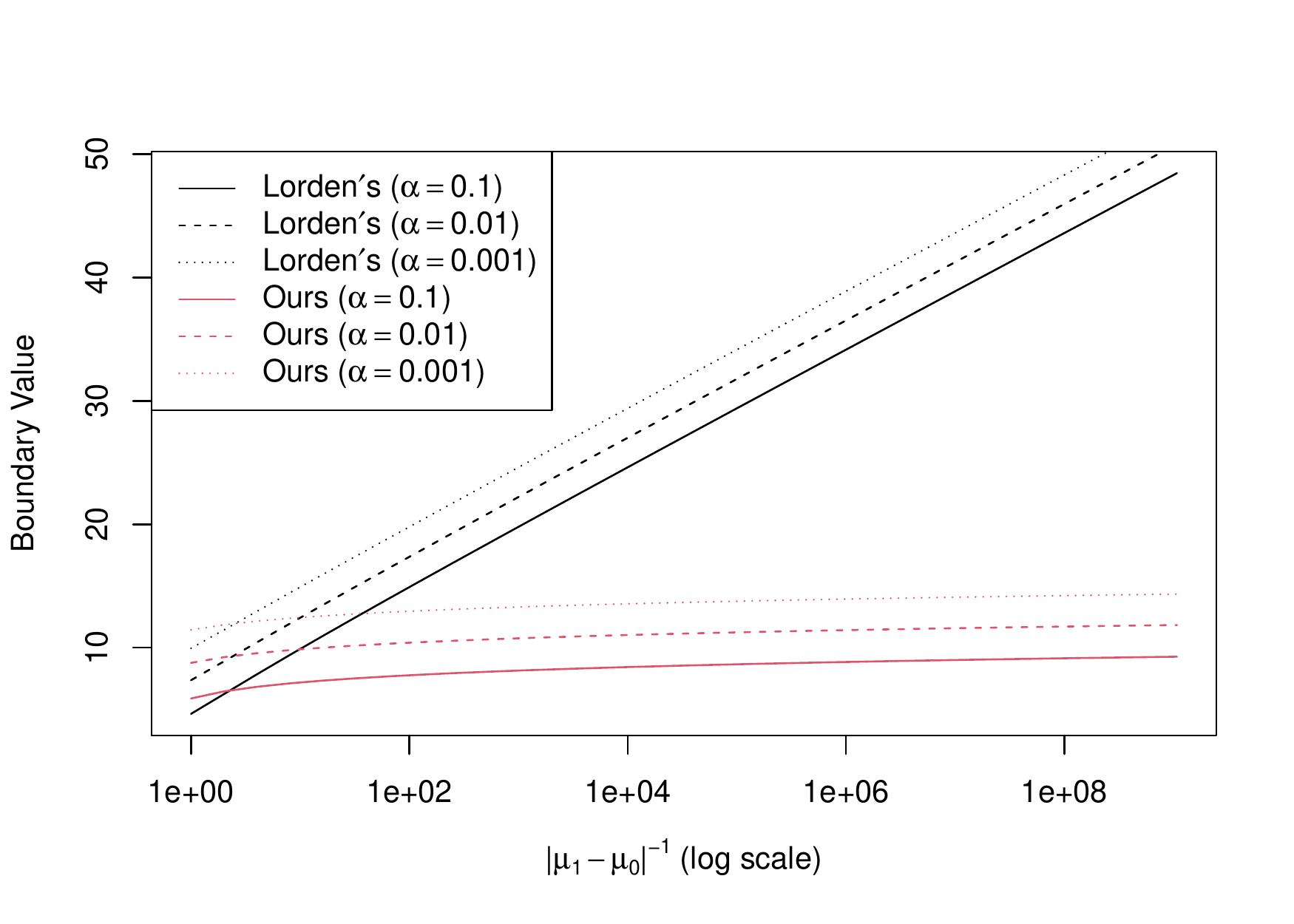}
    \end{center}
    \caption{Boundary values $g_\alpha^L(\mep_1,\mep_0)$ (black lines) and $g_\alpha(\mep_1,\mep_0)$ (red lines) for level $\alpha$ GLR tests based on Lorden's \eqref{eq::Lorden} upper bounds grow significantly faster than our \eqref{eq::const_bound} upper bounds. }
    \label{fig::const_boundary}
\end{figure}

The log-log dependency of the boundary value $g_\alpha(\mep_1,\mep_0)$ and the corresponding upper bound on the expected sample size in \cref{thm::upper_bound_const} are sharp since, in the worst case where $\mep = \mep_1$, the bounds on the expected sample size can be shown to be asymptotically tight in exponential families of distributions. In detail, for an exponential family of distributions parameterized  by the mean parameter $\mep \in \Mep$ and a given level $\alpha \in (0,1]$, let $N$ be a stopping time such that $\mathbb{P}_{\mep_0}(N < \infty) \leq \alpha$  for a fixed $\mep_0 \in \Mep$. Then, by  Farrell's theorem on open-ended tests \cite{farrell1964asymptotic, robbins1974expected}, the expected sample size under alternative sequences can be lower bounded as:
\begin{equation} \label{eq::farrell_org_form}
\begin{aligned}
\limsup_{\mep_1 \to \mep_0}\frac{(|\mep_1 - \mep_0|^2/2\sigma_{\mep_0}^2)\mathbb{E}_{\mep_1} N}{\log\log\left(1/|\mep_1 -\mep_0|\right)}
&\geq  \mathbb{P}_{\mep_0}\left(N = \infty\right) \\
&\geq 1 - \alpha,
\end{aligned}
\end{equation}
where $\sigma_{\mep_0}^2 := \nabla^2 B(\theta_{\mep_0})$ is the variance of the probability distribution at $\mep = \mep_0$. 
Using the Taylor series expansion of $B^*$ and the fact that $\KL(\mep_1,\mep_0) = D_{B^*}(\mep_1,\mep_0)$, we see 
\begin{align*}
\KL(\mep_1,\mep_0) &= D_{B^*}(\mep_1,\mep_0) \\
&= B^*(\mep_1) - B^*(\mep_0) - \nabla B^*(\mep_0) \left(\mep_1 - \mep_0\right) \\
&= \frac{1}{2}\nabla^2B^*(\mep_0)(\mep_1-\mep_0)^2 + o\left(|\mep_1-\mep_0|^2\right) \\
& = \frac{1}{2\sigma_{\mep_0}^2}(\mep_1-\mep_0)^2 + o\left(|\mep_1-\mep_0|^2\right),
\end{align*}
where the last equality is due to the fact that $\nabla^2 B^* (\mep_0) =\left[\nabla^2 B(\theta_{\mep_0})\right]^{-1}$. Farrell's lower bound in \eqref{eq::farrell_org_form} then yields
\begin{equation} \label{eq::farrell_lower}
\begin{aligned}
\limsup_{\mep_1 \to \mep_0}\frac{\KL(\mep_1,\mep_0)\mathbb{E}_{\mep_1} N}{\log\log\left(1/\KL(\mep_1,\mep_0)\right)} 
\geq  \mathbb{P}_{\mep_0}\left(N = \infty\right)
\geq 1 - \alpha.
\end{aligned}
\end{equation}
Therefore, we conclude that the log-log dependence in \eqref{eq::asym_g_alpha} and in the corresponding upper bound on the expected sample size in \eqref{eq::const_bound} cannot be avoided. 

The upper bound on the expected sample size in \cref{thm::upper_bound_const} is not fully adaptive to the underlying true but unknown mean parameter $\mep$ since the boundary value $g_\alpha(\mep_1,\mep_0)$ depends on the boundary of the alternative space $\mep_1$ instead of each alternative mean value $\mep$ itself. Although we can make the separation gap $|\mep_1 -\mep_0|$ between null and alternative spaces double-exponentially small while maintaining the threshold $g_\alpha(\mep_1,\mep_0)$ at a practical level, it is an interesting and practically relevant problem to design a SGLR-like test whose expected sample size is fully adaptive to the unknown alternative distribution. In the next subsection, we address this question using the boundary-crossing probability bound in \cref{thm::boundary_crossing_GLR}. 

\subsection{SGLR-like tests with no separation} \label{subSec::no_separation}
Below we consider the case of no separation (i.e. $\mep_0 = \mep_1$)  between the  null and alternative spaces in the one-sided test:
\begin{equation} \label{eq::one_sided_hypos_no_sep}
    H_0: \mep \leq \mep_0~~\text{vs}~~H_1: \mep > \mep_0.
\end{equation}
From the order-preserving property of the Bregman divergence, the log GLR-like statistic based on the first $\n$ observations can be written as $\n\D(\bar{X}_\n, \mep_0)\mathbbm{1}(\bar{X}_\n \geq \mep_0)$ and the upper bound on the boundary crossing probability from \cref{thm::boundary_crossing_GLR} is of the form
\begin{equation} \label{eq::cross_prob_no_sep}
\begin{aligned}
    \sup_{\mep \leq \mep_0}\mathbb{P}_{\mep} \left(\exists\n \geq 1: \bar{X}_\n \geq \mep_0,   \n\D(\bar{X}_\n, \mep_0) \geq g(\n)\right)
  \leq\inf_{\eta >1} \sum_{k=1}^{\infty}\exp\left\{-g(\eta^k) / \eta\right\}.
\end{aligned}
\end{equation}
Since the right hand side now involves an infinite sum, we cannot use a constant function $g$ for the boundary function. In fact, from the law of iterated logarithm, we know that the boundary function $g$ must increase at least at a log-log scale as the number of samples $n$ goes to infinity to get a nontrivial bound on the crossing probability. 

For simplicity, in order to build SGLR-like tests at level $\alpha \in (0,1]$, we will consider the boundary function 
\begin{equation}
    g_\alpha^c(n):= c\left[\log(1/\alpha) + 2 \log\left(\log_c c n\right)\right],
\end{equation}
where $c> 1$ is a fixed constant. Then, using   \eqref{eq::cross_prob_no_sep},
\begin{align}
        &\sup_{\mep \leq \mep_0}\mathbb{P}_{\mep} \left(\exists\n \geq 1: \bar{X}_\n \geq \mep_0,\n\D(\bar{X}_\n, \mep_0) \geq c\left[\log(1/\alpha) + 2 \log\left(\log_c c n\right)\right]\right) \nonumber\\
  & \quad \quad \quad \quad \quad \quad\leq \inf_{\eta >1} \sum_{k=1}^{\infty}\exp\left\{-\frac{c}{\eta}\log(1/\alpha)\right\} \frac{1}{\left(1 + k \log_c \eta\right)^{2c / \eta}} \nonumber\\
  & \quad \quad \quad \quad \quad \quad\leq \alpha\sum_{k=1}^\infty \frac{1}{(1+k)^2} ~ = \alpha \left(\frac{\pi^2}{6}-1\right) ~ \leq \alpha. \label{eq::GLR_test_ineq_no_sep}
\end{align}
Hence, a level $\alpha$ open-ended sequential testing procedure for  \eqref{eq::one_sided_hypos_no_sep} can be obtained based on the stopping time
    \begin{equation}
    \begin{aligned}
        N_{\GLR}(g_\alpha^c, \mep_0) := \inf\left\{\n \geq 1: \bar{X}_\n \geq \mep_0 ,\n\D(\bar{X}_\n, \mep_0) \geq c\left[\log(1/\alpha) + 2 \log\left(\log_c c n\right)\right] \right\}.
    \end{aligned}
    \end{equation}

Unlike the scenario analyzed in the previous subsection, however, Lorden's inequality is no longer applicable to the non-constant, non-linear boundary $g_\alpha^c$. Below we present novel, nonasymptotic probabilistic bounds on the stopping time $N_\GLR(g_\alpha^c, \mep_0)$ under the alternative, which we can then use to provide a high-probability bound on the sample size. We remark that the finite-sample feature of our bounds sets it apart from other results in the literature, which for the most part have relied on asymptotic arguments.

To derive the desired probabilistic bounds, we rely on a notion of divergence between two sub-$\psi_\Mep$ distributions having different mean parameters $\mep > \mep_0$ given by
\begin{equation}\label{eq:new.divergence}
    D_{\psi^*}^{*}(\mep, \mep_0) := \max\left\{\D(z^*, \mep_0), D_{\psi^*_{\mep}}(z^*,\mep)\right\},
\end{equation}
 where $z^* = z^*(\mep, \mep_0)$ is the minimizer of the function
\begin{equation}
    z \in M \mapsto f(z):= \max\left\{\D(z, \mep_0), D_{\psi^*_{\mep}}(z,\mep)\right\},
\end{equation}
  which is also equal to the unique solution of the equation $\D(z, \mep_0) = D_{\psi^*_{\mep}}(z,\mep)$ with respect to $z \in \Mep$. For the exponential family case, the divergence \eqref{eq:new.divergence} was introduced by Wong \cite{wong1968asymptotically} to characterize the asymptotic behavior of the SGLR test under the alternative. For a concrete example, if the two distribution belongs to a sub-Gaussian class with common variance parameter $\sigma^2$ then
\[
D_{\psi^*}^*(\mep,\mep_0) = \frac{1}{4}D_{\psi^*_{\mep_0}}(\mep, \mep_0) = \frac{1}{8\sigma^2}(\mep - \mep_0)^2,
\]
while for Bernoulli distributions, we instead have
\begin{align*}
\frac{1}{2}(\mep - \mep_0)^2 \leq D_{\psi^*}^*(\mep,\mep_0)\leq \mep \log\left(\frac{\mep}{\mep_0}\right) + (1-\mep) \log\left(\frac{1-\mep}{1-\mep_0}\right).
\end{align*}
Note that the $D_{\psi^*}^*$ divergence satisfies $D_{\psi^*}^*(\mep_0, \mep_0) = 0$ and $D_{\psi^*}^*(\mep, \mep_0) = D_{\psi^*}^*(\mep_0, \mep)$ for any $\mep, \mep_0 \in \Mep$.
Using $D_{\psi^*}^*$, we can formulate the following high probability bound on the stopping time $N_{\GLR}(g_\alpha^c, \mep_0)$ under the alternative.
\begin{theorem}\label{thm::T_high-sub-psi}
For any fixed  $\mep > \mep_0 \in \Mep$,  $\delta \in (0,1]$ and  $c >1$, define $T_{\high}(\delta)$ by
\begin{equation} \label{eq::T_high-sub-psi}
\begin{aligned}
    T_{\high}(\delta)
    := \inf\left\{t  \geq 1: \frac{c\left[ \log(1/\delta)+ 2\log \left(\log_c ct\right)\right]}{D_{\psi^*}^*(\mep, \mep_0)} \leq t \right\}.
\end{aligned}
\end{equation}
Then, for any $\delta \in (0,\alpha)$, we have
\begin{equation}  \label{eq::T_high_bound-sub-psi}
   \mathbb{P}_{\mep}\left(N_{\GLR}(g_\alpha^c, \mep_0) \leq T_{\high}(\delta) \right)\geq 1-\delta.
\end{equation}
Also, $T_{\high}(\delta)$ can be upper bounded by $\max(1,A)$, where
\begin{equation} \label{eq::T_high_bound-sub-psi_explicit}
\begin{aligned}
   A= \frac{2c}{D_\mep^*} \log(1/\delta)  
   +\frac{2c}{D_\mep^*}\log \left( 2\log_{c}\left(\frac{2c^2}{\log c}\right) + 2 \left[\log_{c}\left(1/D^*_\mep \right)\right]_+ \right),
\end{aligned}
\end{equation}
with $D_\mep^* := D_{\psi^*}^*(\mep, \mep_0)$.
\end{theorem}

\noindent The proof of \cref{thm::T_high-sub-psi} can be found in Appendix~\ref{appen::proofs_of_thms}.
As a consequence, we have the following upper bound on the expected sample size:
\begin{equation} \label{eq::T_ave_bound_no_sep}
\begin{aligned}
   \mathbb{E}_{\mep}\left[N_{\GLR}(g_\alpha^c, \mep_0)\right] \leq 1+\frac{2c}{D_\mep^*} \log(1/\alpha)  
  +\frac{2c}{D_\mep^*}\log \left( 2\log_{c}\left(\frac{2c^{2.5}}{\log c}\right) + 2 \left[\log_{c}\left(1/D^*_\mep \right)\right]_+ \right),
\end{aligned}
\end{equation}

Note that both bounds,  in probability and in expectation, are fully adaptive to the true alternative parameter $\mep$ via the divergence $D^*_\mep:= \D^*(\mep, \mep_0)$. For exponential family distributions, the above upper bound matches Farrell's optimal rate given in \eqref{eq::farrell_org_form} up to a constant factor under the moderate confidence regime in which $\mep \to \mep_0$ and $\alpha$ is fixed.

\section{Confidence sequences via GLR-like statistics}
\label{sec::CS}

Previously, we discussed how to build open-ended sequential testing procedure for the one-sided testing problem based on the general upper bound on the boundary crossing probability  given in \cref{thm::boundary_crossing_GLR}. Relying on the same bound, we can construct confidence sequences for $\mu$. A level $\alpha$ confidence sequence for $\mu$ is a sequence of sets $\{\CI_\n\}_{\n \in \mathbb{N}}$ such that 
\begin{equation}
    \mathbb{P}_\mep \left(\forall \n \geq 1 : \mep \in \CI_\n \right) \geq 1-\alpha,~~\forall \mep \in \Mep.
\end{equation}
For any chosen boundary $g$ and mapping $\mep_0 \mapsto \mep_1(\mep_0)$ with $\mep_1 > \mep_0$, each $\CI_\n$ is defined by
\begin{equation}
\begin{aligned}
    \CI_\n := \left\{\mep_0  \in \Mep : \log\GLR_\n(\mep_1, \mep_0) < g(\n),~~\text{or}~\inf_{\eta >1} \sum_{k=1}^{1\vee K_\eta}\exp\left\{-g(\eta^k) / \eta\right\} > \alpha \right\},
\end{aligned}
\end{equation}
where $K_\eta$ is given in \eqref{eq::K_number}. Of course, different choices of the boundary function $g$ and of the  mapping $\mep_0 \mapsto \mep_1(\mep_0)$ result in confidence sequences of different shapes. 
For any such choice, the above construction of $\CI_\n$ is guaranteed to yield a valid level $\alpha$ confidence sequence since
\begin{align*}
    &\mathbb{P}_\mep \left(\forall \n \geq 1: \mep \in \CI_\n \right)  
    = 1- \mathbb{P}_\mep \left(\exists \n\geq 1 : \mep \notin \CI_\n \right) \\
    & \geq 1 - \mathbb{P}_\mep \left(\exists \n \geq 1: \log\GLR_\n(\mep_1(\mep),\mep) \geq g(\n)\right)
    \mathbbm{1}\left(\inf_{\eta >1} \sum_{k=1}^{1\vee K_\eta}\exp\left\{-g(\eta^k) / \eta\right\} \leq \alpha\right) \\
    & \geq 1- \left(\inf_{\eta >1} \sum_{k=1}^{1\vee K_\eta}\exp\left\{-g(\eta^k) / \eta\right\} \right) \mathbbm{1}\left(\inf_{\eta >1} \sum_{k=1}^{1\vee K_\eta}\exp\left\{-g(\eta^k) / \eta\right\} \leq \alpha\right) \\
    &\geq 1 - \alpha,
\end{align*}
for each $\mep \in \Mep$ where the second-to-last inequality comes from the bound in \cref{thm::boundary_crossing_GLR}.  In this section, we present a slightly generalized version of \cref{eq::const_bound} to obtain confidence sequences that are valid over $\mathbb{N}$ and uniformly close to the Chernoff bound on chosen finite time intervals.

\subsection{Confidence sequence uniformly close to the Chernoff bound on  finite time intervals}
Recall that for each $\n$ and $\alpha \in (0,1]$, the (pointwise) Chernoff bound for sub-$\psi_\Mep$ distributions takes the form
\begin{equation} \label{eq::chernoff_org}
    \mathbb{P}_{\mep}\left(\bar{X}_\n \geq \mep_\epsilon\right) \leq e^{-\n D_{\psi_{\mep}^*}(\mep_\epsilon, \mep) },
\end{equation}
for any fixed $\mep_\epsilon > \mep \in \Mep$. Since the mapping $z \mapsto D_{\psi_{\mep}^*}(z, \mep)$ is increasing on $[\mep, \infty) \cap \Mep$, by letting $g:= \n D_{\psi_{\mep}^*}(\mep_\epsilon, \mep)$, the Chernoff bound can be written as
\begin{equation} \label{eq::chernoff_deviation}
    \mathbb{P}_{\mep}\left(\bar{X}_\n \geq \mep,~\n D_{\psi_{\mep}^*}(\bar{X}_\n, \mep) \geq g\right) \leq e^{-g},
\end{equation}
for any fixed $\mep \in \Mep$.

As discussed in \cref{sec::GLR-like test}, from the law of iterated logarithm, we know that the above bound cannot be extended to an anytime-valid bound. That is, there is no time-independent boundary  having a nontrivial upper bound on the  boundary-crossing probability
\begin{equation}
   \mathbb{P}_{\mep}\left(\exists \n \geq 1: \bar{X}_\n \geq \mep,~\n D_{\psi_{\mep}^*}(\bar{X}_\n, \mep) \geq g \right),
\end{equation}
for all $\mep \in \Mep$. However, in virtually all practical cases, there exists an effective limit on the duration of the experiment and therefore on the sample size. And conversely, we often impose a minimum sample size requirement to, e.g.,  meet prespecified accuracy requirements, or to stabilize the experiment or to take account seasonality effects.
As a result, in many situations, practitioners may find it desirable to use anytime confidence sequences that are uniformly close to the optimal pointwise Chernoff bound on a prespecified finite time interval $[\nmin, \nmax]$. We next describe such a guarantee.

\begin{theorem} \label{thm::CS_simple}
For any $g > 0$ and $\mep_0 \in \Mep$, define
 \begin{equation}
     \n_0:=  \inf\left\{\n \geq 1:  \sup_{z\in \Mep, z > \mep_0 }\D(z ,\mep_0) \geq g/ \n \right\}.
 \end{equation}
For any $\n_0\leq\nmin<\nmax \in \mathbb{N}$, 
let $\mep_1 < \mep_2$ be  solutions on $(\mep_0, \infty) \cap \Mep$ of the equations
\begin{equation} \label{eq::mep_def}
\frac{g}{\nmax} = \D(\mep_1,\mep_0) \quad \text{and} \quad \frac{g}{\nmin} = \D(\mep_2,\mep_0)
\end{equation}
 respectively.
Then, the boundary crossing probability for the likelihood ratio process is such that
\begin{equation} \label{eq::const_bound_general}
\begin{aligned}
        &\mathbb{P}_{\mep_0} \left(\exists \n \geq 1:   \sup_{z \in (\mep_1, \mep_2)}\log \left(\LR_\n(z,\mep_0)\vee 1\right) \geq g\right)\\
        &\leq  e^{-g}\mathbbm{1}(\nmin > \n_0) + \inf_{\eta >1} \left\lceil \log_\eta \left(\frac{\nmax}{\nmin}\right)\right\rceil e^{-g / \eta}.
  \end{aligned}
\end{equation}
\end{theorem}

\begin{remark} \label{rmk::how_to_choose_alternative}
If $\n_0 = \nmin < \nmax$, we can get rid of $\mep_2$ in the supremum in the boundary crossing probability in \eqref{eq::const_bound_general}. In this case, the above inequality is reduced to the usual constant boundary crossing inequality from \cref{subSec::well-separated} with a specifically chosen alternative $\mep_1$ given by $\frac{g}{\nmax} = \D(\mep_1,\mep_0)$. From the perspective of the duality between sequential tests and confidence sequences, this observation tells us that if we have a practical upper limit of the sequential testing procedure described in \cref{subSec::well-separated}, it is natural to use the value $\mep_1$ given by $\frac{g_\alpha}{\nmax} = \D(\mep_1,\mep_0)$ as the boundary of the alternative space where $\alpha \in (0,1]$ is the target level of the test and $g_\alpha$ is the constant boundary which makes the right hand side of \eqref{eq::const_bound_general} equal to $\alpha$. 
\end{remark}

\begin{figure}
    \begin{center}
    \includegraphics[scale =  0.75]{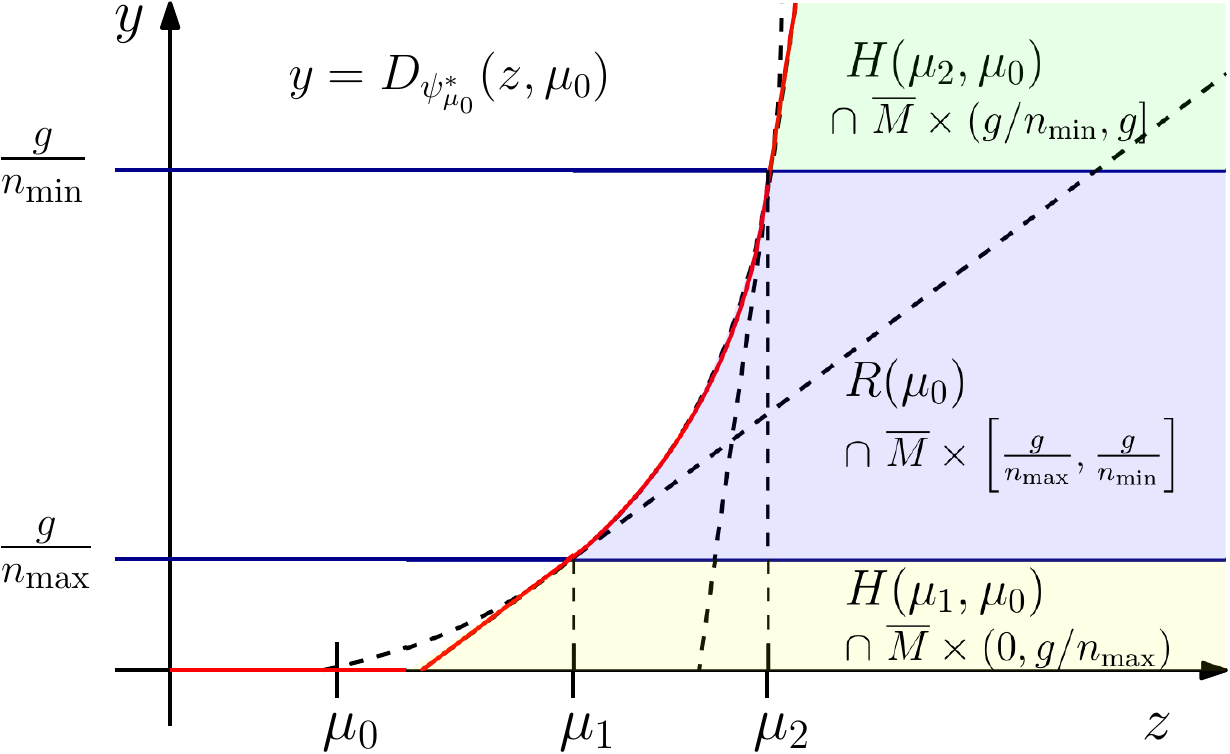}
    \end{center}
    \caption{Illustration of the boundary-crossing events and related regions $H(\mep_1, \mep_0), H(\mep_2, \mep_0)$ and $R(\mep_0)$  in \eqref{eq::GLR_like_CS_explicit}. } 
    \label{fig::GLR-CS}
\end{figure}
Note that the event on the left hand side can be written as 
\begin{equation} \label{eq::GLR_like_CS_explicit}
\begin{aligned}
&\left\{\exists n \geq 1 : \left(\bar{X}_\n, \frac{g}{\n}\right) \in H(\mep_2, \mep_0)\cap R(\mep_0)\cap H(\mep_1, \mep_0) \right\} \\
&= \left\{\exists  n \in [1, \nmin) : \left(\bar{X}_\n, \frac{g}{\n}\right) \in H(\mep_2, \mep_0) \right\}\\
    &~~~~\cup \left\{\exists n \in [\nmin, \nmax]: \left(\bar{X}_\n, \frac{g}{\n}\right) \in R(\mep_0) \right\} \\
    &~~~~~~~~~~~~\cup \left\{\exists n \in (\nmax, \infty) :   \left(\bar{X}_\n, \frac{g}{\n}\right) \in H(\mep_1, \mep_0)\right\},
\end{aligned}
\end{equation}
where the set  $R(\mep_0)$ is defined by
\begin{equation}
\begin{aligned}
    R(\mep_0) 
    := \left\{(z,y) \in \overline{\Mep} \times [0,\infty) : y \leq D_{\psi^*_{\mep_0}}(z,\mep_0), z \geq \mep_0\right\},
\end{aligned}
\end{equation}
and $H(\mep_1,\mep_0)$ and $H(\mep_2,\mep_0)$ are halfspaces contained in and tangent to $ R(\mep_0)$ at $(\mep_1, g / \nmax)$ and $(\mep_2, g / \nmin)$, respectively. See \cref{fig::GLR-CS} for an illustration of $H(\mep_1, \mep_0), H(\mep_2,\mep_0)$ and $R(\mep_0)$.

In particular, on the time interval $[\nmin, \nmax]$, the boundary-crossing event can be written as
\begin{equation}
\left\{\exists n \in [\nmin, \nmax]:\n \bar{X}_\n \geq \mep_0, \D(\bar{X}_\n, \mep_0) \geq g \right\}.
\end{equation}
Therefore, we can check that the confidence sequence based on \cref{thm::CS_simple} is anytime-valid and uniformly close to the pointwise Chernoff bound on the time interval $[\nmin, \nmax]$.

More generally, we may have a sequence of time intervals $\{[\nmin^{(k)}, \nmax^{(k)}]\}_{k=1}^K$ on which we want the confidence sequence to be uniformly closer to the pointwise Chernoff bound. In this case, we can extend the bound in \cref{thm::CS_simple} as follows.

\begin{corollary} \label{cor::CS_multi}
Choose a sequence of boundary values $\{g^{(k)}\}_{k=1}^K$, $\mep_0 \in \Mep$ and sequence of  time intervals $\{[\nmin^{(k)}, \nmax^{(k)}]\}_{k=1}^K$ with $\n_0:= \nmax^{(0)} \leq \nmin^{(1)}< \nmax^{(1)} \leq \cdots \leq \nmin^{(K)}< \nmax^{(K)}$. For each $k \in [K]$, let $\mep_1^{(k)} < \mep_2^{(k)}$ be solutions  on $(\mep_0, \infty) \cap \Mep$ to the equations:
\begin{equation}
    \frac{g^{(k)}}{\nmax^{(k)}} = \D(\mep_1^{(k)},\mep_0),~\text{~ and ~} ~\frac{g^{(k)}}{\nmin^{(k)}} = \D(\mep_2^{(k)},\mep_0).
\end{equation}
 Then, the boundary crossing probability can be bounded as 
\begin{equation} \label{eq::const_bound_general_multi}
\begin{aligned}
&\mathbb{P}_{\mep_0} \left(\exists \n \geq 1, \exists k \in [K]:  \sup_{z \in (\mep_1^{(k)}, \mep_2^{(k)})}\log \left(\LR_\n(z,\mep_0)\vee 1\right) \geq g^{(k)} \right)\\
&\leq  \sum_{k=1}^K \left[e^{-g^{(k)}}\left[\mathbbm{1}(\nmin^{(k)} >\nmax^{(k-1)})\right]+  \inf_{\eta >1} \left\lceil \log_\eta \left(\nmax^{(k)} / \nmin^{(k)}\right)\right\rceil e^{-g^{(k)} / \eta}\right].
\end{aligned}
\end{equation}
\end{corollary}

\begin{remark}
If $\nmax^{(k-1)} = \nmin^{(k)}$ for each $k \in [K]$, the inequality \eqref{eq::const_bound_general_multi} can be viewed as a piecewise constant boundary crossing probability of the form 
\begin{equation}
\begin{aligned}
    \mathbb{P}_{\mep_0} \left(\exists \n \geq 1:  \sup_{z \geq \mep_1^{(K)}}\log \left(\LR_\n(z,\mep_0)\vee 1\right) \geq g_c(\n) \right) \leq  \sum_{k=1}^K \left[\inf_{\eta >1} \left\lceil \log_\eta \left(\frac{\nmax^{(k)}}{\nmin^{(k)}}\right)\right\rceil e^{-g^{(k)} / \eta}\right],
\end{aligned}
\end{equation}
where $g_c(\n)$ is a piecewise constant function defined by
\begin{equation} 
\begin{aligned}
g_c(\n):= &\sum_{k=1}^K \min\left\{g^{(k-1)}, g^{(k)}\right\} \mathbbm{1}\left(n = \nmin^{(k)}\right)\\
&\qquad+g^{(k)}\mathbbm{1}\left(n \in (\nmin^{(k)} , \nmax^{(k)})\right)+ \min\left\{g^{(k)}, g^{(k+1)}\right\} \mathbbm{1}\left(n = \nmax^{(k)}\right).
\end{aligned}
\end{equation}
Here, we set $g^{(0)} := g^{(1)}$ and $g^{(K+1)}:= g^{(K)}$. Note that if all boundary values $\{g^{(k)}\}_{k=1}^K$ are equal to each other, then the above inequality recovers the constant boundary-crossing inequality in \cref{subSec::well-separated} with a specifically chosen alternative $\mep_1^{(K)}$. In fact, the right hand side of the above inequality provides a sharper bound, since it allows us to choose different $\eta$ for each $k$.
\end{remark}

In what follows, we focus on the single time interval case in \cref{thm::CS_simple} for simplicity, but the arguments can be straightforwardly extended to the case of multiple time intervals in \cref{cor::CS_multi}.  

To convert the upper bound on the boundary-crossing probability in \eqref{eq::const_bound_general}, for a given $\alpha \in (0,1]$, let $g_\alpha$ be the constant boundary which makes the upper bound in \eqref{eq::const_bound_general} less than or equal to $\alpha$. The boundary value $g_\alpha$ can be efficiently computed  using the fact 
\begin{equation} \label{eq::CS_bound_equiv_expression}
 \inf_{\eta >1} \left\lceil \log_\eta \left(\frac{\nmax}{\nmin}\right)\right\rceil e^{-g / \eta} = \inf_{k\in\mathbb{N}} k \exp\left\{-g\left(\frac{\nmin}{\nmax}\right)^{1/k}\right\}.
\end{equation}
As shown in the previous section, in the moderate confidence regime where $\alpha$ is fixed and $\nmax / \nmin \to \infty$, we can check that $g_\alpha = O\left(\log\log\left(\nmax/\nmin\right)\right)$. Therefore, in practice, we can still compute a boundary value $g_\alpha$ even for an exponential large $\nmax/\nmin$. Also, since  $g_\alpha$ scales with respect to the ratio between two ends points of the time interval instead of its length, the confidence sequence can be designed to be uniformly close to the pointwise Chernoff bound on a time interval of exponentially large length. This property is of course especially useful when designing a large-scale experimentation. 

For a given $g_\alpha$, the corresponding confidence sequence is
\begin{equation} \label{eq::CS_detailed}
    \CI_n := \begin{cases}  \left\{\mep_0 \in \Mep : \begin{array}{l}
         L_2(\bar{X}_\n -\mep_2) 
        < \frac{g_\alpha}{\n} - \frac{g_\alpha}{\nmin}
  \end{array}  
 \right\}  &\mbox{if } \n \in [1,\nmin), \\
 \left\{\mep_0 \in \Mep: 
\begin{array}{l}
 \D(\bar{X}_\n, \mep_0)
 < \frac{g_\alpha}{\n}~~\text{or}~~\bar{X} < \mep_0
 \end{array}
 \right\} &\mbox{if } \n \in [\nmin, \nmax], \\
\left\{\mep_0 \in \Mep:
    \begin{array}{l}
         L_1(\bar{X}_\n -\mep_1) 
        < \frac{g_\alpha}{\n} - \frac{g_\alpha}{\nmax}
  \end{array}  
  \right\} &\mbox{if } \n \in (\nmax, \infty),
  \end{cases}
\end{equation}
where $L_i := \nabla_z \D(z, \mep_0)\mid_{z = \mep_i}$ for each $i = 1, 2$. Recall that both $\mep_1$ and $\mep_2$ depend on $\mep_0$ via \eqref{eq::mep_def}. 

As a sanity check, we can verify that $[\bar{X}_\n, \infty) \subset \CI_\n$ for each $\n \in \mathbb{N}$, since  $\nabla_z \D(z, \mep_0) > 0 $ for any $z > \mep_0$. We can also check that the width of the confidence sequence (the distance from the left end point to the sample mean) does not shrink to zero even if $\n \to \infty$ which implies that the confidence sequence can be loose outside of the target time interval $[\nmin, \nmax]$. Therefore, in practice, we recommend using the time interval of large enough size to cover the entire intended duration of  experimentation.  

Note that, for additive sub-$\psi$ classes, $\mep_1$ and $\mep_2$ have the following simple relationships:
\begin{align}
\mep_1 &= \mep_0 + \psi_+^{*-1}\left(\frac{g_\alpha}{\nmax}\right) := \mep_0+ \Delta_1\\
\mep_2 &= \mep_0 + \psi_+^{*-1}\left(\frac{g_\alpha}{\nmin}\right) := \mep_0+\Delta_2,
\end{align}
where $\psi_+^{*-1}$ is the inverse function of $z \mapsto \psi^*(z)\mathbbm{1}(z \geq 0)$. We can also check $L_i = \nabla \psi^*(\Delta_i)$ for each $i = 1, 2$. 

Using this observation, we can derive confidence sequence for additive sub-$\psi$ families in explicit form:
\begin{equation} \label{eq::CS_additive}
    \CI_n := \begin{cases}  \left(\bar{X}_\n - \Delta_2 -   L_2^{-1}\left(\frac{g_\alpha}{\n} - \frac{g_\alpha}{\nmin}\right), \infty \right) &\mbox{if } \n \in [1,\nmin), \\
 \left( \bar{X}_\n - \psi_+^{*-1}\left(\frac{g_\alpha}{\n}\right) , \infty \right) &\mbox{if }  \n \in [\nmin, \nmax], \\
 \left(\bar{X}_\n   - \Delta_1-  L_1^{-1}\left(\frac{ g_\alpha}{\n} - \frac{ g_\alpha}{\nmax}\right), \infty \right) &\mbox{if } \n \in (\nmax, \infty).
  \end{cases}
  \end{equation} 
 For sub-Gaussian distributions with a parameter $\sigma^2$, this reduces to
\begin{equation} \label{eq::CS_sub_G}
    \CI_n := \begin{cases}  \left(\bar{X}_\n - \sigma\sqrt{\frac{2g_\alpha}{\nmin}}\left[ \frac{1}{2}\left(\frac{\nmin}{\n}+1\right)\right], \infty\right) &\mbox{if } \n \in [1,\nmin), \\
 \left( \bar{X}_\n - \sigma\sqrt{\frac{2g_\alpha}{\n}}, \infty\right) &\mbox{if }  \n \in [\nmin, \nmax], \\
 \left(\bar{X}_\n - \sigma\sqrt{\frac{2g_\alpha}{\nmax}}\left[  \frac{1}{2}\left(\frac{\nmax}{\n}+1\right)\right], \infty\right) &\mbox{if } \n \in (\nmax, \infty).
  \end{cases}
  \end{equation}

\begin{remark} \label{rmk::CS_for_varying_mean} 
For additive sub-$\psi$ distributions, the sequence of intervals defined in \eqref{eq::CS_additive} can be applied to the time-varying mean case described in \cref{remark::thm1_for_process}. That is, for any chosen $\alpha \in (0,1]$, the confidence sequence in \eqref{eq::CS_additive} satisfies 
\begin{equation}
    \mathbb{P}\left(\exists \n \geq 1 : \bar{\mep}_\n \in \CI_\n\right) \geq 1-\alpha,
\end{equation}
 where $(\bar{\mep}_\n)_{n \geq 1}$ is the sequence of running averages of conditional means, defined by 
\begin{equation}
\bar{\mep}_\n := \frac{1}{\n}\sum_{i=1}^\n \mathbb{E}\left[X_i \mid \mathcal{F}_{i-1}\right].
\end{equation}
\end{remark}

In general, there is no closed form expression for the confidence sequence. However, if a sub-$\psi_\Mep$ family of distributions has order-preserving Bregman divergences, then for any given data, the mapping $\mep_0 \mapsto \D(\bar{X}_\n, \mep_0)\mathbbm{1}\left(\bar{X}_\n \geq \mep_0\right)$ is nonincreasing. Therefore, on the target time interval $[\nmin, \nmax]$, the confidence sequence given in \eqref{eq::CS_detailed} is an open-interval  and it can be efficiently computed by binary search. 

Outside of the target time interval, however, the confidence sequence is not necessarily an open interval. To avoid this potentially undesirable feature, below we introduce a sufficient condition under which we can guarantee that an EF-like sub-$B$ family of distributions admits a confidence sequence consisting of open intervals. 
\begin{proposition} \label{prop::ef-like-convex}
For a given EF-like sub-$B$ family of distributions, suppose $\nabla B$ is a convex function. Then, for any data, the mapping $\mep_0 \mapsto \LR_\n(\mep_1, \mep_0)$ is nonincreasing on $(-\infty, \bar{X}_\n] \cap \Mep$ where $\mep_1$ is a function of $\mep_0$, and any $d>0$,  as defined by the solution of the equation
\begin{equation}
    \D(\mep_1,\mep_0)\mathbbm{1}(\mep_1 \geq  \mep_0) = d.
\end{equation}
Consequently, the corresponding confidence sequence in \eqref{eq::CS_detailed}  is an open-interval for each $\n$.
\end{proposition}

The proof of \cref{prop::ef-like-convex} can be found in Appendix~\ref{appen::proof_of_props}. For example, Poisson, Exponential, and Negative binomial (with a known number of failures) distributions are sub-classes of EF-like sub-$B$ families satisfying the condition in \cref{prop::ef-like-convex}, and thus we can efficiently compute confidence sequences by using binary search. On the other hand, Bernoulli distribution does not satisfies the condition, and thus we need to use a grid-search to compute the confidence sequence on the outside of the target time interval. 

\subsection{Tighter confidence sequences via discrete mixtures}

Confidence sequences based on nonnegative mixture of martingales have been extensively studied \cite{wald1945sequential,darling1967confidence, lai1976confidence, victor2007pseudo, kaufmann2018mixture, howard2021time}. However, as discussed in \cite{howard2021time}, different choices of mixing methods yield different boundaries, and each confidence sequence is typically tighter and looser on some time intervals than others, so there is no time-uniformly dominating one. 

From this point of view, it is an interesting question  to choose a proper mixing method to obtain a confidence sequence with a desired shape, satisfying application-specific constraints. This subsection explains how we can use confidence sequences based on GLR-like statistics in the previous subsection to design discrete mixture-based confidence sequences that are almost uniformly close to the Chernoff bound on finite target time intervals.

The confidence sequences in the previous subsection are built on the constant boundary-crossing probability in \cref{thm::CS_simple}, which is based on a GLR-like curve-crossing time. 
In this subsection, we will demonstrate that, for any GLR-like curve-crossing time, we can also build a discrete mixture of martingales such that the corresponding crossing time is always smaller than or equal to the GLR-like one.
Therefore, we can always construct a tighter confidence sequence by using the discrete mixture-crossing time compared to the GLR-like one. Consequently, the shape of the obtained discrete mixture martingale based confidence sequence is dominated by the GLR-like based one, whose overall shape can be tailored to the specific application at hand.

To elaborate, for any $\alpha \in (0,1]$ and target time interval $[\nmin, \nmax]$, let $g_\alpha$ be a positive value such that the upper bound in \eqref{eq::const_bound_general} is less than or equal to $\alpha$. If $\n_0 < \nmin = \nmax$ then both GLR-like and discrete mixture-crossing events are equal to the line-crossing event given by
\begin{equation}
    \left\{\exists n \geq 1 : \begin{array}{l}
         \D(\mep_1,\mep_0)
         + \nabla_z \D(z, \mep_0)\mid_{z = \mep_1}(\bar{X}_\n -\mep_1)  
 \geq \frac{g_\alpha}{\n}
  \end{array}   \right\},
\end{equation}
which is equal to the pointwise Chernoff bound at $\n = \nmin = \nmax$, and thus $g_\alpha$ can be chosen as $\log(1/\alpha)$. Therefore, in the rest of this subsection, we only consider the nontrivial case  $\n_0 \leq \nmin < \nmax$, where $g_\alpha$ is a positive value such that 
\begin{equation}
     e^{-g_\alpha}\mathbbm{1}(\nmin > \n_0) + \inf_{\eta >1} \left\lceil \log_\eta \left(\frac{\nmax}{\nmin}\right)\right\rceil e^{-g_\alpha / \eta} ~ \leq ~ \alpha.
\end{equation}
For a fixed $g_\alpha$, let $\eta_\alpha > 1$ be the value attaining the infimum in the LHS of the above inequality. From the equivalent expression of $\inf_{\eta >1} \left\lceil \log_\eta \left(\frac{\nmax}{\nmin}\right)\right\rceil e^{-g_\alpha / \eta}$ in \eqref{eq::CS_bound_equiv_expression}, we have $\eta_\alpha = (\nmax / \nmin)^{1/K_\alpha}$ where 
\begin{equation}
    K_\alpha := \argmin_{k\in\mathbb{N}} k \exp\left\{-g_\alpha\left(\frac{\nmin}{\nmax}\right)^{1/k}\right\}.
\end{equation}
Next, for each $k = 0, 1, \dots, K_\alpha$, let $z_k$ be the function of $\mep_0$ defined as the solution of the following equation:
\begin{equation}
    \D(z_k, \mep_0) = \frac{g_\alpha}{\nmin \eta_\alpha^k},~~z_k > \mep_0.
\end{equation}
Finally, for given $n$ samples, define a nonnegative random variable $M_\n(\mep_0; \alpha)$ by
\begin{equation} 
\begin{aligned}
    M_\n(\mep_0; \alpha) := e^{-g_\alpha}  \mathbbm{1}\left(\nmin >  \n_0\right) \LR_\n\left(z_0(\mep_0), \mep_0\right) +  e^{-g_\alpha / \eta_\alpha}\sum_{k=1}^{K_\alpha}\LR_\n\left(z_k(\mep_0), \mep_0\right).
\end{aligned}
\end{equation}
Above, each $\LR_\n(z_k(\mep_0), \mep_0)$ is the LR-like statistic for $H_0: \mep = \mep_0$ vs $H_1: \mep = z_k(\mep_0)$, given by
\begin{equation}
\begin{aligned}
    \LR_\n(z_k(\mep_0), \mep_0) = \exp\left\{\n \left[ \frac{g_\alpha}{\nmin\eta_\alpha^k}+  L_k\left[\bar{X}_\n - z_k(\mep_0)\right]\right]\right\},
\end{aligned}
\end{equation}
where $L_k := \nabla \D\left(z, \mep_0\right)\mid_{z = z_k(\mep_0)} $ for each $k = 0, 1, \dots, K_\alpha$. In particular, for additive sub-$\psi$ family, each $\LR_\n(z_k(\mep_0), \mep_0)$ can be expressed as
\begin{equation}
\begin{aligned}
     \LR_\n(z_k(\mep_0), \mep_0) = \exp\left\{\n \left[ \frac{g_\alpha}{\nmin\eta_\alpha^k} + L_k \left[\bar{X}_\n - \mep_0- \Delta_k\right]\right]\right\},
\end{aligned}
\end{equation}
where $\Delta_k := \psi_+^{*-1}\left(\frac{g_\alpha}{\nmin\eta_\alpha^k}\right)$ and $L_k := \nabla \psi^*\left(\Delta_k)\right)$.

Now, let $M_0(\mep_0; \alpha) :=  e^{-g_\alpha}\mathbbm{1}(\nmin > \n_0) + K_\alpha e^{-g_\alpha / \eta_\alpha} \in (0, \alpha]$. Then, under any sub-$\psi_{\mep_0}$ distribution, $\{M_\n(\mep_0; \alpha) / M_0(\mep_0 ; \alpha) \}_{\n \geq 0}$ is a nonnegative supermartingale with respect to the natural filtration. Therefore, by letting
\[
 \CI_\n^M := \left\{\mep_0 \in \Mep: M_\n(\mep_0; \alpha) / M_0(\mep_0; \alpha) < 1/\alpha\right\},   
\]
we have a discrete mixture confidence sequence which is uniformly tighter than the GLR-like one:
\begin{corollary} \label{cor::dis_mixture}
For any $\alpha \in (0,1]$, the sequence of intervals $\left\{\CI_\n^M\right\}_{\n \in \mathbb{N}}$ is a valid level $\alpha$ confidence sequence satisfying
\begin{equation} \label{eq::dis_mixture_CI_coverage}
    \mathbb{P}_\mep \left(\forall \n : \mep \in \CI_\n^M \right) \geq 1-\alpha,~~\forall \mep \in \Mep.
\end{equation}
Furthermore,  we have 
$\CI_\n^M \subset \CI_\n$ for each $\n \in \mathbb{N}$
where $\{\CI_\n\}_{\n \in \mathbb{N}}$ is the confidence sequence based on GLR-like statistics in \eqref{eq::CS_detailed} which is uniformly close to the pointwise Chernoff bound on the target time interval $[\nmin, \nmax]$. 
\end{corollary}

Since $g_\alpha$ and $K_\alpha$ do not depend on $\mep_0$, using \cref{prop::ef-like-convex} we can check that each $\CI_\n^M$ is an open interval for every additive sub-$\psi$ family and EF-like sub-$B$ family of distributions with convex $\nabla B$. Therefore, we can efficiently compute each confidence interval by binary-search. However, for  general sub-$\psi_\Mep$ family of distributions, we may need to rely on grid-search methods to compute each confidence interval.

\subsection{Examples of sub-Gaussian confidence sequences} \label{subSec::experiment}

To illustrate the practicality of confidence sequences based on GLR-like statistics and corresponding discrete mixtures, in \cref{fig::CI_compare}, we calculate ratios of widths of confidence intervals to the pointwise and asymptotically valid normal confidence intervals based on the central limit theorem.\footnote{The R code to reproduce all the plots and simulation results of the paper is available on the repository \url{https://github.com/shinjaehyeok/SGLRT_paper}.}
 Here, we use the sub-Gaussian with $\sigma = 1$ for simplicity and set $\alpha = 0.025$. 
\begin{figure}
    \begin{center}
    \includegraphics[scale =  0.75]{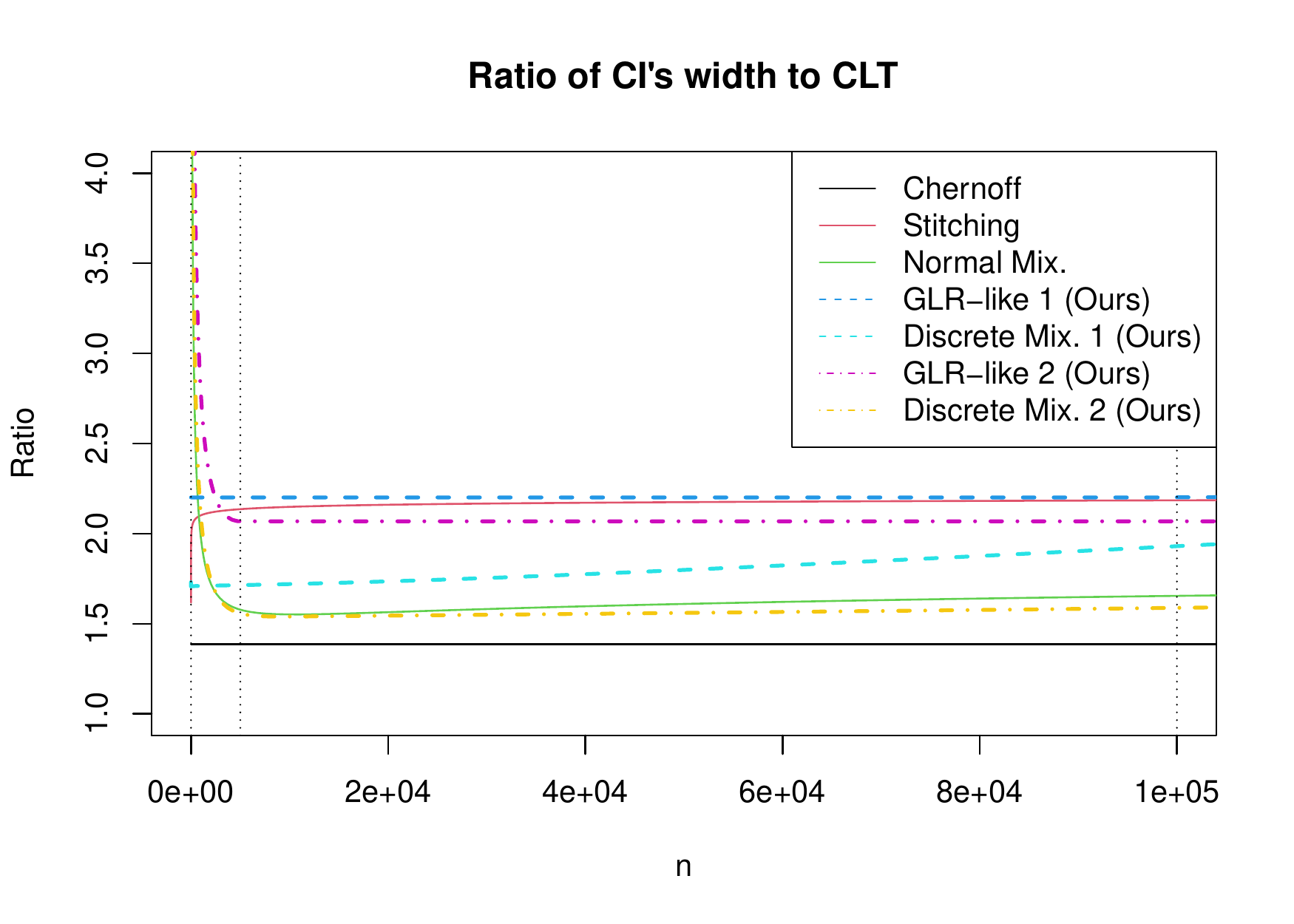}
    \end{center}
    \caption{Ratio of widths of the confidence intervals to the pointwise and asymptotically valid normal confidence intervals based on the central limit theorem. The black solid line corresponds to the pointwise and nonasymptotically valid Chernoff bound. Red and green solid lines come from the stitching and normal mixture method in \cite{robbins1970statistical,howard2021time}.  The rest of lines are based on our GLR-like confidence sequences for sub-Gaussian distributions in \eqref{eq::CS_sub_G} and their discrete mixture counterparts with different choices of target time intervals ($[1, 10^5]$ and $[5 \times 10^3, 4 \times 10^5]$). See \cref{subSec::experiment} for the details of these confidence sequences.}
    \label{fig::CI_compare}
\end{figure}
The black solid line corresponds to the pointwise, nonasymptotically valid Chernoff bound. Red and green solid lines come from the stitching and normal mixture methods in \cite{robbins1970statistical, howard2021time} where each confidence intervals for stitching $\CI_{\n}^{\mathrm{ST}}$ and normal mixture method $\CI_\n^{\mathrm{NM}}$ is given by
\begin{align*}
    \CI_{\n}^{\mathrm{ST}} &:= \left(\bar{X}_\n-  \frac{1.7}{\sqrt{\n}}\sqrt{\log\log(2\n) + 0.72 \log\left(\frac{5.2}{\alpha}\right)}, \infty \right) \\
    \CI_{\n}^{\mathrm{NM}} &:= \left(\bar{X}_\n-  \sqrt{2 \left(\frac{1}{\n} + \frac{\rho}{\n^2}\right) \log\left(\frac{1}{2\alpha} \sqrt{\frac{\n + \rho}{\rho} + 1}\right)}, \infty \right),
\end{align*}
where we set $\rho = 1260$ by following the setting in Figure 9 of \cite{howard2021time}. Note that the original normal mixture confidence interval in \cite{robbins1970statistical} did not have a closed form expression. In this subsection, we use the explicit closed form upper bound in \cite{howard2021time} for simplicity. The rest of lines are based on our GLR-like confidence sequences for sub-Gaussian distributions in \eqref{eq::CS_sub_G} and their discrete mixture counterparts. For  `GLR-like 1' and `Discrete Mixture 1' lines, we set $[\nmin, \nmax] = [1, 10^5]$. For `GLR-like 2' and 'Discrete Mixture 2', we use $[\nmin, \nmax] = [5\times10^3, 4 \times 10^5]$. Vertical dotted lines are corresponding to the lines $\n = 1, 5\times10^3, 10^5$.  We can check both GLR-like confidence sequences are uniformly close to the pointwise Chernoff bound on their target time intervals and, as \cref{cor::dis_mixture} tells, each discrete mixture counterpart has uniformly smaller widths of confidence intervals on its target time interval. 

\section{Discussion}\label{sec:discussion}
We have presented nonasymptotic analyses of sequential tests and confidence sequences based on GLR-like statistics, which can be viewed as a nonparametric generalization of the GLR statistic. Our main contribution is to provide a unified nonasymptotic framework for the sub-$\psi_\Mep$ family of distributions by leveraging a novel geometrical interpretation of GLR-like statistics and of the corresponding time-uniform concentration inequalities. In \cref{tab::summary}, we provide a technical summary of our results by displaying the boundary crossing probabilities for the sequential tests and confidence sequences developed in the paper along with the corresponding expected sample sizes for the moderate confidence regime (fixed $\alpha$, $\mep, \mep_1 \to \mep_0$). In the table, we set $D_1 := \D(\mep_1,\mep_0)$, $D_\mep := \D(\mep, \mep_0)$ and similarly for $D_1^* := \D^*(\mep_1,\mep_0)$ and $D_\mep^*:= \D^*(\mep_1,\mep_0)$.

\begin{table*}[ht]
	\small
\centering
\setlength\tabcolsep{1.5mm}
\def\arraystretch{2}
\caption{Summary of boundary crossing probabilities and expected sample sizes.\\ ($D_1 := \D(\mep_1,\mep_0)$, $D_\mep := \D(\mep, \mep_0)$, $D_1^* := \D^*(\mep_1,\mep_0)$ and $D_\mep^*:= \D^*(\mep_1,\mep_0)$)}
\label{tab::summary}
\begin{tabular}{ccc}
\hline
Task & Boundary crossing probability  & Sample size\\\addlinespace[2mm] \hline\hline 	\addlinespace[2mm]
\makecell{Test ($\mep_1 > \mep_0$) \\ (Sec.~\ref{subSec::well-separated})} & \makecell{$\begin{aligned}
        &\sup_{\mep \leq \mep_0}\mathbb{P}_{\mep} \left(\exists \n \geq 1:   \log\GLR_\n(\mep_1, \mep_0) \geq g\right) \\
  &\leq e^{-g} \mathbbm{1}\left( D_1 \geq g\right) + \inf_{\eta >1} \left\lceil \log_\eta \left(\frac{g}{D_1}\right)\right\rceil e^{-g / \eta} \mathbbm{1}\left( D_1 < g\right)
 \end{aligned}$}    & $\begin{aligned}  O\left(\frac{\log\log\left(1/ D_1\right) }{D_\mep}\right) \end{aligned}$  \\ 	\addlinespace[2mm] \hline 	\addlinespace[2mm]
\makecell{Test ($\mep_1 = \mep_0$) \\ (Sec.~\ref{subSec::no_separation})} & \makecell{$\begin{aligned}
        &    \sup_{\mep \leq \mep_0}\mathbb{P}_{\mep} \left(\exists\n \geq 1: \bar{X}_\n \geq \mep_0,   \n\D(\bar{X}_\n, \mep_0) \geq g(\n)\right)
   \\
  &\leq \inf_{\eta >1} \sum_{k=1}^{\infty}\exp\left\{-g(\eta^k) / \eta\right\} 
 \end{aligned}$ }   & $\begin{aligned} O\left(\frac{\log\log\left(1/ D_\mep^*\right) }{D_\mep^*}\right) \end{aligned}$  \\ 	\addlinespace[2mm] \hline 	\addlinespace[2mm]
 
 \makecell{Confidence sequence \\ (Sec.~\ref{sec::CS})} & \makecell{$\begin{aligned}
     &\mathbb{P}_{\mep_0} \left(\exists \n \geq 1:   \sup_{z \in (\mep_1, \mep_2)}\log \left(\LR_\n(z,\mep_0)\vee 1\right) \geq g\right)\\
        &\leq  e^{-g}\mathbbm{1}(\nmin > \n_0) + \inf_{\eta >1} \left\lceil \log_\eta \left(\frac{\nmax}{\nmin}\right)\right\rceil e^{-g / \eta}. 
 \end{aligned}$ }   & not applicable\\ 	\addlinespace[2mm] \hline 	\addlinespace[2mm]
\end{tabular}
\end{table*}

There remain several important open problems. First, although we have mainly focused on the case where each sample in a data stream $X_1, X_2, \dots,$ is an independent random variable with the same mean, as we discussed in \cref{remark::thm1_for_process}~and~\ref{rmk::CS_for_varying_mean}, some of the presented analyses can be naturally extended to a real-valued process $(X_i)_{i \in \mathbb{N}}$ adapted to a filtration $(\mathcal{F}_i)_{i \in \{0\}\cup\mathbb{N}}$ where the conditional expectation $\mep_i := \mathbb{E}[X_i \mid \mathcal{F}_{i-1}]$ can vary over time. However, except for additive sub-$\psi$ classes, it is unclear how to extend the expected sample size analysis of SGLR-like tests and the construction of the  confidence sequences to the time-varying mean case. For example, the nonasymptotic upper bound on the expected sample size of SGLR-like test in \cref{thm::upper_bound_const} is only applicable to the i.i.d. random variables. Also, except for the additive sub-$\psi$ case, the confidence sequence in \cref{sec::CS} is not applicable in time-varying settings. Generalizing these analyses to a more flexible nonparametric setting is an important open direction.

Second, the SGLR-like tests derived here and the corresponding confidence sequences are applicable only to the univariate case in which the underlying data stream is a real-valued sequence. It is a natural to inquire whether one can generalize the SGLR-like tests and confidence sequences to multiple sources based multivariate data streams. If we have multiple independent univariate data stream then there is a simple method to combine upper bounds on boundary crossing probabilities for univariate data stream in \cref{thm::boundary_crossing_GLR} into a multivariate one. To be specific, let $K$ be the number of independent univariate data streams. For each $a \in [K]$, let $\{X_{N_a(t)}^a\}_{t \geq 0}$ be a sequence of independent observations from a sub-$\psi_{\mep^a}$ distribution where $N_a(t)$ is the number of sample from the $a$-th distribution at time $t \geq 0$. We assume $N_a(t) \geq 1$ for each $a\in [K]$ and $t \geq 0$. Define $\GLR_t^a(\mep_1^a,  \mep_0^a)$ be the GLR-like statistic based on $N_a(t)$ samples from the $a$-th distribution up to time $t$. Then, from \cref{thm::boundary_crossing_GLR}, we can find a boundary function $g_\alpha^a$ such that the following inequality holds for all $\alpha \in (0,1]$ and for each $a \in [K]$:
\begin{equation} \label{eq::bound_for_each_a}
\begin{aligned}
    &\mathbb{P}_0\left(\exists t \geq 0: \log\GLR_t^a(\mep_1^a,  \mep_0^a) \geq g_\alpha^a\left(N_a(t)\right)\right)
    \leq \alpha,
\end{aligned}
\end{equation}
where $\mathbb{P}_0$ is a null distribution of $K$ independent data streams. Suppose each boundary function can be decomposed as $g_\alpha^a(n) = f^a(n) + h^a(\alpha)$ where $f^a$ is a nonnegative function on $\mathbb{N}$ which does not depends on $\alpha$ and $h^a$ is a nonnegative and nonincreasing function on $(0,1]$ such that  $\lim_{\alpha \to 0} h^a(\alpha)= \infty$. Then, for each $\epsilon > 0$, we have the following upper bound on the boundary-crossing probability for the multiple sources of univariate data streams:
\begin{equation} \label{eq::multiple_source_bound}
\begin{aligned}
    \mathbb{P}_0\left(\exists t \geq 0: \sum_{a=1}^K\log\GLR_t^a(\mep_1^a,  \mep_0^a)\geq \sum_{a=1}^K  f^a\left(N_a(t)\right)  + \epsilon\right)
    \leq \mathbb{P}_0\left(\sum_{a=1}^K h^a(U_a) \geq \epsilon\right), 
\end{aligned}
\end{equation}
where each $U_a$ is an independent $\text{Uniform}[0,1]$ random variable. Since the right hand side only depends on $K$ independent uniformly distributed random variables, we can evaluate the probability of right hand side via a simple Monte Carlo simulation. See Appendix~\ref{appen::proof_of_props} for the derivation of the above inequality. It is an interesting open question whether we can derive similar upper bound for general multiple sources of multivariate data streams.

	\bibliographystyle{ieeetr}
	\bibliography{SGLRT_ref}
	
	
	\newpage
	\appendix
	
	\section{Properties of sub-\texorpdfstring{$\psi_\Mep$}{psi} family distributions} \label{appen:facts_on_sub_psi}
Here, we review some properties of sub-$\psi_\Mep$ family distributions used in the main text. 

\subsection{Three different forms of LR-like statistics}
Recall that the LR-like statistic based on first $\n$ samples which is defined by
\begin{equation} \label{eq::LR_like_form1}
       \LR_{\n}(\mep_1, \mep_0) := \exp\left\{\n\left[ \lambda_1\bar{X}_{\n} - \psi_{\mep_0}(\lambda_1)\right]\right\},
\end{equation}
where $\lambda_1 = \nabla \psi_{\mep_0}^*(\mep_1)$. The LR-like statistic can be rewritten as follows:
\begin{align}
       \LR_{\n}(\mep_1, \mep_0) &:= \exp\left\{\n\left[ \lambda_1\bar{X}_{\n} - \psi_{\mep_0}(\lambda_1)\right]\right\} \nonumber\\
       &=\exp\left\{\n\left[ \lambda_1\left(\bar{X}_{\n} - \mep_1\right) + \lambda_1\mep_1 - \psi_{\mep_0}(\lambda_1)\right]\right\} \nonumber\\
       &=\exp\left\{\n\left[ \lambda_1\left(\bar{X}_{\n} - \mep_1\right) +  \psi_{\mep_0}^*(\mep_1)\right]\right\} \nonumber\\
      &= \exp\left\{\n\left[ \D\left(\mep_1, \mep_0\right) + \nabla_z \D\left(z, \mep_0\right)|_{z= \mep_1} \left(\bar{X}_{\n} - \mep_1\right)\right]\right\} \label{eq::LR_like_form2}\\ 
       &=\exp\left\{\n\left[\psi_{\mep_0}^*(\bar{X}_\n) -\left\{\psi_{\mep_0}^*(\bar{X}_\n) - \psi_{\mep_0}^*(\mep_1) - \nabla \psi_{\mep_0}^*(\mep_1)\left(\bar{X}_{\n} - \mep_1\right)\right\} \right]\right\}  \nonumber\\
       &=\exp\left\{\n\left[\D\left(\bar{X}_{\n}, \mep_0 \right) - \D \left(\bar{X}_{\n}, \mep_1 \right) \right]\right\}\label{eq::LR_like_form3}, 
\end{align}
where the third equality comes from the fact $\lambda_1 \mep_1 = \psi_{\mep_0}(\lambda_1) + \psi_{\mep_0}^*(\mep_1)$ which is a consequence of the Fenchel–Young inequality with equality holding from the choice of $\lambda_1$. The fourth and sixth equalities are based on the relationship $\D(z, \mep_0) = \psi_{\mep_0}^*(z)$ for any $z \in \Mep$ (which in turn comes from the assumption $\nabla \psi_{\mep_0}(0) = \mep_0$, or equivalently $\nabla \psi_{\mep_0}^*(\mep_0) = 0$). In sum, we have three different forms of LR-like statistics in \eqref{eq::LR_like_form1}, \eqref{eq::LR_like_form2} and \eqref{eq::LR_like_form3}. From the equivalence between Bregman and KL divergences for exponential family distributions, we can also check the LR statistic for exponential family distributions also has similar three different forms.

\subsection{Bregman divergences of additive sub-\texorpdfstring{$\psi$}{psi} and EF-like sub-\texorpdfstring{$B$}{B} distributions}

Recall that for an additive sub-$\psi$ family, each $\psi_{\mep}$ function is defined for $\lambda \in \mathbb{R}$ as
\begin{equation}
  \psi_0(\lambda) := \psi(\lambda) \text{ and } \psi_{\mep}(\lambda) = \psi_0(\lambda)+ \lambda \mep.
 \end{equation}
 In this case, we have
 \begin{align*}
     \D(\mep_1, \mep_0) &= \psi_{\mep_0}^*(\mep_1) \\
     & = \sup_{\lambda} \lambda \mep_1 - \psi_{\mep_0}(\lambda) \\
     & = \sup_{\lambda} \lambda (\mep_1-\mep_0) - \psi(\lambda) \\
     & = \psi^*(\mep_1 - \mep_0).
 \end{align*}
For an EF-like sub-$B$ family of distributions, $\psi_{\mep}$ is defined by 
\begin{equation}
    \psi_{\mep}(\lambda)  = B(\lambda + \theta_\mep) - B(\theta_\mep),~\forall\lambda \in \mathbb{R}.
\end{equation}
In this case, for each $z \in \Mep$, $\psi_{\mep_0}^*(z)$ can be written as follows:
\begin{align*}
    \psi_{\mep_0}^*(z) & = \sup_{\lambda} \lambda z - \psi_{\mep_0}(\lambda) \\
    & = \sup_{\lambda} \lambda z - B(\lambda + \theta_{\mep_0}) + B(\theta_{\mep_0}) \\
    & = \sup_{\lambda} (\lambda + \theta_{\mep_0}) z - B(\lambda + \theta_{\mep_0}) + B(\theta_{\mep_0}) -\theta_{\mep_0} z\\
    & = B^*(z) + B(\theta_{\mep_0}) - \theta_{\mep_0} z.
\end{align*}
Therefore, for any $z_1, z_0 \in \Mep$, we have
\begin{align*}
    \D(z_1, z_0) &= \psi_{\mep_0}^*(z_1) - \psi_{\mep_0}^*(z_0) - \nabla \psi_{\mep_0}^* (z_0)(z_1- z_0) \\
    & = B^*(z_1) - B^*(z_0) - \theta_{\mep_0}(z_1 - z_0) - \left(\nabla B^*(z_0) - \theta_{\mep_0}\right)(z_1 - z_0) \\
    & = B^*(z_1) - B^*(z_0)- \nabla B^*(z_0)(z_1 - z_0) \\
    & = D_{B^*}(z_1, z_0),
\end{align*}
which proves the first equality in \eqref{eq::LR-like_LR_equiv}. To prove the second equality, note that the KL divergence between $p_{\mep_1}$ and $p_{\mep_0}$ is defined by
\begin{equation}
        \KL(\mep_1, \mep_0) := - \mathbb{E}_{\mep_0} \log \frac{p_{\mep_1}(X)}{p_{\mep_0}(X)},
\end{equation}
which can be written as follows:
\begin{align*}
    \KL(\mep_1, \mep_0) &= - \mathbb{E}_{\mep_0} \log \frac{p_{\mep_1}(X)}{p_{\mep_0}(X)} \\
    & = B(\theta_{\mep_0}) - B(\theta_{\mep_1}) - \mep_0 \left(\theta_{\mep_0} - \theta_{\mep_1}\right) \\
    & = B(\theta_{\mep_0}) - B(\theta_{\mep_1}) - \nabla B(\theta_{\mep_0}) \left(\theta_{\mep_0} - \theta_{\mep_1}\right) \\
    & = D_{B}(\theta_{\mep_0}, \theta_{\mep_1}) \\
    & = D_{B^*}(\mep_1, \mep_0),
\end{align*}
where $\theta_{\mep_i} := \left(\nabla B\right)^{-1}(\mep_i)$ for each $i = 0, 1$.  Therefore, the KL divergence is equal to the Bregman divergence with respect to $B^*$ which proves the second equality in \eqref{eq::LR-like_LR_equiv}.

\section{Proofs of theorems and related statements}
\label{appen::proofs_of_thms}

\subsection{Proof of Theorem~\ref{thm::boundary_crossing_GLR}}
\label{appen::proof_of_thm_1}

We first review linear-boundary crossing probability and its generalization for discrete mixture supermartingales on which the curved boundary crossing probability in \cref{thm::boundary_crossing_GLR} is based. Theses probabilities for sub-$\psi_{\Mep}$ family distributions were studied comprehensively by \cite{howard2020time}. Consider the following simple null and alternative hypotheses:
\[
H_0: \mep = \mep_0~~\text{vs}~~H_1: \mep = \mep_1,
\]
for some fixed $\mep_1> \mep_0 \in \Mep$. 

Let $\LR_n(\mep_1,\mep_0)$ be the corresponding LR-like statistic based on first $n$ samples defined in \eqref{eq::LR-like_def}. From the definition of sub-$\psi_{\Mep}$ distributions, it can be checked that $\{\LR_n(\mep_1,\mep_0)\}_{\n \in \mathbb{N}}$ is a super-martingale with respect to the natural filtration under the null with $\mathbb{E}_{\mep_0}\left[\LR_{1}(\mep_1,\mep_0)\right] = 1$. From the Ville's maximal
inequality for nonnegative supermartingales~\cite{ville_etude_1939}, we have
\begin{equation} \label{eq::line-crossing_bound}
    \mathbb{P}_{\mep_0}\left(\exists \n \geq 1 : \LR_{\n}(\mep_1,\mep_0) \geq \frac{1}{\alpha}\right) \leq \alpha,
\end{equation}
 for any $\alpha \in (0,1]$. From alternative expressions of LR-like statistics in \eqref{eq::LR-like_diff_form} and \eqref{eq::LR-like_tangent_line_form}, the Ville's maximal inequality above implies that
\begin{align}
   \alpha & \geq \mathbb{P}_{\mep_0}\left(\exists \n\geq 1 : \LR_{\n}(\mep_1,\mep_0) \geq \frac{1}{\alpha}\right)\\
     &= \mathbb{P}_{\mep_0}\left(\exists \n\geq 1 : D_{\psi^*_{\mep_0}}\left(\bar{X}_{\n}, \mep_0\right) - D_{\psi^*_{\mep_0}}\left(\bar{X}_{\n}, \mep_1\right) \geq \frac{1}{n}\log(1/\alpha)\right) \\
     & =\mathbb{P}_{\mep_0}\left(\exists \n\geq 1 :  \D\left(\mep_1, \mep_0\right) + \nabla_z \D\left(z, \mep_0\right)|_{z= \mep_1} \left(\bar{X}_{\n} - \mep_1\right) \geq \frac{1}{n}\log(1/\alpha)\right) \\ 
     &= \mathbb{P}_{\mep_0}\left(\exists \n\geq 1 : \left(\bar{X}_{\n}, \frac{\log(1/\alpha)}{\n}\right) \in H(\mep_1,\mep_0)\right), \label{eq::line_crossing}
\end{align}
where $H(\mep_1,\mep_0)$ is the half space contained in and tangent to $ R(\mep_0)$ at $(\mep_1, D_{\psi^*_{\mep_0}}(\mep_1,\mep_0))$, and $R(\mep_0)$ is defined by
\begin{equation}
    R(\mep_0) := \left\{(z,y) \in \overline{\Mep} \times [0,\infty) : y \leq D_{\psi^*_{\mep_0}}(z,\mep_0), z \geq \mep_0\right\}.
\end{equation}
See \cref{fig::GLR_illustration} for an illustration of $H(\mep_1,\mep_0)$ and $R(\mu_0)$. The event in \eqref{eq::line_crossing} happens if there exists a time $\n$ such that the stochastic process $\left(\bar{X}_{\n}, \frac{\log(1/\alpha)}{\n}\right)_{\n \in \mathbb{N}}$ crosses the linear boundary of $H(\mep_1,\mep_0)$. From this observation, we call it a line-crossing event and refer its probability as a line-crossing probability which is less than $1/\alpha$ by construction.  

\begin{figure}
    \begin{center}
    \includegraphics[scale =  0.75]{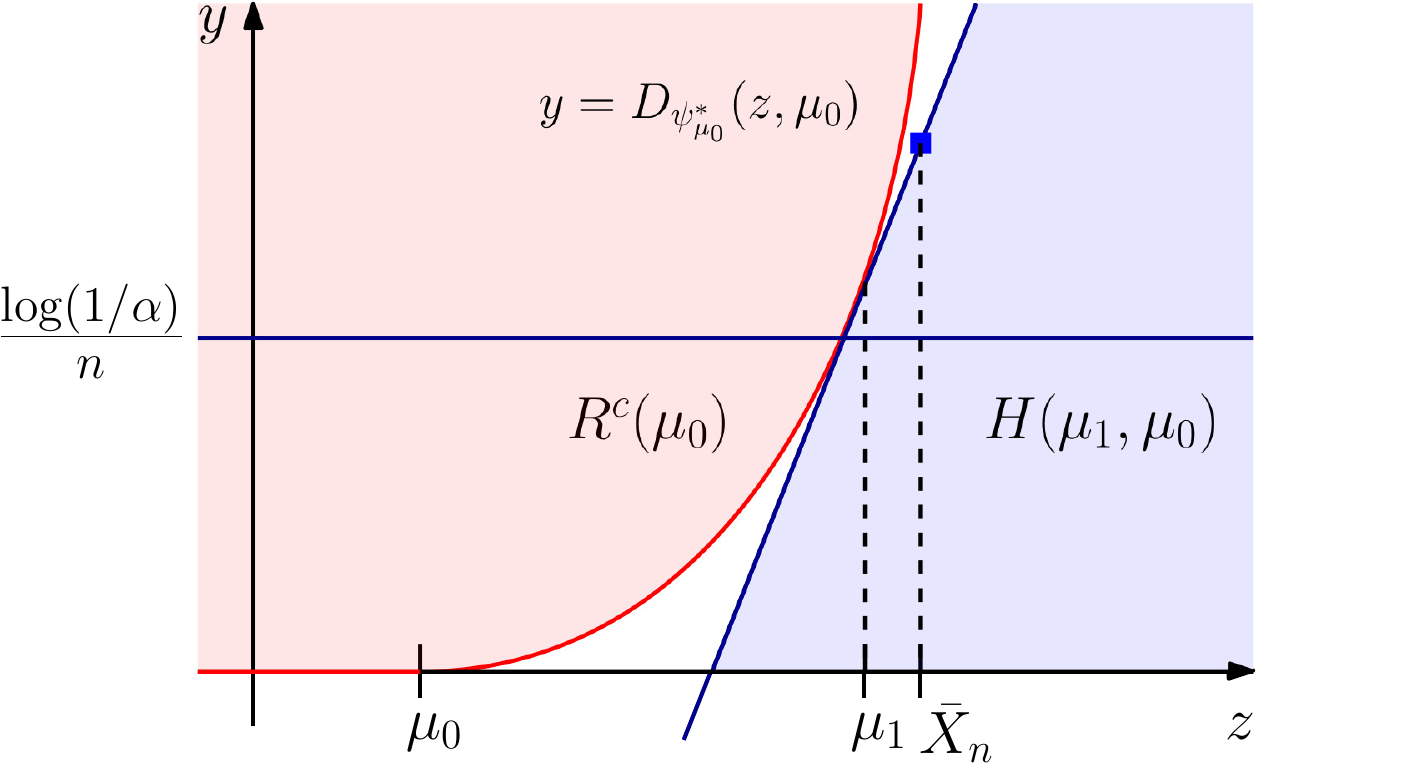}
    \end{center}
    \caption{  Geometric interpretation of the line-crossing event. The red line corresponds to the Bregman divergence from $\mep_0$ and the blue line is tangent to the divergence function at $\mep_1$. The blue square dot on the line corresponds to the LR-like statistics with respect to the observed sample mean $\bar{X}_\n$. In this illustration, the line-crossing event happened since at time $\n$, we have $\left(\bar{X}_\n, \frac{\log(1/\alpha)}{n}\right) \in H(\mep_1,\mep_0)$.} 
    \label{fig::GLR_illustration}
\end{figure}

More generally, let $\{h(j)\}_{j=1}^K$ be a nonnegative sequence such that $\sum_{j=1}^K h^{-1}(j) <\infty$ for some $K \in \mathbb{N} \cup \{\infty\}$. Then, for any given sequence of mean parameters $\{z_k\}_{j=1}^K \in \Mep^K$, $\left\{\sum_{j=1}^K h(j) \LR_n(z_k,\mep_0) \right\}_{\n \geq 1}$ is a supermartingale under $\mathbb{P}_{\mep_0}$ with $\mathbb{E}_{\mep_0}\left[\sum_{j=1}^K h(j) \LR_1(z_k,\mep_0)\right] = \sum_{j=1}^K h^{-1}(j)$. Hence, the Ville's inequality implies the following inequality:
\begin{equation}  \label{eq::discrete_mixture_inequality}
    \mathbb{P}_{\mep_0}\left(\exists \n\geq 1 : \sum_{j=1}^K h(j)^{-1} \LR_n(z_k,\mep_0) \geq \frac{1}{\alpha} \right) \leq \alpha \sum_{j=1}^K h^{-1}(j). 
\end{equation}

Based on these two underlying inequalities, 
we prove the following generalized version of \cref{thm::boundary_crossing_GLR} in which we allow the minimum sample size $\nmin$ of the boundary crossing event to be larger than one, and this can be used to build a SGLR test starting with a pre-collected batch of data. \cref{thm::boundary_crossing_GLR} is directly followed by setting $\nmin = 1$.

\begin{theorem} \label{thm::boundary_crossing_GLR_with_nmin}
  For a given minimum sample size $\nmin \geq 1$, suppose the boundary function $g: [\nmin,\infty) \to [0,\infty)$ satisfies the following two conditions:
    \begin{enumerate}
        \item $g$ is nonnegative and nondecreasing;
        \item the mapping $t \mapsto g(t) / t$ is nonincreasing on $[\nmin,\infty)$ and $\lim_{t\to \infty} g(t) /t = 0$. 
    \end{enumerate}
    Then, for any sub-$\psi_{\Mep}$ family of distributions with  order-preserving  Bregman divergences, the boundary crossing probability under the null can be bounded as follows:
    \begin{align} \label{eq::GLR_bound_thm1_with_nmin}
    &\sup_{\mep \leq \mep_0}\mathbb{P}_{\mep} \left(\exists\n \geq \nmin:   \log\GLR_\n(\mep_1, \mep_0) \geq g(\n)\right) \\
  &\quad \quad \quad \quad \quad \quad \leq \begin{cases} e^{-g(\nminbar)} & \mbox{if } \D(\mep_1,\mep_0) \geq g(\nminbar) / \nminbar \\
\inf_{\eta >1} \sum_{k=1}^{K_\eta}\exp\left\{-g(\nminbar\eta^k) / \eta\right\} & \mbox{otherwise }
  \end{cases} 
  \end{align}
where  $\nminbar:=  \inf\left\{\n \geq \nmin:  \sup_{z\in \Mep, z > \mep_0 }\D(z ,\mep_0) \geq g(\n)/ \n \right\}$ and $K_{\eta} \in \mathbb{N}\cup \{0, \infty\}$ is defined for any $\eta >1$ by
\begin{equation} \label{eq::K_number_with_nmin}
K_{\eta} := K_{\eta}(g;\mep_1,\mep_0):= \inf\left\{k \in \{0\}\cup \mathbb{N}: \D(\mep_1,\mep_0) \geq \frac{g(\nminbar\eta^k)}{\nminbar\eta^{k}}\right\}.
\end{equation}
 \end{theorem}

The proof is based on a construction of a sequence of simple stopping times based on linear boundary crossing events. To be specific, for each $\eta > 1$ and $k \in [K_\eta]$, let $h_k = h_k(\eta) := \exp\left\{-g(\nminbar \eta^k) / \eta\right\}$ . Then, set $z_{K_\eta} := \mep_1$ and define $z_k = z_k(\eta)$ as the solution to the  equations
\begin{equation}
    \D(z_k,\mep_0)= \frac{g (\nminbar \eta^{k})}{\nminbar \eta^k},~~z_k > \mep_0,
\end{equation}
for each $k \in \left[K_\eta - 1\right]$. From the definition of $K_\eta$, we can check that $z_1 > z_2 > \cdots > z_{K_\eta} = \mep_1$.

Now, for each $h_k$ and $z_k$, let $N_{\LR}(h_k; z_k, \mep_0)$ be the stopping time based on  the LR-like test for the simple hypothesis testing problem $H_0: \mep =  \mep_0$ versus $H_1 : \mep = z_k$ at level $h_k$, which is given by
\begin{equation}
         N_\LR = N_\LR\left(h_k; z_k, \mep_0\right) := \inf\left\{\n \geq 1: \LR_\n (z_k, \mep_0)\geq h_k^{-1}\right\} \label{eq::lr_st}.
\end{equation}
It can be checked that $N_{\LR}$ induces a valid level $h_k$ test and satisfies
\begin{equation}
 \mathbb{P}_{\mep_0}\left(N_{\LR} <\infty\right) = \mathbb{P}_{\mep_0}\left(\exists\n \geq 1: \LR_\n (z_k, \mep_0)\geq h_k^{-1}\right)\leq h_k.   
\end{equation}

Now, we are ready to present our main technical lemma upon which \cref{thm::boundary_crossing_GLR_with_nmin} is based.
\begin{lemma} \label{lem::relationships}
Under the conditions of \cref{thm::boundary_crossing_GLR}, define the GLR-like curve-crossing time 
    \begin{equation}
        N_\GLR := \inf\left\{\n \geq \nmin: \log \GLR_\n (\mep_1, \mep_0)\geq g(\n)\right\}. \label{eq::glr_st}
    \end{equation}
If $\D(\mep_1, \mep_0) \geq g(\nminbar)/\nminbar$ then we have $N_\LR\left(e^{-g(\nminbar)}; \mep_1, \mep_0\right) \leq N_\GLR$. Otherwise, define the following stopping times corresponding to different boundary crossing events with respect to a fixed $\eta > 1$:
\begin{enumerate} 
    \item Minimum of lines-crossing time:
            \begin{equation}
        \begin{aligned}
        N_{\ML}(\eta) &:= \inf\left\{\n \geq 1: \max_{k\in[K_\eta]} h_k \LR_\n (z_k, \mep_0) \geq 1\right\} \label{eq::max_linear_st} \\
        & = \min_{k\in [K_\eta]} N_{\LR}(h_k; z_k, \mep_0) 
        \end{aligned}
    \end{equation}
    \item Discrete mixture-crossing time:
     \begin{equation} \label{eq::dm_st}
         N_{\DM}(\eta):= \inf\left\{\n \geq 1 : \sum_{k=1}^{K_\eta} h_k \LR_\n (z_k, \mep_0)  \geq 1 \right\}.
     \end{equation}
\end{enumerate}
Then, for each $\eta > 1$ , it holds that
\begin{equation} \label{eq::relation}
    N_{\DM}(\eta) \leq N_{\ML}(\eta)  \leq N_{\GLR}.
\end{equation}
\end{lemma}

 We provide a constructive proof of \cref{lem::relationships}  based on the geometric interpretation of the relationship between GLR-like curve-crossing and LR-like line-crossing events in the next subsection. As an important remark, unlike the other stopping times considered here, the GLR-like curve-crossing stopping time $N_{\GLR}$ does not depend on the choice of $\eta$. This property makes it possible to use the GLR-like curve-crossing time as a reference to choose an optimal value of $\eta$ to construct sequential tests and confidence sequences based on smaller maximum of lines-crossing and discrete mixture-crossing stopping times, as we do below in Section~\ref{subSec::well-separated}, \ref{subSec::no_separation} and \ref{sec::CS}.

\begin{proof}[Proof of \cref{thm::boundary_crossing_GLR_with_nmin}]

For a fixed $\mep \leq \mep_0$, we first consider the case $\D(\mep_1, \mep_0) \geq g(\nminbar)/\nminbar$. In this case, from \cref{lem::relationships}, we have 
\begin{align*}
    &\mathbb{P}_{\mep} \left(\exists\n \geq \nmin:   \log\GLR_\n(\mep_1,\mep_0) \geq g(\n)\right)\\ 
    &= \mathbb{P}_{\mep} \left(N_\GLR(g; \mep_1, \mep_0) < \infty \right)\\
    &\leq \mathbb{P}_{\mep} \left(N_\LR\left(e^{-g(\nminbar)};\mep_1, \mep_0\right) < \infty \right) \\
    & = \mathbb{P}_{\mep} \left(\exists\n \geq 1:  \log\LR_\n(\mep_1,\mep_0) \geq g(\nminbar)\right).
\end{align*}
Recall that the LR-like statistic based on first $\n$ samples which is defined by
\begin{equation} 
       \LR_{\n}(\mep_1, \mep_0) := \exp\left\{\n\left[ \lambda_1\bar{X}_{\n} - \psi_{\mep_0}(\lambda_1)\right]\right\},
\end{equation}
where $\lambda_1 = \nabla \psi_{\mep_0}^*(\mep_1)$. Since $\psi_{\mep_0}^*(\mep_0) = \nabla \psi_{\mep_0}^*(\mep_0) = 0 $ and $\mep_0 < \mep_1$, we have $\lambda_1 > 0$. Now, from the order-preserving property, it can be checked that $\psi_{\mep} (\lambda_1) \leq \psi_{\mep_0} (\lambda_1)$ which implies $\LR_\n (\mep_1, \mep_0) \leq \LR_\n (\hat{\mep}, \mep)$ where $\hat{\mep}_1 := \nabla \psi_{\mep}(\lambda_1)$. Therefore, we have
\begin{align*}
    &\mathbb{P}_{\mep} \left(\exists\n \geq 1:  \log\LR_\n(\mep_1,\mep_0) \geq g(\nminbar)\right)\\
    &\leq \mathbb{P}_{\mep} \left(\exists\n \geq 1:  \log\LR_\n(\hat{\mep}_1,\mep) \geq g(\nminbar)\right) \\
    & \leq e^{-g(\nminbar)},
\end{align*}
where the last inequality comes from the Ville's inequality in \eqref{eq::line-crossing_bound}. This proves \cref{thm::boundary_crossing_GLR_with_nmin} for the case of $\D(\mep_1, \mep_0) \geq g(\nminbar)/\nminbar$. 

The proof of \cref{thm::boundary_crossing_GLR_with_nmin} for the $\D(\mep_1, \mep_0) < g(\nminbar)/\nminbar$ case follows in a similar way. For any fixed $\mep \leq \mep_0$ and $\eta > 1$, from \cref{lem::relationships}, we have 
\begin{align*}
    &\mathbb{P}_{\mep} \left(\exists\n \geq \nmin:   \log\GLR_\n(\mep_1,\mep_0) \geq g(\n)\right)\\ 
    &= \mathbb{P}_{\mep} \left(N_\GLR(g; \mep_1, \mep_0) < \infty \right)\\
    &\leq \mathbb{P}_{\mep} \left(N_{\DM}(\eta) < \infty \right) \\
    & = \mathbb{P}_{\mep} \left(\exists\n \geq 1:  \sum_{k=1}^{K_\eta} h_k \LR_\n (z_k, \mep_0)  \geq 1\right) \\
    & \leq \mathbb{P}_{\mep} \left(\exists\n \geq 1:  \sum_{k=1}^{K_\eta} h_k \LR_\n (\hat{z}_k, \mep)  \geq 1\right),
\end{align*}
where $\hat{z}_k := \nabla \psi_{\mep}(\lambda_k)$ and $\lambda_k := \nabla \psi_{\mep_0}^*(z_k)$ for each $k \in [K_\eta]$. Therefore, from the consequence of the Ville's inequality in \eqref{eq::discrete_mixture_inequality}, we have
\begin{align*}
     \mathbb{P}_{\mep} \left(\exists\n \geq 1:  \sum_{k=1}^{K_\eta} h_k \LR_\n (\hat{z}_k, \mep)  \geq 1\right)
    \leq \sum_{k=1}^{K_\eta} h_k 
     =  \sum_{k=1}^{K_\eta}\exp\left\{-g(\nminbar\eta^k) / \eta\right\}.
\end{align*}
Since the right hand side of the above inequality does not depend on the choice of $\eta > 1$, it completes the proof of \cref{thm::boundary_crossing_GLR_with_nmin}, as desired.

\end{proof} 

\subsection{Proof of Remark~\ref{remark::thm1_for_process}} \label{appen::proof_of_thm1_remark}

As we did in the proof of \cref{thm::boundary_crossing_GLR}, in this section, we prove the generalized version of Remark~\ref{remark::thm1_for_process} in which we allow the minimum sample size $\nmin$ of the boundary crossing event to be larger than one. The original inequality in \cref{remark::thm1_for_process} can be obtained by setting $\nmin = 1$.

To prove the inequality, we first fix a probability distribution $P \in \mathcal{P}_{\mep_0}$. For any $\mep_0 \leq  z \in \Mep$, set $\lambda(z):= \nabla\psi_{\mep_0}^*(z)$ which is equal to $(\nabla\psi_{\mep_0})^{-1}(z)$. Under the setting in \cref{remark::thm1_for_process}, from the concavity of $\mep \mapsto \psi_\mep (\lambda)$, we have $\sum_{i=1}^\n \psi_{\mep_i} (\lambda) \leq  \n \psi_{\bar{\mep}_\n} (\lambda)$ which implies
\begin{equation}
    \LR_n(\hat{\mep}_{z,\n}, \bar{\mep}_\n):=\exp\left\{\lambda S_\n - \n \psi_{\bar{\mep}_\n} (\lambda) \right\} \leq \exp\left\{\lambda S_\n - \sum_{i=1}^\n \psi_{\mep_i} (\lambda) \right\} := M_\n(\lambda, \{\mep^i\}_{i=1}^\n),
\end{equation}
where $\hat{\mep}_{z,\n}:= \nabla \psi_{\bar{\mep}_\n}(\lambda(z))$. Note that the LR-like statistic in LHS is not necessarily a supermartingale process but it is upper bounded by $(M_\n)_{\n \in \mathbb{N}}$ which is a nonnegative supermartingale process under $P$.  

From the order-preserving property of Bregman divergences, it can be checked that, given the condition $\bar{\mep}_\n \leq \mep_0$, we have $ \LR_n(\hat{\mep}_{z,\n}, \bar{\mep}_\n) \geq  \LR_n(z, \mep_0)$ for any $z \geq \mep_0$. Therefore, if $\D(\mep_1, \mep_0) \geq g(\nminbar)/\nminbar$, as we shown in the proof of \cref{thm::boundary_crossing_GLR}, we have
\begin{align*}
   P \left(\exists\n \geq \nmin:   \log\GLR_\n(\mep_1,\mep_0) \geq g(\n)\right) &\leq  P \left(\exists\n \geq 1: \LR_\n(\mep_1, \mep_0) \geq e^{g(\nminbar)}  \right)\\
    &\leq  P \left(\exists\n \geq 1:\LR_n(\hat{\mep}_{\mep_1,\n}, \bar{\mep}_\n) \geq e^{g(\nminbar)}  \right)\\
     &\leq P\left(\exists \n \geq 1: M_\n\left(\lambda(\mep_1), \{\mep^i\}_{i=1}^\n\right) \geq e^{g(\nminbar)}\right) \\
    & \leq e^{-g(\nminbar)},
\end{align*}
which proves the claimed inequality for the $\D(\mep_1, \mep_0) \geq g(\nminbar)/\nminbar$ case. For the $\D(\mep_1, \mep_0) < g(\nminbar)/\nminbar$ case, by the similar argument, we have 
\begin{align*}
          P \left(\exists\n \geq \nmin:   \log\GLR_\n(\mep_1,\mep_0) \geq g(\n)\right) &\leq P\left(\exists \n \geq 1: \sum_{k=1}^{K_\eta}h_k \LR_\n(z_k, \mep_0) \geq 1\right) \\
        &\leq P\left(\exists \n \geq 1: \sum_{k=1}^{K_\eta}h_k \LR_\n(\hat{\mep}_{z_k,\n}, \bar{\mep}_\n) \geq 1\right) \\
        &\leq P\left(\exists \n \geq 1: \sum_{k=1}^{K_\eta}h_k M_\n\left(\lambda(z_k), \{\mep^i\}_{i=1}^\n\right) \geq 1\right) \\
        & \leq \sum_{k=1}^{K_\eta}\exp\left\{-g(\nminbar\eta^k)/\eta\right\},
\end{align*}
for any fixed $\eta > 1$. Since the right hand side of the last inequality does not depend on the choice $\eta>1$, it proves the claimed inequality, as desired.

\subsection{Proof of Lemma~\ref{lem::relationships}} \label{appen::proof_of_relation_lemma}

Recall that $\nminbar$ is defined by
\begin{equation} 
    \nminbar := \inf\left\{\n \geq \nmin:  \sup_{z\in \Mep, z > \mep_0 }\D(z ,\mep_0) \geq g(\n)/ \n \right\}.
\end{equation}
Therefore, if $\nmin < \nminbar$, we have 
 \begin{equation} \label{eq::nmin_small}
 \sup_{z\in \Mep, z > \mep_0 }\D(z ,\mep_0) < \frac{g(\nmin)}{\nmin}.
 \end{equation}
From the order-preserving property of the Bregman divergence with the fact $\bar{X}_\n \subset \overline{\Mep}$, we have
 \begin{equation}
 \frac{1}{\n}\log\GLR_\n(\mep_1,\mep_0) \leq \sup_{z\in \Mep, z > \mep_0 }\D(z ,\mep_0).
 \end{equation}
The above fact with the inequality~\eqref{eq::nmin_small} implies that
 \begin{equation}
     \left\{\exists \n \in [\nmin, \nminbar): \log\GLR_\n(\mep_1,\mep_0) \geq g(\n)\right\} = \emptyset.
 \end{equation}
Hence, the stopping time $N_\GLR$ can be re-written as follows:
\begin{equation}
N_\GLR =  \inf\left\{\n \geq \nminbar: \log \GLR_\n (\mep_1,\mep_0)\geq g(\n)\right\}.    
\end{equation}

To prove \cref{lem::relationships}, we first consider the case $\D(\mep_1, \mep_0) < g(\nminbar)/\nminbar$. In this case, from the definition of $K_\eta$ for any $\eta > 1$, we have $K_\eta \geq 1$. Now, under the same condition in \cref{lem::relationships}, define a curve-crossing time $N_{\C}$ as follows:
\begin{equation}
    N_C(\eta):= \inf\left\{\n \geq \nminbar: \sup_{z_1 > \mep_1}\log \left(\LR_\n(z_1,\mep_0)\vee 1\right)\geq g\left( n \wedge \nminbar\eta^{K_\eta} \right)\right\}, \label{eq::glr_at_single_null_st}
\end{equation}
for a given $\eta > 1$. Since the function $g$ is nondecreasing, from the definition of the GLR-like statistic, we can check $N_\C(\eta) \leq N_\GLR$ for each $\eta > 1$. Thus, to prove \cref{lem::relationships}, it is enough to show the following inequalities hold:
\begin{equation} \label{eq::lem_target_ineq}
        N_{\DM}(\eta) \leq N_{\ML}(\eta) \leq N_{\C}(\eta),
\end{equation}
for each $\eta > 1$. 

Now, let us first consider the no separation setting $(\mep_1 = \mep_0)$. In this case, we have
\[
\sup_{z_1 > \mep_1}\log \LR_\n(z_1,\mep_0) = \n D_{\psi^*_{\mep_0}}\left(\bar{X}_{\n},\mep_0\right) \mathbbm{1}\left(\bar{X}_{\n} \geq \mep_0\right)~~\text{and}~~K_\eta(\mep_1,\mep_0) = \infty. 
\]
Therefore, the curve-crossing event of the stopping time $N_C(\eta)$ in \eqref{eq::glr_at_single_null_st} can be simplified as
\begin{equation}
    \left\{\exists n \geq \nminbar: D_{\psi^*_{\mep_0}}\left(\bar{X}_{\n},\mep_0\right) \mathbbm{1}\left(\bar{X}_{\n} \geq \mep_0\right)\geq \frac{g(\n)}{\n}\right\}  = \left\{\exists n \geq \nminbar:  \left(\bar{X}_\n, \frac{g(\n)}{\n}\right) \in R(\mep_0) \right\},    
\end{equation}
where the set $R(\mep_0)$ is defined by
\begin{equation}
    R(\mep_0) := \left\{(z,y) \in \overline{\Mep} \times [0,\infty) : y \leq D_{\psi^*_{\mep_0}}(z,\mep_0), z \geq \mep_0\right\}.
\end{equation}

Now, fix a parameter $\eta >1$ and  set $I_k := [\nminbar\eta^{k-1}, \nminbar\eta^k)$ for all $k \geq 1$. Then, since $n \mapsto g(n)/n$ is a nonincreasing mapping, we have 
\begin{align*}
     &\left\{\exists n \geq \nminbar: D_{\psi^*_{\mep_0}}\left(\bar{X}_{\n},\mep_0\right) \mathbbm{1}\left(\bar{X}_{\n} \geq \mep_0\right)\geq \frac{g(\n)}{\n}\right\}\\
     & = \left\{\exists k \geq 1, \exists \n \in I_k: D_{\psi^*_{\mep_0}}\left(\bar{X}_{\n},\mep_0\right) \mathbbm{1}\left(\bar{X}_{\n} \geq \mep_0\right)\geq \frac{g(\n)}{\n}\right) \\
     &\subset\left\{\exists k\geq 1, \exists \n \in I_k: D_{\psi^*_{\mep_0}}\left(\bar{X}_{\n},\mep_0\right) \mathbbm{1}\left(\bar{X}_{\n} \geq \mep_0\right)\geq \frac{g(\nminbar\eta^k)}{\nminbar\eta^k}\right) \\
     &\subset\left\{\exists k\geq 1, \exists \n \in I_k:  \bar{X}_{\n} \geq z_k ,  \frac{g(\nminbar\eta^k)}{\nminbar\eta^k} \geq \frac{g(\nminbar\eta^k)}{\eta \n}\right),
\end{align*}
where $z_k > \mep_0$ is the solution of the following equation: 
\[
D_{\psi^*_{\mep_0}}\left(z_k,\mep_0\right)\mathbbm{1}\left(z_k \geq \mep_0\right) = \frac{g (\nminbar\eta^k)}{\nminbar\eta^k} := d_k.
\]
By setting $S_k:= [z_k,\infty) \times [0,d_k]$ and $h_k := \exp\left\{-g(\nminbar\eta^k) / \eta\right\} \in (0, 1]$, we have
\begin{align*}
    &\left\{\exists n \geq \nminbar: D_{\psi^*_{\mep_0}}\left(\bar{X}_{\n},\mep_0\right)  \mathbbm{1}\left(\bar{X}_{\n} \geq \mep_0\right)\geq \frac{g(\n)}{\n}\right\}\\
    & =\left\{ \exists \n \geq \nminbar: \left(\bar{X}_\n, \frac{g(\n)}{\n}\right) \in R(\mep_0) \right\} \\
    &\subset \left\{\exists k\geq 1, \exists \n \in I_k: \left(\bar{X}_{\n}, \frac{g (\nminbar\eta^k)}{\eta \n} \right)\in S_k \right\} \\
    &\subset \left\{\exists k\geq 1, \exists \n \in I_k: \left(\bar{X}_{\n}, \frac{\log(1/h_k)}{\n} \right)\in H(z_k, \mep_0) \right\},
\end{align*}
where $H(z_k,\mep_0)$ is the half space contained in and tangent to $ R(\mep_0)$ at $(z_k,d_k)$. See \cref{fig::GLR-LR} for illustration of the $S_k, H(z_k, \mep_0)$ and $R(\mep_0)$.

\begin{figure}
    \begin{center}
    \includegraphics[scale =  0.75]{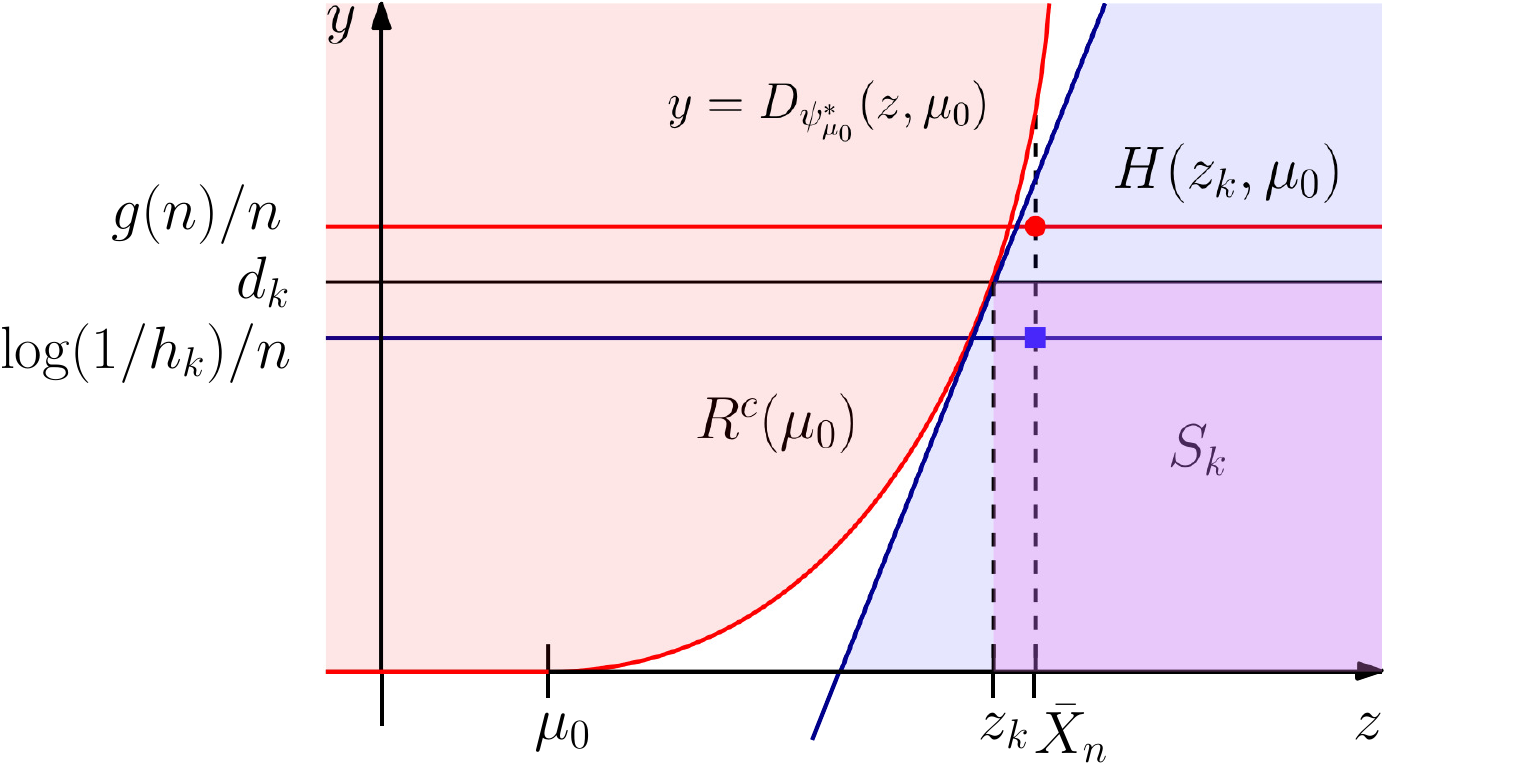}
    \end{center}
    \caption{Illustration of the relationship between the GLR-like curve-crossing and line-crossing times on the time interval $I_k$ when $k < K_\eta$. We can check that, on $I_k$, the GLR-like curve-crossing event $(\bar{X}_\n, g(n) / n) \in R(\mep_0)$ implies the line-crossing event $(\bar{X}_\n, \log(1/h_k) / n) \in H(z_k, \mep_0)$.} 
    \label{fig::GLR-LR}
\end{figure}

As we shown in the proof of \cref{thm::boundary_crossing_GLR}, the half space and LR-like statistic $\LR_\n(z_k, \mep_0)$ has a close relationship as follow
\[
\LR_\n(\mep_1, \mep_0) \geq \frac{1}{\alpha} \Longleftrightarrow \left(\bar{X}_{\n}, \frac{\log(1/\alpha)}{\n} \right)\in H(\mep_1, \mep_0),
\]
for any $\mep_0 \leq \mep_1$ and $\alpha \in (0,1]$. Therefore, we have
\begin{align*}
    &\left\{\exists \n \geq \nminbar: D_{\psi^*_{\mep_0}}\left(\bar{X}_{\n},\mep_0\right) \mathbbm{1}\left(\bar{X}_{\n} \geq \mep_0\right)\geq \frac{g(\n)}{\n}\right\}\\
      & =\left\{ \exists \n \geq \nminbar: \left(\bar{X}_\n, \frac{g(\n)}{\n}\right) \in R(\mep_0) \right\} \\
     &\subset \left\{\exists k\geq 1, \exists \n \in I_k: \left(\bar{X}_{\n}, \frac{\log(1/h_k)}{\n} \right)\in H(z_k, \mep_0) \right\}\\
     & =\left\{\exists k\geq 1, \exists \n \in I_k: \LR_\n (z_k, \mep_0) \geq \frac{1}{h_k} \right\}\\
     & \subset \left\{\exists \n \geq  1: \max_k h_k \LR_\n (z_k, \mep_0)  \geq 1 \right\} \\
     &\subset\left\{\exists \n \geq 1: \sum_{k=1}^\infty h_k \LR_\n (z_k, \mep_0)  \geq 1 \right\}.
\end{align*}
Thus, as desired, the following inequalities hold for all $\eta > 1$:
\begin{equation}
    N_{\DM}(\eta) \leq N_{\ML}(\eta) \leq N_{\C}(\eta).
\end{equation}

Now, let us consider the well-separated alternative case $(\mep_0 < \mep_1)$. As we discussed in \cref{subSec::well-separated}, we can express $\sup_{z_1 > \mep_1}\log \left(\LR_\n(z_1,\mep_0)\vee 1\right)$ as $n f(\bar{X}_\n; \mep_1,\mep_0)$ where the function $f$ is defined by
\begin{equation} \label{eq::GLR-like_fn_in_proof}
f(z ; \mep_1, \mep_0):=
    \begin{cases} 
    \left[\D(\mep_1,\mep_0) + \nabla_z \D(z, \mep_0)\mid_{z = \mep_1}(z -\mep_1)  \right]_+  & \mbox{if } z\leq \mep_1 \\
    \D(z, \mep_0) &\mbox{if } z > \mep_1
    \end{cases}. 
\end{equation}
Then, the curve-crossing event of the stopping time $N_C(\eta)$ in \eqref{eq::glr_at_single_null_st} can be written as
\begin{equation}
    \left\{\exists \n\geq \nminbar: f\left(\bar{X}_{\n};\mep_1,\mep_0\right) \geq \frac{g(\n)}{\n}\right\}  = \left\{\exists \n \geq \nminbar: \left(\bar{X}_\n, \frac{g(\n)}{\n}\right) \in R(\mep_1,\mep_0) \right\},
    \end{equation}
where the set $R(\mep_1,\mep_0)$ is defined by
\begin{equation}
    R(\mep_1,\mep_0) := \left\{(z,y) \in \overline{\Mep} \times [0,\infty) : y \leq f(z; \mep_1,\mep_0), z \geq \mep_0\right\},
\end{equation}
which is equal to $R(\mep_0) \cap H(\mep_1, \mep_0)$.

Similar to the no-separation case above, for each $k \in \mathbb{N}$, let $z_k' > \mep_0$ be the solution of the following equation: 
\[
f(z ; \mep_1, \mep_0) = \frac{g (\nminbar \eta^{k\wedge K_\eta})}{\nminbar\eta^k} := d_k'.
\]
Recall that $K_\eta$ is defined for each $\eta >1$ as
\begin{equation}
K_{\eta} := \inf\left\{k \in \mathbb{N}: \D(\mep_1,\mep_0) \geq \frac{g(\nminbar\eta^k)}{\nminbar\eta^{k}}\right\}.
\end{equation}
Therefore, we have $d_k = d_k' \geq \mep_1$ for each $k = 1, \dots, K_\eta -1$. This observation with the same argument for the well-separation case, we have for each $k = 1,\dots, K_\eta - 1$,
\begin{align*}
 &\left\{\exists \n \in I_k: \sup_{z_1 > \mep_1}\log \left(\LR_\n(z_1,\mep_0)\vee 1\right)\geq g\left(\n \wedge \nminbar\eta^{K_\eta}\right)\right\} \\
&=\left\{ \exists \n \in I_k : \left(\bar{X}_\n, \frac{g(\n)}{\n}\right) \in R(\mep_1,\mep_0) \right\} \\
    &\subset \left\{ \exists \n \in I_k: \left(\bar{X}_{\n}, \frac{g (\nminbar\eta^k)}{\eta \n} \right)\in S_k \right\} \\
    &\subset \left\{ \exists \n \in I_k: \left(\bar{X}_{\n}, \frac{\log(1/h_k)}{\n} \right)\in H(z_k, \mep_0) \right\}.
\end{align*}
For $k = K_{\eta}, K_{\eta}+1, \dots$, we can check that
\begin{align*}
      &\left\{\exists \n  \in I_k: \sup_{z_1 > \mep_1}\log \left(\LR_\n(z_1,\mep_0)\vee 1\right)\geq g\left(\n \wedge \nminbar\eta^{K_\eta}\right)\right\} \\
     &\subset \left\{\exists \n \in I_k: f(\bar{X}_\n; \mep_1,\mep_0)\geq \frac{g(\nminbar\eta^{K_{\eta}})}{\nminbar\eta^k}\right\} \\
     &\subset\left\{ \exists \n \in I_k:  \bar{X}_{\n} \geq z_k' ,  \frac{g(\eta^{K_\eta})}{\eta^k} \geq \frac{g(\nminbar\eta^{K_\eta})}{\eta \n}\right\}.
\end{align*}
From the definition of $K_\eta$ and the condition $k \geq K_\eta$, we have
$z_k' \in (\mep_0, \mep_1]$ which implies that $(z_k', d_k')$ lies on the boundary of   $H(\mep_1,\mep_0)$. Therefore, we have $S_k':= [z_k',\infty) \times [0,d_k'] \subset H(\mep_1,\mep_0)$, which implies
\begin{align*}
 &\left\{\exists \n \in I_k:  \sup_{z_1 > \mep_1}\log \left(\LR_\n(z_1,\mep_0)\vee 1\right) \geq g\left(\n \wedge \nminbar\eta^{K_\eta}\right)\right\} \\
&=\left\{ \exists \n \in I_k : \left(\bar{X}_\n, \frac{g\left(\n \wedge \nminbar\eta^{K_\eta}\right)}{\n}\right) \in R(\mep_1,\mep_0) \right\} \\
    &\subset \left\{ \exists \n \in I_k: \left(\bar{X}_{\n}, \frac{g (\nminbar\eta^{K_\eta})}{\eta \n} \right)\in S_k' \right\} \\
    &\subset \left\{ \exists \n \in I_k: \left(\bar{X}_{\n}, \frac{\log(1/h_{K_\eta})}{ \n} \right)\in H(\mep_1, \mep_0) \right\}.
\end{align*}
See \cref{fig::GLR-LR2} for illustration of the $S_k', H(\mep_1, \mep_0)$ and $R(\mep_1, \mep_0)$.
\begin{figure}
    \begin{center}
    \includegraphics[scale =  0.75]{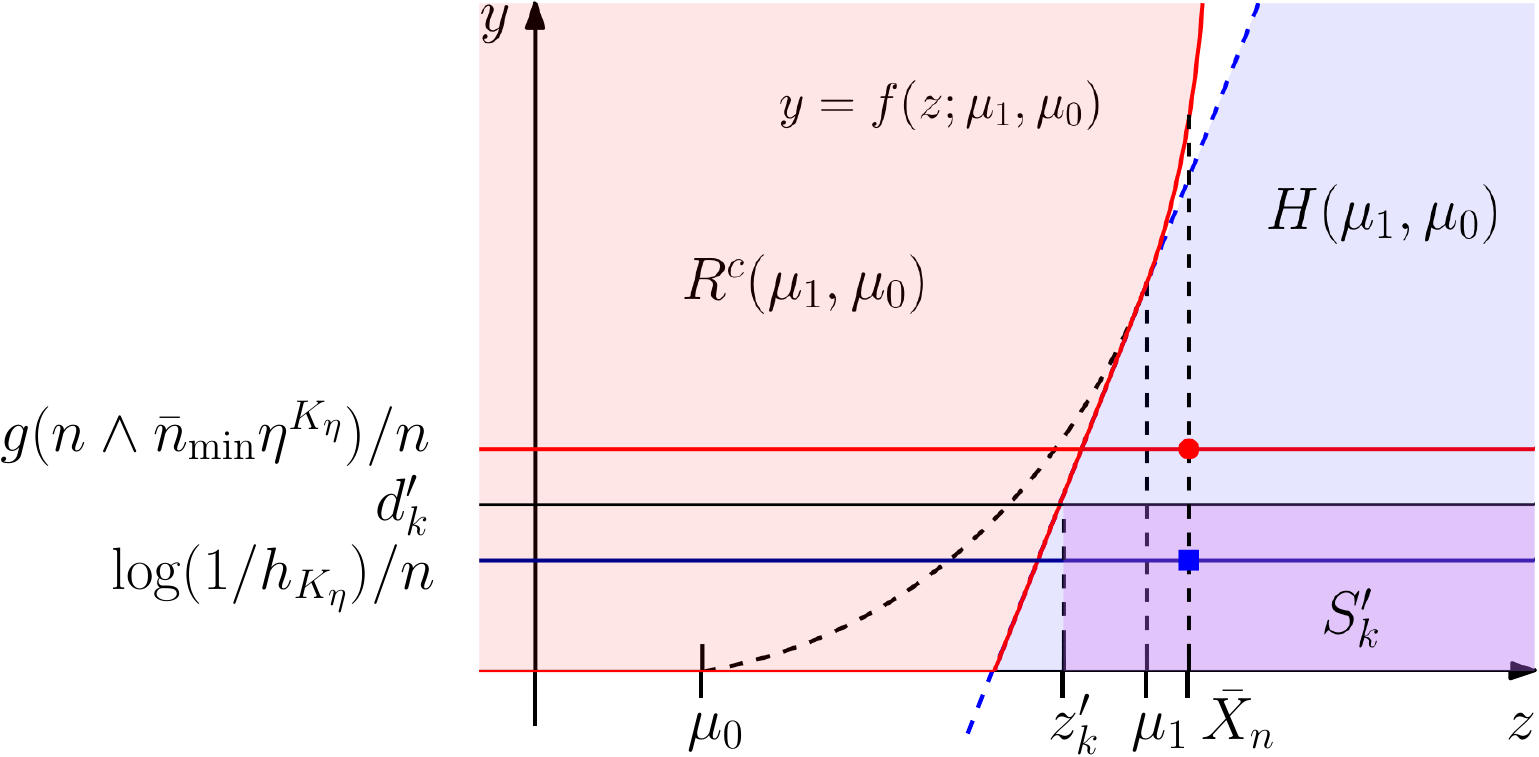}
    \end{center}
    \caption{Illustration of the relationship between the GLR-like curve-crossing and line-crossing times on the time interval $I_k$ when $k \geq K_\eta$. We can check that the GLR-like curve-crossing event $(\bar{X}_\n, g(n\wedge\nminbar\eta^{K_\eta} ) / n) \in R(\mep_1,\mep_0)$ implies the line-crossing event $(\bar{X}_\n, \log(1/h_{K_\eta}) / n) \in H(\mep_1, \mep_0)$ on each $I_k$ for all $k \geq K_\eta$.} 
    \label{fig::GLR-LR2}
\end{figure}
It is important to note that the last event in the previous equation depends on $k$ only through the condition $\exists n \in I_k$. Hence, we have that 
\begin{align*}
    &\left\{ \exists k \in \{K_\eta, K_\eta + 1 ,\dots\},\exists \n \in I_k : \left(\bar{X}_\n, \frac{g\left(\n \wedge \nminbar\eta^{K_\eta}\right)}{\n}\right) \in R(\mep_1,\mep_0) \right\} \\
    &\subset \left\{  \exists \n \in \bigcup_{j=K_\eta}^\infty I_k: \left(\bar{X}_{\n}, \frac{\log(1/h_{K_\eta})}{ \n} \right)\in H(\mep_1, \mep_0) \right\} \\
    & = \left\{  \exists \n \geq \eta^{K_\eta - 1} : \left(\bar{X}_{\n}, \frac{\log(1/h_{K_\eta})}{ \n} \right)\in H(\mep_1, \mep_0) \right\}.
\end{align*}
In sum, we have the following relationships between events:
\begin{align*}
      &\left\{\exists n \geq \nminbar: \sup_{z_1 > \mep_1}\log \left(\LR_\n(z_1,\mep_0)\vee 1\right)\geq g\left(\n \wedge \eta^{K_\eta}\right)\right\}\\
      & =\left\{ \exists \n \geq \nminbar : \left(\bar{X}_\n, \frac{g\left(\nminbar \left[\frac{\n}{\nminbar} \wedge \eta^{K_\eta}\right]\right)}{\n}\right) \in R(\mep_1,\mep_0) \right\} \\
     &\subset \left\{\exists k \in \{1,\dots, K_\eta -1\}, \exists \n \in I_k: \left(\bar{X}_{\n}, \frac{\log(1/h_k)}{\n} \right)\in H(z_k, \mep_0) \right\}\\
     &~~~~~~~~\cup \left\{ \exists \n \geq \nminbar\eta^{K_\eta-1}: \left(\bar{X}_{\n}, \frac{\log(1/h_{K_\eta})}{\n} \right)\in H(\mep_1, \mep_0) \right\}\\
     & =\left\{\exists k \in \{1,\dots, K_\eta -1\},  \exists \n \in I_k: \LR_\n (z_k, \mep_0) \geq \frac{1}{h_k} \right\}\\
     &~~~~~~~~\cup\left\{\exists \n \geq \nminbar\eta^{K_\eta -1}: \LR_\n (\mep_1, \mep_0) \geq \frac{1}{h_{K_\eta}} \right\} \\
     & \subset \left\{\exists \n \geq 1: \max_{j\in[K_\eta]} h_k \LR_\n (z_k, \mep_0) \geq 1 \right\}~~(\text{with}~~z_{K_\eta} :=\mep_1), \\
     &\subset\left\{\exists \n \geq 1 : \sum_{j=1}^{K_\eta} h_k \LR_\n (z_k, \mep_0)  \geq 1\right\},
\end{align*}
which implies $N_{\DM}(\eta) \leq N_{\ML}(\eta) \leq N_{\C}(\eta)$
for each $\eta > 1$, as desired. 

Finally, for the case of $\D(\mep_1, \mep_0) \geq g(\nminbar)/\nminbar$, by the definition of $K_\eta$ in \eqref{eq::K_number_with_nmin}, we have $K_\eta = 0$. From the fact $\frac{g(n)}{n} \leq \frac{g(\nminbar)}{\nminbar}$ for all $n \geq \nminbar$ and the following relationship between $R(\mep_1,\mep_0)$ and $H(\mep_1, \mep_0)$:
\[
R(\mep_1,\mep_0) \cap \overline{\Mep} \times \left[0, \D(\mep_1, \mep_0)\right] = H(\mep_1, \mep_0) \cap \overline{\Mep} \times \left[0, \D(\mep_1, \mep_0)\right],
\]
we know that $\left(\bar{X}_\n, \frac{g(n)}{n}\right) \in R(\mep_1,\mep_0)$ is equivalent to $\left(\bar{X}_\n, \frac{g(n)}{n}\right) \in H(\mep_1,\mep_0)$. Therefore, we have the following relationships between events:
\begin{align*}
    &\left\{\exists \n \geq \nminbar: \log \GLR_\n (\mep_1,\mep_0)\geq g(\n) \right\} \\
    & \subset \left\{\exists \n \geq \nminbar:  \sup_{z_1 > \mep_1}\log \left(\LR_\n(z_1,\mep_0)\vee 1\right)\geq g(\n) \right\} \\
    & = \left\{\exists \n \geq \nminbar: \left(\bar{X}_\n, \frac{g(n)}{\n}\right) \in R(\mep_1,\mep_0)\right\}\\
    & = \left\{\exists \n \geq \nminbar: \left(\bar{X}_\n, \frac{g(n)}{\n}\right) \in H(\mep_1, \mep_0)\right\}\\
    & = \left\{\exists \n \geq \nminbar: \LR_\n (\mep_1, \mep_0) \geq e^{g(n)}\right\}\\
    & \subset \left\{\exists \n \geq 1: \LR_\n (\mep_1, \mep_0) \geq e^{g(\nminbar)}\right\}~~\text{(since $g$ is non-decreasing)},
\end{align*}
which implies  $N_\LR (e^{-g(\nminbar)}; \mep_1, \mep_0) \leq N_\GLR$ as desired.

\subsection{Proof of Theorem~\ref{thm::upper_bound_const}}
\label{append::proof_of_upper_bound_const}
In this proof, we use the term $g_\alpha$ instead of $g_\alpha(\mep_1,\mep_0)$ for simplicity.  The statement of the type-1 error control is a direct consequence of \cref{thm::boundary_crossing_GLR} with the definition of $g_\alpha$. 

The proof of the upper bound on the expected sample size mainly follows the proof of  \cite[Theorem 1]{lorden1973open} in which a  similar bound  for exponential family distributions was presented. Here, we generalize the previous bound in \cite{lorden1973open} to the sub-$\psi_\Mep$ family of distributions. In this case, the stopping time of the SGLR-like test is given by
\begin{equation}
    N_{\GLR}(g_\alpha) = \inf\left\{\n \geq 1 : \sup_{z > \mep_1} \log \LR_\n(z, \mep_0) \geq g_\alpha \right\}
\end{equation}
For any $g > 0$, let $N_\LR(g) := \inf\left\{\n \geq 1: \log \LR_\n (\mep, \mep_0) \geq g \right\}$. Since $N_{\GLR}(g_\alpha) \leq N_\LR(g_\alpha)$, to prove \cref{thm::upper_bound_const}, it is enough to show the following inequality holds:
\begin{equation}
   \mathbb{E}_{\mep}\left[N_\LR(g_\alpha)\right] \leq \frac{g_\alpha}{\D(\mep, \mep_0)} + \left[ \frac{\sigma_{\mep}\nabla\psi_{\mep_0}^*(\mep)}{\D(\mep, \mep_0)}\right]^2 + 1.
\end{equation}
Recall that the log LR-like statistic based on first $n$ samples can be written as
\begin{equation}
    \log\LR_\n(\mep,\mep_0) = \sum_{i=1}^\n \left[\D(\mep,\mep_0) + \lambda \left(X_i -\mep\right)\right\},
\end{equation}
where $\lambda := \nabla \psi_{\mep_0}^*(\mep) = \nabla \D(z, \mep_0)\mid_{z = \mep}$. From the renewal theory, it is well-known (e.g., see Section 2.6 in  \cite{durrett2019probability})  that $\mathbb{E}_\mep N_\LR(g) < \infty$ for each $g >0$. Therefore, by Wald's equation, 
\begin{equation}
    \D(\mep, \mep_0) \mathbb{E}_{\mep}\left[N_\LR(g_\alpha)\right] 
    = \mathbb{E}_{\mep} \left[\log\LR_{N_\LR(g_\alpha)}(\mep,\mep_0)\right].
\end{equation}
Now, for each $g > 0$, let $R_g := \log\LR_{N(g)}(\mep, \mep_0) -g$. Then, we have
\begin{align}
    \D(\mep, \mep_0) \mathbb{E}_{\mep}\left[N_\LR(g_\alpha)\right]
    &= \mathbb{E}_{\mep} \left[\log\LR_{N_\LR(g_\alpha)}(\mep,\mep_0)\right]\nonumber\\
    & \leq g_\alpha + \sup_{g > 0}\mathbb{E}_\mep \left[ R_g\right]\nonumber \\
    & \leq g_\alpha + \frac{\lambda^2 \sigma_\mep^2}{\D(\mep,\mep_0)} + \D(\mep,\mep_0), \label{eq::sample_complex_inner}
\end{align}
where the first inequality comes from the definition of $N_\LR(g_\alpha)$ and the second inequality is based on the Lorden's inequality \cite{lorden1970excess}. For the completeness, we state the inequality below: 
\begin{fact}[Lorden's inequality \cite{lorden1970excess}]
Suppose $X_1, X_2, \dots$ are i.i.d. samples with $\mathbb{E}X_1 = \mep > 0$ and $\mathbb{E}X_1^2 = \sigma^2 <\infty$. For each $g >0$, set $N(g) := \inf\left\{\n: S_n := \sum_{i=1}^\n X_i \geq g \right\}$ and $R_g := S_{N(g)} - g$. Then, the following inequality holds:
\begin{equation}
\sup_{g > 0}\mathbb{E} \left[ R_g\right] \leq \frac{\mathbb{E}X_1^2}{\mep} = \mep + \frac{\sigma^2}{\mep}.
\end{equation}
\end{fact}
By multiplying $1  / \D(\mep,\mep_0)$ on both sides of the inequality~\eqref{eq::sample_complex_inner}, we have the claimed upper bound on $ \mathbb{E}_{\mep}\left[N_{\LR}(g_\alpha)\right]$ as desired. 

\subsection{Proof of Theorem~\ref{thm::T_high-sub-psi}}
\label{append::proof_of_high_prob}

We first prove the following lemma.
\begin{lemma} \label{lem::upper_bound_cal}
 For fixed $t >1, \eta > 1, m \geq 1, d >0$ and $w \in (0,1)$, suppose the following holds:
 \begin{equation} \label{eq::T_ave_condition}
     d \leq \frac{\eta \log(1/w)}{t} + \frac{2 \eta \log\left(m + \log_\eta t \right)}{t}.
 \end{equation}
 Then, we have
 \begin{equation} \label{eq::T_ave_result}
     t \leq \frac{2\eta}{d} \log\left(2m+2 + 2 \log_\eta \frac{2}{\log\eta} + 2 \left[\log_\eta (1/d)\right]_+\right) + \frac{2\eta}{d}\log(1/w),
 \end{equation}
 where $[x]_+ := \max\{x, 0\}$.
\end{lemma}

\begin{proof}[Proof of \cref{lem::upper_bound_cal}]
This proof is based on arguments in \cite{jamieson_lil_2014}. From \cref{eq::T_ave_condition}, we have
\begin{align} \label{eq::simple_form1}
    t \leq \frac{2\eta}{d}\log\left(\frac{\log \eta^m t}{\sqrt{w} \log \eta}\right).
\end{align}
Since $\log x \leq \min\left\{\sqrt{x}, x\right\},~~\forall x >0$, we can further simplify the previous inequality as
\begin{align} \label{eq::bound_on_eta_t}
    d \leq \frac{2 \eta}{t} \frac{\sqrt{\eta^m t}}{\sqrt {w} \log \eta}
    \Leftrightarrow \sqrt{\eta^m t} \leq \frac{2\eta^{m +1}}{d\sqrt{w}\log \eta} .
\end{align}
By plugging in \cref{eq::bound_on_eta_t} to the RHS of \cref{eq::simple_form1}, we have  
\begin{align}
    t &\leq \frac{2\eta}{d}\log\left(\frac{2\log \frac{2\eta^{m/2 +3/2}}{d\sqrt{w}\log\eta}}{\sqrt{w}\log\eta}\right) \nonumber\\
    & = \frac{2\eta}{d} \log \left(2\log_\eta \frac{2\eta^{m+1}}{\log \eta} + 2\log_\eta (1/d) + \log_\eta(1/w)\right)  + \frac{\eta}{d}\log(1/w)  \nonumber\\
    &\leq \frac{2\eta}{d} \log \left(2m+2 + 2\log_\eta \frac{2}{\log \eta} + 2\left[\log_\eta (1/d)\right]_+ + \log_\eta(1/w)\right)  + \frac{\eta}{d}\log(1/w). \label{eq::t_bound_further}
\end{align}
Note that $a+ \log x \leq a \sqrt{x},~~\forall x \geq 1$ if $a \geq 2$. Hence, \cref{eq::t_bound_further} can be further upper bounded by
\begin{align*}
  t \leq \frac{2\eta}{d} \log\left(2m+2 + 2 \log_\eta \frac{2}{\log\eta} + 2 \left[\log_\eta (1/d)\right]_+ \right) + \frac{2\eta}{d}\log(1/w),   
\end{align*}
as claimed. 
\end{proof}

Now, we are ready to prove \cref{thm::T_high-sub-psi} as follows:
\begin{proof}[Proof of \cref{thm::T_high-sub-psi}]
Recall that 
\begin{equation}
    D_{\psi^*}^{*}(\mep, \mep_0) := D_{\psi^*_{\mep_0}}(x^*, \mep_0) = D_{\psi^*_{\mep}}(x^*, \mep),
\end{equation}
where $x^* = x^*(\mep, \mep_0)$ is the unique solution of the following equation:
\begin{equation}
    D_{\psi^*_{\mep_0}}(x, \mep_0) = D_{\psi^*_{\mep}}(x, \mep). 
\end{equation}
Note that for any $\delta \in (0,1)$, $T_{\high}(\delta) \geq 1$ can be re-written as
\begin{equation}
    T_{\high}(\delta) := \inf\left\{t  \geq 1: \frac{c \log(1/\delta)}{t} + \frac{2c \log \left(1+\log_{c}t\right)}{t} \leq D_{\psi^*}^*(\mep, \mep_0) \right\}.
\end{equation}
Then, by the strict convexity of $x \mapsto D_{\psi^*_{\mep}}(x, \mep)$ and $x \mapsto D_{\psi^*_{\mep_0}}(x, \mep_0)$ with
\[
D_{\psi^*_{\mep}}(\mep, \mep) = 0 = D_{\psi^*_{\mep_0}}(\mep_0, \mep_0),
\]
we have the following lemma.
\begin{lemma} \label{lemma::exchange_mu}
	Suppose $\mep > \mep_0 \in \Mep$. Then, for any $z \in \Mep$, we have
	\[
	D_{\psi^*_{\mep_0}} (z, \mep_0)\mathbbm{1}(z > \mep_0) ~ < ~ \frac{c \log(1/\delta)}{T_{\high}(\delta)} + \frac{2c  \log\left(1 + \log_c T_{\high}(\delta) \right)}{T_{\high}(\delta)},
	\]
	if and only if
	\[
	D_{\psi^*_{\mep}} (z, \mep)\mathbbm{1}(z < \mep)  ~ > ~ \frac{c \log(1/\delta)}{T_{\high}(\delta)} + \frac{2c \log\left(1 + \log_c T_{\high}(\delta) \right)}{T_{\high}(\delta)}.
	\]
\end{lemma}
Based on \cref{lemma::exchange_mu} with the definition of $\st$, we have
\begin{align*}
    &\mathbb{P}_{\mep}\left(N  > T_{\high}(\delta)\right) \\
    & = \mathbb{P}_{\mep}\left(	D_{\psi^*_{\mep_0}} (\bar{X}_{T_{\high}(\delta)}, \mep_0) <\frac{c \log(1/\alpha)}{T_{\high}(\delta)} + \frac{2c  \log\left(1 + \log_c T_{\high}(\delta) \right)}{T_{\high}(\delta)}~~\text{or}~~\bar{X}_{T_{\high}(\delta)} \leq \mep_0\right) \\
    & = \mathbb{P}_{\mep}\left(	D_{\psi^*_{\mep_0}} (\bar{X}_{T_{\high}(\delta)}, \mep_0) \mathbbm{1}\left(\bar{X}_{T_{\high}(\delta)} > \mep_0\right) <\frac{c \log(1/\alpha)}{T_{\high}(\delta)} + \frac{2c  \log\left(1 + \log_c T_{\high}(\delta) \right)}{T_{\high}(\delta)}\right) \\
    & \overset{(i)}{\leq} \mathbb{P}_{\mep}\left(	D_{\psi^*_{\mep_0}} (\bar{X}_{T_{\high}(\delta)}, \mep_0) \mathbbm{1}\left(\bar{X}_{T_{\high}(\delta)} > \mep_0\right) <\frac{c \log(1/\delta)}{T_{\high}(\delta)} + \frac{2c  \log\left(1 + \log_c T_{\high}(\delta) \right)}{T_{\high}(\delta)}\right) \\
    & \overset{(ii)}{=} \mathbb{P}_{\mep}\left(	D_{\psi^*_{\mep}} (\bar{X}_{T_{\high}(\delta)}, \mep) \mathbbm{1}\left(\bar{X}_{T_{\high}(\delta)} < \mep\right) >\frac{c \log(1/\delta)}{T_{\high}(\delta)} + \frac{2c  \log\left(1 + \log_c T_{\high}(\delta) \right)}{T_{\high}(\delta)}\right) \\
    & \leq \mathbb{P}_{\mep}\left(\exists \n \geq 1: \n D_{\psi^*_{\mep}}\left(\bar{X}_\n, \mep \right)\mathbbm{1}\left(\bar{X}_\n < \mep\right) \geq c \left[\log(1/\delta) + 2\log(1+ \log_c n)\right]\right)\\
    & = \mathbb{P}_{\mep}\left(\exists \n \geq 1: \bar{X}_\n < \mep,~~ \n D_{\psi^*_{\mep}}\left(\bar{X}_\n, \mep\right) \geq c \left[\log(1/\delta) + 2\log(1+ \log_c \n)\right]\right) \\
    & \overset{(iii)}{\leq} \delta,
\end{align*}
where the inequality~$(i)$ comes from the condition $\delta \in (0, \alpha)$, the equality~$(ii)$ results from \cref{lemma::exchange_mu} and the inequality~$(iii)$ is a consequence of \cref{eq::GLR_test_ineq_no_sep}. This proves the first claimed inequality~\eqref{eq::T_high_bound-sub-psi}.

To derive the explicit upper bound in \cref{eq::T_high_bound-sub-psi_explicit}, first note that, by continuity, $T_{\high}(\delta)$ can be re-defined as
\begin{equation}
    T_{\high}(\delta) = 1 \vee \sup\left\{t \geq 1: D_{\psi^*}^*(\mep, \mep_0) \leq \frac{c \log(1/\delta)}{t} + \frac{2c \log \left(1+\log_{c}t\right)}{t} \right\}.
\end{equation}
Hence, from \cref{lem::upper_bound_cal} with $d = D_{\psi^*}(\mep, \mep_0), w = \delta$ and $m = 1$, we have
\begin{equation}
    T_{\high}(\delta) \leq 1\vee \frac{2c}{D_{\psi^*}^*(\mep,\mep_0)} \left[\log \left( 4 + 2\log_{c}\left(2 / \log c\right) + 2 \left[\log_{c}\left(1/D_{\psi^*}^*(\mep,\mep_0) \right)\right]_+ \right) +  \log(1/\delta)\right],
\end{equation}
which implies the claimed inequality~\eqref{eq::T_high_bound-sub-psi_explicit}. 

Finally, to prove the upper bound on the expected sample size in \eqref{eq::T_ave_bound_no_sep}, let $M$ be a random variable defined by
\begin{equation} \label{eq::def_of_M}
M := \frac{D_{\psi^*}^*(\mep,\mep_0)}{2c}[N_\GLR(g_\alpha^c,\mep_0) - 1]  +\log \left( 4 + 2\log_{c}\left(2 / \log c\right) + 2 \left[\log_{c}\left(1/D_{\psi^*}^*(\mep,\mep_0) \right)\right]_+ \right).
\end{equation}
Then, from the high probability bound in \eqref{eq::T_high_bound-sub-psi_explicit}, we have
\begin{equation}
    \mathbb{P}_\mep\left(M > \log(1/\delta)\right) \leq \delta,
\end{equation}
for any $\delta \in (0,\alpha)$ which implies
\begin{align*}
    \mathbb{E}_\mep M &= \int_{0}^\infty \mathbb{P}_\mep\left(M > \epsilon\right) \mathrm{d}\epsilon \\
    & \leq \log(1/\alpha) + \int_{\log(1/\alpha)}^\infty e^{-\epsilon} \mathrm{d}\epsilon \\
    & = \log(1/\alpha) + \alpha.
\end{align*}
From the definition of the random variable $M$ in \eqref{eq::def_of_M}, we have the claimed upper bound on the expected sample size $\mathbb{E}_\mep [N_\GLR(g_\alpha^c,\mep_0)]$ in \eqref{eq::T_ave_bound_no_sep}, as desired. 
\end{proof}

\subsection{Proof of Theorem~\ref{thm::CS_simple}}
\label{append::proof_of_CS_simple}

In the setting of \cref{thm::CS_simple}, we have $\nminbar = \nmin$ since $\n_0 \leq \nmin$. Also, by the definition of $\mep_1$, we have $\D(\mep_1, \mep_0) < g / \nminbar$. Recall that, in the proof of \cref{lem::relationships}, we proved that if $\D(\mep_1, \mep_0) < g / \nminbar$ then the inequality $ N_{\DM}(\eta) \leq N_{\C}(\eta)$ holds for each $\eta > 1$ where $N_{\DM}$ is the discrete mixture-crossing time defined in \cref{lem::relationships} and $N_{\C}(\eta)$ is given by
\begin{equation}
    N_C(\eta):= \inf\left\{\n \geq \nmin: \sup_{z_1 > \mep_1}\log \left(\LR_\n(z_1,\mep_0)\vee 1\right)\geq g\left( n \wedge \nminbar\eta^{K_\eta} \right)\right\},
\end{equation}
for each $\eta > 1$. Therefore, by applying the Ville's maximal inequality to the above inequality $ N_{\DM}(\eta) \leq N_{\C}(\eta)$ for the constant boundary case, we have the following upper bound on the curve-crossing probability:
\begin{equation} \label{eq::const_bound_nmin}
\mathbb{P}_{\mep_0} \left(\exists \n \geq \nmin:   \sup_{z_1 > \mep_1}\log \left(\LR_\n(z_1,\mep_0)\vee 1\right) \geq g\right) \\
  \leq   \inf_{\eta >1} \left\lceil \log_\eta \left(\frac{g/\nmin}{\D(\mep_1,\mep_0)}\right)\right\rceil e^{-g / \eta},
\end{equation}
for any given $\mep_1 \geq \mep_0$ and $\nmin \geq \n_0$.  

Now, to prove the inequality in \cref{thm::CS_simple}, let us first consider the case $\nmin = \n_0$. if $\nmin = \n_0$ then the boundary crossing event in \eqref{eq::const_bound_general} is a subset of the GLR-like curve-crossing event as following shows: 
\begin{align*}
  \left\{\exists \n \geq 1:   \sup_{z \in (\mep_1, \mep_2)}\log \left(\LR_\n(z,\mep_0)\vee 1\right) \geq g\right\}
    &\subset\left\{\exists n \geq \n_0 : \left(\bar{X}_\n, \frac{g}{\n}\right) \in R(\mep_0)\cap H(\mep_1, \mep_0) \right\}\\
    &=  \left\{\exists n \geq \n_0 : \left(\bar{X}_\n, \frac{g}{\n}\right) \in R(\mep_1,\mep_0)  \right\} \\
    & =  \left\{\exists n \geq \n_0 :   \sup_{z > \mep_1}\log \left(\LR_\n(z,\mep_0)\vee 1\right) \geq g\right\},
\end{align*}
where the first inclusion comes from the following fact:
 \begin{equation}
     \left\{\exists \n \in [1, \n_0): \sup_{z \in (\mep_1, \mep_2)}\log \left(\LR_\n(z,\mep_0)\vee 1\right) \geq g\right\} = \emptyset,
 \end{equation}
 which is based on the definition of $\n_0 := \inf\left\{\n \in \mathbb{N}:  \sup_{z\in \Mep, z > \mep_0 }\D(z ,\mep_0) \geq g/ \n \right\}$ with the fact $\bar{X}_\n \in \overline{\Mep}$ and
 \begin{equation}
 \frac{1}{\n}\sup_{z \in (\mep_1, \mep_2)}\log \left(\LR_\n(z,\mep_0)\vee 1\right) \leq \sup_{z\in \Mep, z > \mep_0 }\D(z ,\mep_0).    
 \end{equation}
 
Therefore, from the inequality~\eqref{eq::const_bound_nmin}, with the setting $\nmax = g / \D(\mep_1, \mep_0)$, we have the claimed upper bound in \eqref{eq::const_bound_general}. 

To prove the theorem for the $\nmin  > \n_0$ case, note that the boundary crossing probability in \eqref{eq::const_bound_general} can be upper bounded as follows:
\begin{align*}  &\mathbb{P}_{\mep_0} \left(\exists \n \geq 1:   \sup_{z \in (\mep_1, \mep_2)}\log \left(\LR_\n(z,\mep_0)\vee 1\right) \geq g\right)\\
    & \leq \mathbb{P}_{\mep_0}\left(\exists n \geq \n_0 : \left(\bar{X}_\n, \frac{g}{\n}\right) \in  H(\mep_2, \mep_0)\cap R(\mep_0)\cap H(\mep_1, \mep_0) \right)\\
    & \leq   \mathbb{P}_{\mep_0}\left(\exists n \in [\n_0,  \nmin) : \left(\bar{X}_\n, \frac{g}{\n}\right) \in H(\mep_2,\mep_0)  \right) + \mathbb{P}_{\mep_0}\left(\exists n \geq \nmin : \left(\bar{X}_\n, \frac{g}{\n}\right) \in R(\mep_1,\mep_0)  \right) \\
    &\leq   \mathbb{P}_{\mep_0}\left(\exists n \geq \n_0 : \left(\bar{X}_\n, \frac{g}{\n}\right) \in H(\mep_2,\mep_0)  \right) + \mathbb{P}_{\mep_0}\left(\exists n \geq \nmin : \left(\bar{X}_\n, \frac{g}{\n}\right) \in R(\mep_1,\mep_0)  \right) \\
    &\leq   e^{-g} + \mathbb{P}_{\mep_0}\left(\exists n \geq \nmin :  \sup_{z > \mep_1}\log \left(\LR_\n(z,\mep_0)\vee 1\right)  \geq g  \right) \\
    & \leq e^{-g} + \inf_{\eta >1} \left\lceil \log_\eta \left(\frac{\nmax}{\nmin}\right)\right\rceil e^{-g / \eta},
\end{align*}
which proves the claimed bound in \cref{thm::CS_simple}. The \cref{cor::CS_multi} is a direct consequence of \cref{thm::CS_simple} based on the union bound argument.

\subsection{Proof of Corollary~\ref{cor::dis_mixture} }
The coverage guarantee in \cref{eq::dis_mixture_CI_coverage} is directly comes from the consequence of the Ville's inequality in \eqref{eq::discrete_mixture_inequality}. To prove $\CI_\n^M \subset \CI_\n$ for each $\n \in \mathbb{N}$, it is enough to show the following inclusion holds:
\begin{equation}
\begin{aligned}
    &\left\{\exists \n \geq 1:   \sup_{z \in (\mep_1, \mep_2)}\log \left(\LR_\n(z,\mep_0)\vee 1\right) \geq g_\alpha \right\} \\
    &\subset \left\{\exists \n \geq 1:   e^{-g_\alpha}  \mathbbm{1}\left(\nmin >  \n_0\right) \LR_\n\left(z_0(\mep_0), \mep_0\right) 
+  e^{-g_\alpha / \eta_\alpha}\sum_{k=1}^{K_\alpha}\LR_\n\left(z_k(\mep_0), \mep_0\right) \geq 1\right\},
\end{aligned}
\end{equation}
for each $\mep_0 \in \Mep$. If $\nmin = n_0$ then this inclusion is a direct consequence of \cref{lem::relationships}. For the $\nmin > \n_0$ case, as shown in the proof of \cref{thm::CS_simple}, we have 
\begin{align*}
    &\left\{\exists \n \geq 1:   \sup_{z \in (\mep_1, \mep_2)}\log \left(\LR_\n(z,\mep_0)\vee 1\right) \geq g_\alpha \right\} \\
    & \subset \left\{\exists n \geq \n_0 : \left(\bar{X}_\n, \frac{g_\alpha}{\n}\right) \in H(\mep_2,\mep_0)  \right\} \cup \left\{ n \geq \nmin : \left(\bar{X}_\n, \frac{g_\alpha}{\n}\right) \in R(\mep_1,\mep_0)  \right\} \\
    & \subset \left\{\exists n \in [n_0, \nmin) : \log \LR_\n(\mep_2,\mep_0)  \geq g_\alpha   \right\} \cup \left\{ n \geq \nmin : \sup_{z > \mep_1}\log \left(\LR_\n(z,\mep_0)\vee 1\right)  \geq g_\alpha \right\} \\
    & \subset \left\{\exists n \in [n_0, \nmin) : e^{-g_\alpha} \LR_\n(z_0(\mep_0),\mep_0)  \geq 1   \right\} \cup \left\{ n \geq \nmin : e^{-g_\alpha / \eta_\alpha}\sum_{k=1}^{K_\alpha}\LR_\n\left(z_k(\mep_0), \mep_0\right) \geq 1 \right\} \\
       &\subset \left\{\exists \n \geq 1:   e^{-g_\alpha}\LR_\n\left(z_0(\mep_0), \mep_0\right) 
+  e^{-g_\alpha / \eta_\alpha}\sum_{k=1}^{K_\alpha}\LR_\n\left(z_k(\mep_0), \mep_0\right) \geq 1\right\},
\end{align*}
where the second last inclusion comes from \cref{lem::relationships} with the fact $\mep_2 = z_0(\mep_0)$ by the definition.

\section{Proofs of propositions and miscellaneous statements}
\label{appen::proof_of_props}
\subsection{Proof of Proposition~\ref{prop::ef-like-convex}}
For simplicity, we only prove the case $\Mep = \mathbb{R}$.  However, the proof can be straightforwardly extended to general open convex subset $\Mep \subset \mathbb{R}$. First note that since $\nabla B$ is an increasing function with the fact $\nabla B^* = \left(\nabla B\right)^{-1}$, we know that $\nabla B^*$ is a concave function from the convexity assumption on $\nabla B$.

For a given EF-like sub-$B$ family of distributions, the normalized log LR-like statistic 
is
\begin{equation}
\begin{aligned}
\frac{1}{\n} \log\LR_\n(\mep_1,\mep_0) 
&=  D_{B^*}(\mep_1, \mep_0) + \nabla_z D_{B^*}(z, \mep_0)\mid_{z = \mep_1} (\bar{X}_\n  -\mep_1) \\
&= d + \left[\nabla B^* (\mep_1) - \nabla B^*(\mep_0)\right] (\bar{X}_\n  -\mep_1),
\end{aligned}
\end{equation}
where the first equality comes from the fact $\D(\mep_1,\mep_0) = D_{B^*}(\mep_1, \mep_0)$ for  EF-like sub-$B$ distributions and the second equality comes from the definition of $d$ in \cref{prop::ef-like-convex} and $D_{B^*}(z, \mep_0) = B^*(z) - B^*(\mep_0) - \nabla B^*(\mep_0)(z-\mep_0)$. 

Now, for a given data, let $f(\mep_0):=  d + \left[\nabla B^* (\mep_1) - \nabla B^*(\mep_0)\right] (\bar{X}_\n  -\mep_1)$. To show the mapping $\mep_0 \mapsto \LR_\n (\mep_1,\mep_0)$ is non-increasing on $(-\infty, \bar{X}_\n]$, it is enough to prove $\nabla f(\mep_0) \leq 0$ for all $\mep_0 \leq \bar{X}_\n$. Since $\mep_1$ is a function of $\mep_0$ via the equation $d =  B^*(\mep_1) - B^*(\mep_0) - \nabla B^*(\mep_0)(\mep_1-\mep_0)$, by taking derivative with respect to $\mep_0$ on the both sides, we have
\begin{align*}
    &0 = \nabla B^*(\mep_1) \frac{\partial \mep_1}{\partial \mep_0} - B^*(\mep_0) - \nabla B^*(\mep_0)\left(\frac{\partial \mep_1}{\partial \mep_0} - 1\right) - \nabla^2B^*(\mep_0)(\mep_1 - \mep_0) \\
    &\Longleftrightarrow \nabla^2B^*(\mep_0) = \frac{\partial \mep_1}{\partial \mep_0}\left[\frac{\nabla B^* (\mep_1) - \nabla B^*(\mep_0)}{\mep_1 - \mep_0}     \right].
\end{align*}
From the above observation, we can express $\nabla f(\mep_0)$ as follows:
\begin{align}
    \nabla f(\mep_0) &= \left[\nabla^2 B^*(\mep_1)\frac{\partial \mep_1}{\partial \mep_0} -  \nabla^2 B^*(\mep_0)\right]\left(\bar{X}_\n - \mep_1\right) - \left[\nabla B^* (\mep_1) - \nabla B^*(\mep_0)\right] \frac{\partial \mep_1}{\partial \mep_0} \nonumber \\
    &=  \nabla^2 B^*(\mep_1) \left(\bar{X}_\n - \mep_1\right)\frac{\partial \mep_1}{\partial \mep_0} - \nabla^2 B^* (\mep_0) \left(\bar{X}_\n - \mep_0\right) \label{eq::nabla_f_1} \\
    &= \left[\nabla^2 B^*(\mep_1) \frac{\partial \mep_1}{\partial \mep_0} - \nabla^2 B^2(\mep_0)\right]\left(\bar{X}_\n - \mep_1\right) - \nabla^2 B^*(\mep_0)\left(\mep_1 - \mep_0\right) \nonumber \\
    &=  \left[\nabla^2 B^*(\mep_1) - \frac{\nabla B^* (\mep_1) - \nabla B^*(\mep_0)}{\mep_1 - \mep_0} \right] \left(\bar{X}_\n - \mep_1\right)\frac{\partial \mep_1}{\partial \mep_0} - \nabla^2 B^*(\mep_0)\left(\mep_1 - \mep_0\right).\label{eq::nabla_f_2}
\end{align}
 For the $\mep_0 \leq \bar{X}_\n < \mep_1 $ case, since $\nabla^2B^* \geq 0$ and $ \frac{\partial \mep_1}{\partial \mep_0} \geq 0$, we can check $\nabla f(\mep_0) \leq 0$ from the expression~\eqref{eq::nabla_f_1}. Similarly, for the $\bar{X}_\n \geq \mep_1$ case, from the expression~\eqref{eq::nabla_f_2}, we have $\nabla f(\mep_0) \leq 0$ since $\nabla^2 B^*(\mep_1) - \frac{\nabla B^* (\mep_1) - \nabla B^*(\mep_0)}{\mep_1 - \mep_0} \leq 0$ from the concavity of $\nabla B^*$. Therefore, we proved the mapping  $\mep_0 \mapsto \LR_\n (\mep_1,\mep_0)$ is non-increasing on $(-\infty, \bar{X}_\n]$.


\subsection{Proof of the inequality~(\ref{eq::const_g_upper_bound})}
For a simple notation, we denote $\D(\mep_1,\mep_0)$ by $D$ throughout this proof. For any fixed $\eta > 1$, the right hand side of \eqref{eq::const_bound} is upper bounded by $
\left[ 1+ \log_\eta \left(\frac{g}{D} \vee 1\right)\right] e^{-g / \eta} 
$. Therefore, $g_\alpha(\mep_1,\mep_0)$ is upper bounded by any constant $g > 0$ such that
\begin{equation} \label{eq::const_g_upper_inner_bound}
    \left[ 1+ \log_\eta \left(\frac{g}{D} \vee 1\right)\right] e^{-g / \eta}  \leq \alpha ~~\Longleftrightarrow~~ g \leq \eta \log\left(\frac{\log\left(\eta (\frac{g}{D} \vee 1)\right)}{\alpha \log \eta} \right).
\end{equation}
From the fact $\log x \leq \min\left\{\sqrt{x},  x \right\},~~\forall x >0$, the above inequality further implies
\begin{align*}
   &g \leq \eta \log\left(\frac{\log\left(\eta(\frac{g}{D} \vee 1
   )\right)}{\alpha \log \eta} \right) \\
   & \Rightarrow  g \leq \frac{\eta\sqrt{\eta}\sqrt{\frac{g}{D} \vee 1}}{\alpha\log\eta}\\
   & \Rightarrow  \sqrt{\frac{g}{D}} \leq \frac{\eta\sqrt{\eta}}{\alpha D\log\eta} \vee 1.
\end{align*}
By plugging-in the last inequality into \eqref{eq::const_g_upper_inner_bound}, we have
\begin{equation}
    g \leq \eta\log(1/\alpha) + \eta\log\left(1 + 2\log_\eta \left(\frac{\eta\sqrt{\eta}}{\alpha D\log\eta}\vee 1\right)\right).
\end{equation}
Since this inequality holds for any $\eta > 1$, it proves the claimed inequality \eqref{eq::const_g_upper_bound}, as desired.

\subsection{Proof of the inequality~(\ref{eq::multiple_source_bound})}

This proof argument mainly follows the proof of Proposition 33.9 in \cite{lattimore2020bandit} in which a similar inequality was proved for the case $h^a(u) = \log(1/u)$ for each $a \in [K]$. Here, we show how to prove a similar bound for general $h^a$ functions.   

For each $a \in [K]$, define a random variable $W_a$ as follows:
\begin{equation}
    W_a := \sup\left\{\alpha \in (0,1] : \log\GLR_t^a(\mep_1^a, \mep_0^a) < f^a\left(N_a(t)\right)  + h^a(\alpha),~~\forall t \geq 0  \right\}.
\end{equation}
Then, from the inequality~\eqref{eq::bound_for_each_a},  we have 
\begin{align*}
    &\mathbb{P}_0\left(W_a < x \right) \\
    & =    \mathbb{P}_0\left(\exists t \geq 0: \log\GLR_t^a(\mep_1^a, \mep_0^a) \geq f^a\left(N_a(t)\right) + h^a(x)\right) \\
    &\leq x = \mathbb{P}_0 \left(U_a < x\right),~~\forall x \in (0,1], 
\end{align*}
which implies that each $W_a$ is first-order stochastically dominant over $U_a$. Since each $h^a$ is non-increasing and $W_a$ and $U_a$ are independent random variables, $\sum_{a=1}^K h^a(U_a)$ is first-order stochastically dominant over $\sum_{a=1}^K h^a(W_a)$. Therefor, for any $\epsilon > 0$, we have 
\begin{align*}
    &\mathbb{P}_0\left(\exists t \geq 0: \sum_{a=1}^K\log\GLR_t^a(\mep_1^a, \mep_0^a)\geq \sum_{a=1}^K  f^a\left(N_a(t)\right)  + \epsilon\right)\\
    &\leq \mathbb{P}_0\left(\sum_{a=1}^K h^a(W_a) \geq \epsilon\right)\\
    & \leq \mathbb{P}_0\left(\sum_{a=1}^K h^a(U_a) \geq \epsilon\right),
\end{align*} 
 which proves the claimed inequality in \eqref{eq::multiple_source_bound}. 
 
 Note that if each $h^a(u) = \log(1/u)$ then we can also use the following explicit upper bound as Proposition 33.9 in \cite{lattimore2020bandit} shown:
 \begin{equation}
  \mathbb{P}_0\left(\sum_{a=1}^K \log(1/W_a) \geq \epsilon\right) \leq  \left(\frac{\epsilon}{K}\right)^K \exp(K-\epsilon),
 \end{equation}
 which is based on the fact $\mathbb{P}_0\left(\log(1/W_a) \geq \epsilon\right) \leq e^{-\epsilon}$ for each $a \in [K]$ and $\epsilon > 0$.

\section{A practical method to design a test for a given power constraint.} \label{appen::simulation_study}
Open-ended tests such GLR-like and discrete mixture ones are often used in the cases in which we wish to keep monitoring the data stream if it appears to follow the null (or safe zone) and want to stop the data stream if it instead the alternative hypothesis fit the data better (the dangerous zone). For example, suppose each sample in the data stream represents each user's dissatisfaction after implementing a new feature into an online service. If it looks like there are only few dissatisfied users ($\mep \sim 0$) then we can keep running the service with the new feature. Otherwise, if a significant fraction of users are not satisfied by the new feature ($\mep \gg 0$), we need to stop to deploy the new feature and consider some alternative approaches.
Because of the possibility of indefinite sampling, open-ended tests have power of 1 against alternatives. However, in practice, we often use a testing procedure which has a power less than 1 to reduce the sample size. Also, in most case, there exists a termination time for experimentation. Therefore, even if we are running an open-ended test, it has a practical benefit to consider the upper bound on the maximum sample size and to design the open-ended test accordingly.   

With this motivation in mind, in this appendix, we describe a heuristic but practical method to  design SGLR-like and discrete mixture test which has an upper bound on the sample size while maintaining large enough power to detect the alternative for exponential family distributions. For many exponential family distributions, there exist standard tests with fixed sample sizes in which we can compute the minimum fixed sample size to achieve a prespecified power $1-\beta$ while controlling the type-1 error by $\alpha$.

Concretely, let $n^*$ be the minimum sample size for a  fixed sample size test to satisfy the type I and type II error constraints with values $\alpha, \beta \in (0,1)$ respectively, for the following one-sided testing problem:
\begin{equation}
    H_0 : \mep \leq \mep_0 ~~\text{vs}~~H_1 : \mep > \mep_1,
\end{equation}
for some $\mep_0 \leq \mep_1$. For example, in the Gaussian with known variance setting, the minimum sample size of the Z-test is given by 
\begin{equation}
    n^* = \inf\left\{n \in \mathbb{N}: \mathbb{P}_{\mep_1}\left(\bar{X}_n \geq \mep_0 + z_\alpha \frac{\sigma}{\sqrt{n}}\right) \geq 1-\beta\right\}.
\end{equation}
Similarly, for the Bernoulli sequence, the minimum sample size of the exact binomial test is given by
\begin{equation}
    n^* = \inf\left\{n \in \mathbb{N}:  \mathbb{P}_{\mep_1}\left(\sum_{i=1}^n X_i \geq k_\alpha(n)\geq 1-\beta \right)\right\},
\end{equation}
where $k_\alpha(n) \in [0, n]$ is the smallest integer such that  $\mathbb{P}_{\mep_0}\left(\sum_{i=1}^n X_i \geq k_\alpha(n)\right) \leq \alpha$.

Now, by using the minimum sample size of the standard test as a reference, we can design an efficient sequential testing procedure which controls the type-1 error at the same level but can detect, on average, the alternative even faster than the standard test.
To construct SGLR-like and discrete mixture tests, instead of using the boundary of alternative $\mep_1$ as we did in \cref{sec::GLR-like test}, based on the minimum sample size $n^*$ of the standard test, we set the target interval as $\nmin := \lceil n^* / 10\rceil$ and $\nmax := 2 n^*$, and use the confidence sequence approach in \cref{sec::CS} to build sequential testing procedures. The intuition behind this design of sequential testing procedure is as follows: first, we may not be able to detect the alternative significantly faster than in the standard test which is often an optimal fixed sample sized test. Therefore, instead of using $\nmin = 1$, we set the lower bound of the target interval as the $10\%$ of the sample size $n^*$ of the standard test to increase the sample efficiency. Also, since we are aiming to achieve a power strictly less than 1, instead of running the open-ended test indefinitely, we can stop the testing procedure at the same order of the minimum sample size of the standard test. Since the near-optimal sequential test will have the average sample size at most as large as the minimum sample size of the standard test, we set the upper bound of the target interval as the double size of $n^*$ to allow some random fluctuation around the minimum sample size of the standard test. 

Although this is a heuristic approach to design SGLR-like and discrete mixture tests, we have observed that these ``upper bounded'' sequential tests outperform even some optimal fixed sample size tests such as Z-test for Gaussian random variables and exact binomial test for Bernoulli sequences empirically. To demonstrate the effectiveness of this approach, below we conduct two simulation studies.

{\bf Simulation 1. Gaussian random variables with $\sigma = 1$}

In this simulation, we test the following one-sided problem:
\begin{equation}
    H_0: \mu \leq 0~~\text{vs}~~H_1: \mu > 0.1.
\end{equation}
The Z-test is to reject the null if $\bar{X}_n \geq z_\alpha / \sqrt{n}$ for a fixed $n$. We set $\alpha  = \beta = 0.1$. In this case, the minimum sample size of the Z-test is given by $n^* := 657$, and we used this sample size to conduct the Z-test. Based on the design method above, we set the target interval for SGLR-like and discrete mixture test as $[\lceil n^* / 10\rceil, 2 n^*] = [66, 1314]$. Each sequential testing procedure runs up to $n = 2 n^*$. We run Z-test and both sequential tests $2000$  times for Gaussian data streams where the underlying mean is equal to one of $(-0.1, -0.05, \dots, 0.2)$. 

\begin{table}[ht]
    \caption{Estimated probabilities of rejecting the null hypothesis. (Gaussian)}
    \label{tab:G_reject}
  \centering
  \small
\begin{tabular}{rrrrrrr}
  \hline
  True hypothesis & True $\mu$ & Repeated Z-test (p-hacking) & SGLR & SGLR (Discrete mix.) & Z-test \\ 
  \hline
  \hline
\multirow{2}{*}{$H_0$ is true} &-0.05 & 0.42 & 0.00 & 0.02 & 0.00 \\ 
 &0.00 & 0.66 & 0.00 & 0.10 & 0.10 \\ 
 \hline
Neither $H_0$ nor $H_1$ &0.05 & 0.94 & 0.11 & 0.40 & 0.52 \\ 
 \hline
\multirow{3}{*}{$H_1$ is true} &0.10 & 1.00 & 0.66 & 0.89 & 0.90 \\ 
 &0.15 & 1.00 & 0.99 & 1.00 & 1.00 \\ 
 &0.20 & 1.00 & 1.00 & 1.00 & 1.00 \\ 
   \hline
\end{tabular}
\end{table}

\begin{table}[ht]
    \caption{Estimated average sample sizes of testing procedures. (Gaussian)}
    \label{tab:G_sample_size}
  \centering
  \small
\begin{tabular}{rrrrrrr}
  \hline
True hypothesis & True $\mu$ & SGLR & SGLR (Discrete mix.) & SPRT (oracle) &Z-test \\ 
   \hline
  \hline
\multirow{2}{*}{$H_0$ is true} & -0.05 & 1312.80 & 1284.72& & 657 \\ 
 &0.00 & 1310.93 & 1201.76 & & 657 \\ 
 \hline
 Neither $H_0$ nor $H_1$ & 0.05 &  1251.36 & 958.40 & & 657 \\
 \hline
 \multirow{3}{*}{$H_1$ is true} & 0.10 &  926.14 & 509.28 &449.33& 657 \\ 
& 0.15 &  488.84 & 222.06 &209.88& 657\\ 
& 0.20 &  280.21 & 127.01 &120.11& 657 \\ 
   \hline
\end{tabular}
\end{table}

\begin{table}[h!]
    \caption{Estimated probabilities of tests being stopped earlier than Z-test.}
    \label{tab:G_early_stop}
  \centering
  \small
\begin{tabular}{rrrrrrr}
  \hline
True hypothesis & True $\mu$ & SGLR & SGLR (Discrete mix.) & SPRT (oracle) \\ 
  \hline
     \hline
\multirow{2}{*}{$H_0$ is true} &  -0.05 & 0.00 & 0.02&  \\ 
  & 0.00 & 0.00 & 0.09& \\ 
   \hline
   Neither $H_0$ nor $H_1$ & 0.05 & 0.04 & 0.30& \\ 
   \hline
 \multirow{3}{*}{$H_1$ is true} &    0.10 & 0.28 & 0.68 &0.78\\ 
  & 0.15 & 0.75 & 0.95 & 0.97\\ 
  & 0.20 & 0.97 & 1.00 & 1.00\\ 
   \hline
\end{tabular}
\end{table}

Table~\ref{tab:G_reject}-\ref{tab:G_early_stop} summarizes the simulation result. In Table~\ref{tab:G_reject}, we show that, for each underlying true mean, how frequently each testing procedure rejects the null hypothesis. Here the column Repeated Z-test (p-hacking) represents the naive usage of Z-test as a sequential procedure in which we stop and reject the null whenever p-value of the Z-test goes below the level $\alpha$. This is an example of p-hacking which inflates the type-1 error significantly larger than the target level $\alpha$. From Table~\ref{tab:G_reject}, we can check that the Z-test with continuous monitoring of p-values yields a large type-1 error ($0.43$) even under a null distribution ($\mep = -0.5)$ which is a way from the boundary of the null ($\mep_0 = 0)$.  In contrast, for all other valid level $\alpha$ testing procedures, type-1 errors are controlled under the null distributions. 

For alternative distributions ($\mep > 0.1$), the Z-test with the minimum sample size $n^*$ has larger powers than the prespecified bound $1-\beta = 0.9$ as expected. Note that the discrete mixture based SGLR test achieves almost the same power compared to the Z-test. However, as we can check from Table~\ref{tab:G_sample_size}~and~\ref{tab:G_early_stop}, under the alternative distributions, the discrete mixture test detects the signal faster than the Z-test both on average and with high probability. As a reference, we also run the same simulation for the (oracle) SPRT where each test is designed for a simple null $H_0 : \mep = 0$ versus each true but unknown alternative hypothesis. This is an oracle test since it requires knowledge of the true underlying parameter which is not available in most practical scenarios. If the underlying true distribution belongs to one of two simple hypotheses, the corresponding SPRT is known to be the optimal test~\cite{wald1948optimum} thus we can treat its average sample size under alternative as the best possible sample size we can achieve. From  Table~\ref{tab:G_sample_size}~and~\ref{tab:G_early_stop}, we can check that the performance of the discrete mixture based SGLR test is comparable to the optimal oracle test although the former does not require the information about the exact true parameter but only requires its boundary.

The SGLR test yields a weaker power at the boundary of the alternative space but it achieves higher powers and smaller sample size as the underlying true means being farther away from the boundary. 
However, SGLR and discrete mixture tests do not always perform better than the Z-test. If the underlying true mean lies between boundaries of null and alternative spaces, both test have weaker power and requires more samples to detect the signal compared to the Z-test. Therefore, we recommend to set the boundary of the alternative conservatively in practice.

{\bf Simulation 2. Bernoulli random variables}

In this simulation, we test the following one-sided problem:
\begin{equation}
    H_0: \mu \leq 0.1~~\text{vs}~~H_1: \mu > 0.12
\end{equation}
For the Bernoulli case, we use the exact binomial test as the standard test based on a fixed sample size in which we reject the null if $\sum_{i=1}^n X_i \geq k_\alpha(n)$ where $k_\alpha(n) \in [0, n]$ is the smallest integer such that  $\mathbb{P}_{\mep_0}\left(\sum_{i=1}^n X_i \geq k_\alpha(n)\right) \leq \alpha$. As we did for the Gaussian case, we set $\alpha  = \beta = 0.1$. In this case, the minimum sample size of the exact binomial test is given by $n^* := 1644$, and we used this sample size to conduct the exact binomial test. Based on the design method above, we set the target interval for SGLR-like and discrete mixture test as $[\lceil n^* / 10\rceil, 2 n^*] = [165, 3290]$. Each sequential testing procedure runs up to $n = 2 n^*$. We run the exact binomial test and both sequential tests $2000$  times for data streams from the Bernoulli distribution where the underlying mean is equal to one of $(0.9, 0.1, \dots, 0.14)$.

\begin{table}[h!]
    \caption{Estimated probabilities of rejecting the null hypothesis. (Bernoulli)}
    \label{tab:ber_reject}
  \centering
  \small
\begin{tabular}{rrrrrrr}
  \hline
 True hypothesis & True $\mu$ & Repeated binom. (p-hacking) & SGLR & SGLR (Discrete mix.) & Exact binom. \\ 
  \hline
    \hline
  \multirow{2}{*}{$H_0$ is true} &0.09 & 0.34 & 0.00 & 0.02 & 0.00 \\ 
   &0.10 & 0.59 & 0.00 & 0.08 & 0.09 \\ 
  \hline
   Neither $H_0$ nor $H_1$&0.11 & 0.93 & 0.12 & 0.41 & 0.51 \\ 
   \hline
  \multirow{3}{*}{$H_1$ is true}   &0.12 & 1.00 & 0.69 & 0.90 & 0.90 \\ 
   &0.13 & 1.00 & 0.99 & 1.00 & 1.00 \\ 
   &0.14 & 1.00 & 1.00 & 1.00 & 1.00 \\ 
   \hline
\end{tabular}
\end{table}

\begin{table}[h!]
    \caption{Estimated average sample sizes of testing procedures. (Bernoulli)}
    \label{tab:ber_sample_size}
  \centering
  \small
\begin{tabular}{rrrrrrr}
  \hline
  True hypothesis &True $\mu$ & SGLR & SGLR (Discrete mix.)& SPRT (oracle) & Exact binom. \\ 
  \hline
    \hline
  \multirow{2}{*}{$H_0$ is true} &0.09 & 3290.00 & 3234.39 && 1645 \\ 
  & 0.10 & 3278.18 & 3069.49 && 1645 \\
   \hline
     Neither $H_0$ nor $H_1$ &0.11 & 3128.98 & 2354.69 && 1645 \\ 
   \hline
  \multirow{3}{*}{$H_1$ is true}   &0.12 & 2259.14 & 1203.76 &1060.18& 1645 \\ 
   &0.13 & 1235.90 & 558.26 &523.05& 1645 \\ 
   &0.14 & 701.95 & 308.51 & 292.87& 1645 \\ 
   \hline
\end{tabular}
\end{table}

\begin{table}[h!]
    \caption{Estimated probabilities of tests being stopped earlier than the exact binomial test.}
    \label{tab:ber_early_stop}
  \centering
  \small
\begin{tabular}{rrrrrrr}
  \hline
  True hypothesis  &True $\mu$ & SGLR & SGLR (Discrete mix.)& SPRT (oracle) \\ 
  \hline
    \hline
   \multirow{2}{*}{$H_0$ is true} &0.09 & 0.00 & 0.02& \\ 
  &0.10 & 0.00 & 0.07& \\ 
  \hline
   Neither $H_0$ nor $H_1$ &0.11 & 0.04 & 0.31& \\ 
  \hline
   \multirow{3}{*}{$H_1$ is true}  &0.12 & 0.30 & 0.71 & 0.81\\ 
  &0.13 & 0.73 & 0.95 &0.96\\ 
  &0.14 & 0.97 & 1.00& 1.00 \\  
   \hline
\end{tabular}
\end{table}

Table~\ref{tab:ber_reject}-\ref{tab:ber_early_stop} summarizes the simulation result. In all three tables, we can check the same pattern we observed from the Gaussian case. In Table~\ref{tab:ber_reject}, we show that, for each underlying true mean, how frequently each testing procedure rejects the null hypothesis. Here the column Repeated exact binomial (p-hacking) represents the naive usage of the exact test as a sequential procedure in which we stop and reject the null whenever p-value of the test goes below the level $\alpha$. As we observed before, the p-hacking inflates the type-1 error significantly larger than the target level $\alpha$. From Table~\ref{tab:ber_reject}, we can check that the exact binomial test with continuous monitoring of p-values yields a large type-1 error ($0.23$) even under a null distribution ($\mep = 0.09)$ which is a way from the boundary of the null ($\mep_0 = 0.1)$.  In contrast, for all other valid level $\alpha$ testing procedures, type-1 errors are controlled under the null distributions. 

For alternative distributions ($\mep > 0.12$), the exact binomial test with the minimum sample size $n^*$ has larger powers than the prespecified bound $1-\beta = 0.9$ as expected. Again, as same as the Gaussian case, the discrete mixture based SGLR test achieves almost the same power compared to the exact binomial test. However, as we can check from Table~\ref{tab:ber_sample_size}~and~\ref{tab:ber_early_stop}, under the alternative distributions, the discrete mixture test uses, on average and with a high probability, smaller numbers of samples to detect the signal than the exact binomial test with a fixed sample size. Also, we can see that the performance of the  SGLR test based on discrete mixtures is comparable to the optimal oracle test (SPRT), although the former does not require any information about the exact true parameter but only requires its boundary. The SGLR test yields a weaker power at the boundary of the alternative space but it achieves higher powers and smaller sample size as the underlying true means being farther away from the boundary. 

However, as same as the Gaussian case, SGLR and discrete mixture tests do not always perform better than the exact binomial test. If the underlying true mean lies between boundaries of null and alternative spaces, both test have weaker power and requires more samples to detect the signal compared to the exact binomial. Therefore, it is recommended to set the boundary of the alternative close to the boundary of the null in practice.

\end{document}